\pdfoutput=1

\documentclass[12pt]{amsart}
\pdfoutput=1
\usepackage{graphicx}
\usepackage{enumerate}
\usepackage{pdfsync}

\usepackage{subfigure}
\usepackage{multicol}
\usepackage{mathtools}
\usepackage{url}
\usepackage{mathptmx}
\usepackage{type1cm}
\usepackage{helvet}
\usepackage{color}
\usepackage{esint}

\usepackage{epsfig,rotating,psfrag}
\usepackage{tikz-cd}
\usepackage{pgf, tikz}
\usetikzlibrary{arrows, automata}
\usepackage{float}
\usepackage{amssymb,amsfonts,latexsym,amsmath,amsthm,times, amscd, amsxtra, euscript, MnSymbol}

 \let\mathscr\relax
\theoremstyle{plain}
\newtheorem{theorem}{Theorem}[section]
\newtheorem{corollary}{Corollary}[section]
\newtheorem{lemma}{Lemma}[section]
\newtheorem{proposition}{Proposition}[section]
\theoremstyle{definition}

\newtheorem{definition}{Definition}[section]

\newcommand{\N}{\mathbb{N}}
\newcommand{\Z}{\mathbb{Z}}

\newcommand{\E}{\mathcal E}

\newcommand{\G}{\mathbf G}
\newcommand{\F}{\mathbf F}
\newcommand{\M}{\mathbf M}
\DeclareMathOperator{\id}{id}

\DeclareMathOperator{\Cov}{Cov}
\DeclareMathOperator{\Eqv} {Eqv}

\title[Generalized metric spaces]{Generalized metric spaces. Relations with graphs, ordered sets and automata : A survey.}
\author{Mustapha KABIL}
\address{Laboratory of Mathematics, Computer Science and Applications,\\ Department of Mathematics, Faculty of Sciences and Technologies Mohammedia, Hassan II University, Casablanca, Morocco,\\  E-mail : kabilfstm@gmail.com}
 \author{Maurice POUZET}
 \address{ ICJ, Math\'ematiques, Universit\'e Claude-Bernard Lyon1,\\ France and Mathematics \& Statistics Department, University of Calgary, Calgary, Canada,\\   E-mail : pouzet@univ-lyon1.fr}
\begin{document}

\begin{abstract}

 In this survey we present a generalization of the notion of metric space  and some applications to discrete structures as graphs, ordered sets and transition systems.  Results in that direction started in the middle eighties based on the  impulse given by Quilliot (1983). Graphs and ordered sets were considered as  kind of  metric spaces, where - instead of real numbers - the values of the distance functions $d$ belong  to an ordered semigroup  equipped with an involution. In this frame, maps preserving graphs  or posets  are exactly the nonexpansive mappings  (that is the maps $f$ such that $d(f(x),f(y))\leq d(x,y)$, for all $x,y$). It was observed that many known results on retractions  and fixed point property for classical metric spaces (whose morphisms are the nonexpansive mappings) are also valid for these spaces. For  example, the characterization of absolute retracts, by  Aronszajn and Panitchpakdi (1956), the construction of the injective envelope  by  Isbell (1965) and the fixed point theorem of Sine and Soardi (1979) translate into the  Banaschewski-Bruns theorem (1967), the MacNeille completion of a poset (1933) and the famous Tarski fixed point theorem (1955). This prompted an analysis of several classes of discrete structures from a metric point of view. In this paper, we report the results obtained over the years with a particular emphasis on the fixed point property.
\end{abstract}

\maketitle
\section{Introduction}
This survey delves into a generalization of metric spaces and its applications to discrete structures as graphs, ordered sets and transition systems. 
The results presented here originate in a paper by the second author \cite{pouzet}, motivated by the work of Quilliot \cite{Qu1, Qu2}. The genesis of this topic is to be found in two theses  \cite{jawhari},  \cite{misane} and a paper  \cite{JaMiPo}. The theme was subsequently developped in \cite{Ka}, \cite{PR}, \cite{KP1}, \cite{Sa}, \cite{kabil}, \cite{KP2},  \cite{KPR}, \cite{bandelt-pouzet}, \cite{bandelt-pouzet-saidane} and \cite{khamsi-pouzet}.

Since its introduction by Fr\'echet  (1906), the notion of metric space has motivated many extensions (cf.  the encyclopedia \cite{deza-deza},  also \cite{lawvere, blumenthal, blumenthal1, blumenthal-menger}, and recently \cite{conant}). In the sequel,  a \emph{generalized metric space} (see \cite{deza-deza} p. 82) is a set $E$ equipped with a \emph{distance}, that is, with a  map $d$ from the direct product $E\times E$ into an ordered monoid,  say $\mathcal H$,  equipped with an involution $-$ preserving the order and reversing the monoid operation.This operation will be  denoted by $\oplus$  (despite that  it is  not necessarily commutative) and its neutral element will be denoted by 0.

The conditions on $d$ are the following
\begin{enumerate}
\item [(i)]
\begin{equation}\label{defdistance}
 d(x,y)\;\leq \;0\;   \text{if and only if}\;  x=y; d(x,y)\leq d(x,z)\oplus d(z,y); \end{equation}\\
 \item [(ii)]
\begin{equation} d(x,y)=\overline{d(y,x)}\;  \text{for all}\;  x,y, z\in E.
\end{equation}
\end{enumerate}

The  focus in this survey will be on the special case in which the following assumptions are imposed on $\mathcal H$.

\begin{enumerate}
\item $0$ is the least element  of $\mathcal H$;  in which case,  condition $(i)$ for $d$ reduces to
$d(x,y)=0$ if and only if $x\;=\;y$.

\item $\mathcal H$ is a complete lattice and
 the following distributivity condition holds:
\[\bigwedge  _{\alpha \in A, \beta \in B} (p_\alpha \oplus q_\beta) =
\bigwedge _{\alpha \in A} p_\alpha  \oplus \bigwedge _{\beta \in B} q_\beta \]

for all $p_\alpha \in {\mathcal H}$ $(\alpha \in A)$ and
$p_{\beta} \in {\mathcal H}$ $(\beta \in B)$.

\end{enumerate}

In previous papers (e.g.  \cite{pouzet}) such  a structure has been called a \emph {Heyting algebra},  or an  \emph{involutive Heyting algebra}. This terminology will be retained in this survey despite the fact that a more  appropriate term could be \emph {dual of an  integrale involutive quantale}, to refer to the notion of quantale introduced by Mulvey \cite {mulvey} in 1985. Indeed, according to the terminology of  \cite{kaarli-radeleczki} (see also \cite{eklund-al, rosenthal}),  a  \emph{quantale} is an ordered monoid  satisfying the dual of the distributivity  condition stated in $(2)$, it is \emph{involutive} if it is equipped with an involution and it is \emph{integral} if the largest  element is the neutral element of the monoid.

Besides ordinary metric  spaces, there are plenty of examples of this generalized structure. Reflexive graphs, undirected as well as directed, ordered sets, involutive and reflexive transition systems  are the basic ones. Due to the conditions imposed on $\mathcal H$ there are important classes of objects that fall beyond this framework. For examples, metric space with distances in Boolean algebras, as introduced in \cite {blumenthal}  (except if the Boolean algebra  is the power set of a set); ultrametric spaces  with values in an arbitrary poset; 
graphs which are not necessarily reflexive, or arbitrary transition systems.  Attempts to capture these situations have been made in  \cite{PR}; the case of generalized metric space over a Heyting algebra for which the least element is not necessarily the neutral element (cf.  condition $(1)$ above) being particularly studied.

We have restricted the scope of our approach  to generalized metric spaces over a Heyting algebra  because there are significant  results, easy to present and with  the potential to be extended to more general situations.

The emphasis of this presentation is on retracts and on the fixed point property. Considering the class of generalized metric spaces over a Heyting algebra $\mathcal H$,  we introduce the \emph{nonexpansive maps}  as maps $f$ from a metric space $\mathbf E:=(E,d)$ into an other, say $\mathbf E':=(E',d')$, such that
\begin{equation}
d'(f(x), f(y)) \leq d(x,y) \; \text {for all}\; x,y\in E.
\end{equation}
From this,  we derive the notions of \emph{isometry}, \emph{retraction}   and \emph{coretraction}. Since the Heyting algebra under consideration  is a complete lattice, arbitrary products of spaces can be defined, hence, as Duffus and Rival did \cite{DuRi} 1981 for graphs and posets,  we may introduce \emph{varieties} of metric spaces as classes of metric spaces closed under products and retracts. Among generalized metric spaces,  those having the \emph{fixed point property} (fpp), that is, those spaces such that every nonexpansive  map  $f
$ has a fixed point, i.e.,  some $x$ such that $f(x)=x$, have a particular interest. As in any category, (fpp) is preserved under retractions.  This elementary fact has  a significant consequence. Indeed, observing that coretractions are isometric embeddings, those generalized metric spaces for which  this necessary condition is sufficient, spaces called \emph{absolute retracts},  have a special role. If there are \emph{enough absolute retracts}, meaning that every generalized metric space isometrically embeds into an absolute retract, then absolute retracts are the natural candidates to look for spaces with the fixed point property.  Indeed, it suffices that they embed into some space with the fixed point property. This point of view is illustrated by the fact that in the category of ordered sets with ordered maps as morphisms, absolute retracts coincide with complete lattices   (Banaschewski,  Bruns \cite{banaschewski-bruns}) and  according to  the famous theorem of Tarski \cite{tarski}, these lattices have the fixed point property.  In the category of (ordinary) metric spaces with nonexpansive mappings as morphisms, the absolute retracts are the  hyperconvex metric spaces  introduced by  Aronszajn and  Panitchpakdi \cite{ArPa} and according to Sine-Soardi theorem \cite{sine, soardi}, the bounded ones have the fixed point property.

These results being expressible  in terms of generalized metric spaces, it was natural  to look at absolute retracts  in the category of generalized metric spaces over a Heyting algebra. Four basic facts obtained in \cite{JaMiPo} are presented in this paper. First, we show that on the  Heyting algebra $\mathcal H$, there is a distance $d_{\mathcal H}$ and that every metric space over $\mathcal H$ embeds isometrically into some power of the space $\mathbf H:= (\mathcal H, d_{\mathcal H})$,  equipped with the sup-distance (cf. Theorem \ref{thm:embedding}). Next, we show that the notion of absolute retract is much simpler than in other categories. It coincides with three other notions:  \emph{extension property}, \emph{injectivity} and   \emph{hyperconvexity} (cf Theorem \ref{thm: caracterisation-hyperconvexity}). This yields a straightforward extension of the characterization of absolute retracts, due to Aronszajn and  Panitchpakdi \cite{ArPa} for ordinary metric spaces.  The latter in conjunction with the fact that $\mathbf H:=(\mathcal H, d_{\mathcal H})$ is hyperconvex (Theorem \ref{thm: space values-hyperconvex}), implies that every generalized metric space embeds isometrically into an absolute retract (cf. (4) of Theorem \ref{thm: caracterisation-hyperconvexity}). The third fact is the existence of an \emph{injective envelope},  that is the existence of a minimal injective space extending  an arbitrary space isometrically (cf Theorem \ref{thm: env injective}).  For ordinary metric spaces,  this was  done by Isbell \cite{isbell}, while for posets,  Banaschewski and Bruns \cite{banaschewski-bruns} showed that  the injective enveloppe of a poset is  its  MacNeille completion. This last fact is based on the observation that, in general, coretractions  are more than isometries. Coretractions preserve \emph{holes}, that is, families of balls with empty intersection. Considering the holes preserving maps, introduced by  Duffus and Rival  for posets under the name of \emph{gap preserving maps}  \cite{DuRi}, and then by Hell an Rival for graphs \cite{hell-rival}, we show that for the hole preserving maps,  the absolute retracts  and the injectives coincide, that every generalized metric space embeds in one of them -by a hole preserving map- and consequently, that they form a variety (Theorem \ref{thm:hole-preserving}).

We  illustrate the  results about generalized metric spaces presented above   with  metric spaces,  graphs, posets and transition systems.  We start with absolute retracts.  We mention that the Aronszjan-Panitchpakdi  characterization of absolute retracts was extended to ultrametric spaces  by Bayod and Martinez \cite{bayod-martinez}. We also refer the reader to some developments in Ackerman \cite{ackerman}  Considering  reflexive and symmetric graphs, with the usual distance of the shortest path, paths are absolute retracts and every graph isometrically embeds into a product of paths (a result due independently to  Quilliot \cite{Qu1},  Nowakowski and  Rival \cite{NoRi}). Furthermore, it has   a minimal retract of product of paths (this last fact has been obtained independently by Pesch \cite{pesch}). This extends to directed graphs: Quilliot \cite{Qu1} introduced a new kind of distance,  the  \emph{zigzag distance},  on  a  directed graph $\G$. It takes into account all oriented paths joining two vertices  of $\G$.  The values of this distance are final segments of the monoid  $\Lambda^{*}$ of words over the two-letter alphabet $\Lambda:= \{+, -\}$.   The set $\F(\Lambda^{*})$ of these final segments can be viewed as a Heyting algebra.  It turns out that this Heyting algebra has not only a metric structure, but also a graph structure, rendering it an absolute retract into the category of graphs. Every directed graph embeds isometrically into a power of that graph, and the absolute retracts are retracts of products of that  graph. The notion of injective envelope  of two-element metric spaces was used to produce  a family  of finite directed graphs  generating the variety of absolute retracts. A specialization to posets of the zigzag  distance yields the notion of \emph{fence distance}  (Quilliot \cite{Qu1}); in this case, absolute retracts of posets  are retracts of product of fences (Nevermann, Rival \cite{nevermann-rival}).  A graph is a \emph{zigzag}  if it symmetrisation is a path. Oriented zigzag graphs  are absolute retracts in the variety of directed graphs,  but are too simple to generate all absolute retracts in the variety of directed graphs. The full description was given in \cite{KP2}. As shown in \cite{bandelt-pouzet-saidane},  zigzags generate the variety of absolute retracts in the category of oriented graphs. Considering  the hole preserving maps,  posets that are absolute retracts are  those with the \emph{strong selection property} (notion introduced by  Rival and  Wille \cite{rival-wille} for lattices and extended to posets by  Nevermann and Wille \cite{nevermann-wille}). For posets and graphs considered with the fence  distance  and the graph distance,  Theorem \ref{thm:hole-preserving}  is due to  Nevermann,  Rival \cite{nevermann-rival} and  Hell,  Rival \cite{hell-rival}, respectively. Of course,  it applies to directed graphs equipped with the zigzag distance  and to classical metric spaces as well.

It appears that the zigzag distance between two vertices $x$ and $y$ of a directed graph $\G:= (V, \mathcal E)$ is the language accepted by the automaton having $V$ as set of states, $T:= \{(p, +, q):  (p,q)\in \mathcal E\}\cup \{(p,-,q): (q,p)\in \mathcal E\}$ as set of transitions,  $x$ as initial state and $y$ as final state. This fact leads to the consideration of transition systems   over an arbitrary alphabet $\Lambda$ as a kind of metric spaces, the distance between two states being the language accepted between these two states. If the alphabet is equipped with an involution, we may consider \emph{reflexive} and \emph{involutive} transition systems. The distance function takes values in the set $\F(\Lambda^*)$ of final segments of the set $\Lambda^*$ of words over the alphabet $\Lambda$.  As for the two-letter  alphabet, $\F(\Lambda^*)$ is a Heyting algebra,  and our transition systems are generalized metric spaces, thus the above results apply. The existence of the injective envelope  of a two- element metric space was used to prove that $\F(\Lambda^*)$ is a free monoid \cite{KPR}. A presentation of this result is given in Section \ref{section-illustration}.

Turning to the fixed property, we might say that over the years fixed point results for discrete of for continuous structures have proliferated. The theorem by Sine-Soardi has been extended to  metric spaces endowed with a compact normal structure in the sense of Penot (Kirk's Theorem, \cite{kirk}). It has also been extended to bounded hyperconvex generalized metric spaces, with an appropriate notion of boundedness \cite{JaMiPo}. Baillon \cite{baillon} proved that arbitrary sets of commuting maps on a bounded hyperconvex metric space  has a common fixed point. Khamsi \cite{khamsi} extended the conclusion  to metric spaces with a compact normal structure. Quite recently, Khamsi and the second author (\cite {khamsi-pouzet}) extended it to generalized metric spaces endowed with a compact normal structure.  As  a consequence,   every set of commuting order-preserving maps  on a retract of a power of a finite fence, has a fixed point (the case of one map followed  from a result due to I. Rival \cite{Rival} 1976  for finite posets and  Baclawski and Bj\"orner \cite{baclawski-bjorner} for infinite posets). This applies in the same way to directed graphs (reflexive and antisymmetric) equipped with the zigzag distance  and substantially completes the results of Quilliot \cite{Qu1} (Theorem \ref{thm:cor4}).

We left  untouched some aspects of generalized metric spaces. A notable one concerns homogeneity and amalgamation. In 1927,  Urysohn \cite {urysohn} discovered  a separable metric space such that every isometry from a finite subset to an other one extends to some isometry on the whole space and, furthermore, every finite metric space embeds into it. Later on, Fra\"{\i}ss\'e \cite{fraisse} and then J\'onsson \cite {jonsson}, identified the notion of homogeneity and the test of  amalgamation,  showing that several classes of structures, now called Fra\"{\i}ss\'e classes,  and  including the class of metric spaces, had   an homogeneous structure, from which the existence of the Urysohn space was a special case. Then,  in 2005, Kechris, Pestov and Todorcevic \cite{KPT} characterized Fra\"{\i}ss\'e's classes with the Ramsey property. This characterisation led to numerous  papers on  homogeneity and particularly on (ordinary) homogeneous metric and ultrametric spaces  \cite{DLPS, DLPS2, DLPS3, vanthe}.  As indicated in \cite{JaMiPo} (Fact 4 of page 181), the class of metric space over a Heyting algebra  has the amalgamation property, thus it may have homogeneous structures (e.g. when the algebra is countable).
Independently of our work, some research has been devoted to  generalized metric spaces which are also homogeneous \cite{braunfeld, hubicka-all, sauer}.

 \section{Metric space over a Heyting algebra.} 
 In what follows,  we introduce the basic terminology, see \cite{birkhoff, bondy-murty, Fra}. 
 Let $\mathcal H$ be a complete lattice, with a least element $0$ and a 
 greatest element 1, equipped with a monoid operation $\oplus$ and an involution $-$
 satisfying the following properties:\\
 (i) The monoid operation is compatible with the ordering, that is
 $p\leq p'$ and $q \leq q'$ imply $p\oplus q\leq p'\oplus q'$ for every $p,p',q,q' \in \mathcal H$.\\
 (ii) The involution is order-preserving and reverses the monoid
 operation, that is $$\overline {p\oplus q} = \bar {q} \oplus \bar {p}\;\;{\rm  holds}
 \;\;{\rm for\; every\;}\; p, q \in \mathcal H.$$
 We say that $\mathcal H$ is a \emph{Heyting algebra} if it satisfies the
 following distributivity condition:
 
 \begin{equation}\label{heyting}
 	\bigwedge  _{\alpha \in A, \beta \in B} (p_\alpha \oplus q_\beta) =
 	\bigwedge _{\alpha \in A} p_\alpha  \oplus \bigwedge _{\beta \in B} q_\beta
 \end{equation} 
 for all $p_\alpha \in {\mathcal H}$ $(\alpha \in A)$ and
 $p_{\beta} \in {\mathcal H}$ $(\beta \in B)$ or equivalently, (because of the
 involution)
 \begin{equation}
 	\bigwedge _{\alpha \in A} (p_\alpha \oplus q)  =
 	\bigwedge _{\alpha \in A} p_\alpha\; \oplus \;  q \
 \end{equation}
 for all $p_\alpha \in {\mathcal H}$ $(\alpha \in A)$ and $q\in {\mathcal H}$. 
 
 Note
 that this distributivity condition contains the fact that the  monoid
 operation and the ordering are compatible. 
 
 In the sequel, the following assumption is made:
 \begin{equation}\label{roleof0}
 	\text{The least element}\;  0\; \text{of}\; \mathcal H \; \text{is the  neutral element of the operation}\; \oplus.
 \end{equation}
 Let E be a set. A  \emph{distance} on $E$ is a map
 $d : E\times E \rightarrow {\mathcal H}$ satisfying the following properties
 for all $x,y,z \in E$:\\
 d1) $d(x,y)\;=\;0$ if and only if $x\;=\;y$;\\
 d2) $d(x,y)\; \leq \;d(x,z)\;\oplus \;d(z,y)$;\\
 d3) $d(x,y)\;=\;\overline{d(y,x)}$.\\
 The pair $\mathbf E:=(E,d)$ is a {\it metric space} over $\mathcal H$. If there is no
 danger we will denote it $E$. If we replace the monoid operation $\oplus$
 by its reverse, that is by the
 operation $(x,y) \mapsto y\oplus x$, and we leave unchanged the ordering and 
 the involution, then the new structure ${\mathcal H}^{'}$ satisfies the same
 properties as $\mathcal H$  and so we can define distances over ${\mathcal 
 	H}^{'}$. For example, if $d : E\times E \rightarrow {\mathcal H}$ is a distance then
 $\bar {d} : E\times E \rightarrow {\mathcal H}^{'}$ defined by
 $\bar {d}(x,y) = d(y,x)$ is a distance over ${\mathcal H}'$, the
 {\it dual distance}. We denote $\bar {\mathbf E}:=(E,\bar {d})$ or simply $\bar {E}$ the
 corresponding space. For typographical reasons, we will use
 $\bar {d}(x,y)$ instead of $\overline {d(x,y)}$.
 This causes no confusion. 
 
 Let $\mathbf E:= (E,d)$ be a metric space over $\mathcal H$. For all $x\in E$ and
 $r\in \mathcal H$, we define the {\it ball with center $x$ and radius r}, as
 the set $B_{\mathbf E}(x,r)=\{y \in E:d(x,y)\leq r\}$; if there is no danger of
 confusion we will denote it $B(x,r)$ instead of $B_{\mathbf E}(x,r)$.
 
 If $\mathbf E:= (E,d)$ and $\mathbf E':= (E',d')$ are two metric spaces over $\mathcal H$,
 then a map
 $f : E \rightarrow E^{'}$ is \emph{nonexpansive} (or \emph{contracting})
 provided that 
 \begin{equation} \label{eq:non-expansive} 
 	d'\left(f(x),f(y)\right) \leq d(x,y)\;\text{for all}\;x,y\in E.
 \end{equation}
 If in inequality (\ref{eq:non-expansive}) the equality holds for all $x,y\in E$, then $f$ is an
 {\it isometry} of $\mathbf E$ in $\mathbf E'$. Hence, in our terminology, an isometry is not necessarily
 surjective. We say that $\mathbf E$ and $\mathbf E'$ are {\it isomorphic}, a fact we 
 denote $\mathbf E\cong \mathbf E'$, if there is a surjective isometry from $\mathbf E$ onto $\mathbf E'$.
 If  $E$ is a subset of $E'$ and the identity map $\id: E \rightarrow E'$ is nonexpansive, we
 say that $\mathbf E$ is a {\it subspace} of $\mathbf E'$, or that
 $\mathbf E'$ is an {\it extension } of $\mathbf E$. If, moreover, this map 
 is an isometry (that is $d$ is the restriction of $d'$ to $E' \times E'$),
 then we call $\mathbf E$ an {\it isometric subspace} of $\mathbf E'$ and
 $\mathbf E'$  an {\it isometric extension} of $\mathbf E$. The \emph{restriction of $d'$ to $\mathbf E$}, denoted by $d'_{\restriction E}$,   is the restriction of the map $d'$ to $E\times E$. This is a distance, the resulting space, denoted by  $\mathbf E'_{\restriction E}:= (E, d'_{\restriction E})$,
 is the \emph{restriction of $\mathbf E'$ to $\mathbf E$}; this is an isometric subspace of $\mathbf E'$.  As usual in categories, $Hom(\mathbf E, \mathbf E')$ denote  the set of all nonexpansive maps from $\mathbf E$ to $\mathbf E'$

 The fact that $\mathcal H$ is a complete lattice allows to define arbitrary product of metric spaces.
 If $(\mathbf  E_i)_{i\in I}$, where  $\mathbf E_i:=(E_i,d_i)$,  is a family of metric spaces over
 $\mathcal H$, the direct product $\mathbf E:=
 \displaystyle \prod_{i\in I}\mathbf E_i$, is the cartesian product
 $E:= \displaystyle \prod_{i\in I} E_i$, equipped with the "sup"
 (or $\ell^{\infty}$) distance $d : E\times E \to {\mathcal H}$ defined by:
 \[d\Big((x_i)_{i\in I},(y_i)_{i\in I} \Big) :=
 \bigvee_{i \in I} d_i(x_i,y_i). \]\vskip2mm
 The distributivity condition on $\mathcal H$ allows to define a distance on the space of values $\mathcal H$. This fact relies on the classical notion of {\it residuation} (see \cite{ward-dilworth, blyth-janowitz}).\\
 Let $v\in \mathcal H$. Given $\gamma \in {\mathcal H}$, the sets
 $\{r \in {\mathcal H}: v \leq r \oplus \gamma\}$ and
 $\{r \in {\mathcal H}: v \leq  \gamma \oplus r \}$ have least elements, that we 
 denote respectively $\lceil v - \gamma \rceil$ and $\lceil -\gamma \oplus  v  \rceil$
 (where, in fact, $\overline {\lceil- \gamma \oplus  v \rceil} =
 \lceil \bar v - \bar \gamma \rceil$). It follows that for all
 $p, q \in \mathcal H$, the set
 $$D(p,q):=\{r \in {\mathcal H}: p\leq q \oplus \bar r\;\;{\rm and}\; \; q\leq 
 p\oplus r\}$$
 has a least element, namely $\lceil \bar p - \bar q \rceil \vee
 \lceil - p \oplus q \rceil$. We set 
 \begin{equation}\label{eq:distance}
 	d_{\mathcal H}(p,q):= Min D(p,q).
 \end{equation}
 
 As shown in \cite{JaMiPo}:
 \begin{theorem}\label{thm:embedding}
 	The map $(p,q) \mapsto d_{\mathcal H}(p,q)$ is a distance on $\mathcal H$ and every metric space over $\mathcal H$  embeds isometrically into a power of the space $\mathbf H:=({\mathcal H},d_{\mathcal H})$.
 \end{theorem}  
 This result follows from the fact that for
 every metric space $\mathbf E:=(E,d)$ over $\mathcal H$, and for all $x, y\in E$,
 the following equality holds:
 \begin{equation}\label{eq:sup-distance}
 	d(x,y) =  \bigvee_{z\in E} d_{\mathcal H}\left(d(z,x),d(z,y)\right).
 \end{equation}
 Indeed, for each $x\in E$, let
 $\bar \delta(x) : E \to \mathcal H$ be the map defined by
 $\bar \delta(x)(z) = d(z,x)$ for all $z \in E$; the equality above expresses
 that the map from $\mathbf E$ into the power  $\mathbf H^{E}$
 is an isometric embedding (on an other hand,  this equality expresses that
 $\bar \delta (x)$ is a nonexpansive map from $\mathbf E$ into
 $\bar{\mathcal H}:=({\mathcal H}',d_{{\mathcal H'}})$). 
  \section{Examples}
 \subsection{Ordinary metric and ultrametric spaces}
 Let  $\mathcal H:=\mathbb{R}^{+} \cup \{ +\infty\}$ with the  addition on the non-negative reals extended to $\mathcal H$ in the obvious way, the involution being the identity. 
 The metric spaces  we get are just unions of disjoint copies of ordinary metric spaces. The fact that we add to $\mathbb{R}^{+}$ an infinite value is an inessential difference.
 We do it  to make $\mathcal H$  a complete poset  and have infinite products, this  avoiding   $\ell^{\infty}$ type constructions. On $\mathbf H:= (\mathcal H, d_{\mathcal H})$, the distance is the absolute value, except that the distance from $\infty$ to any other element is $\infty$. 
 Every space in our sense embeds isometrically into a power of $\mathbf H$ and in fact in a power of  $\mathbb{R}^{+}$ equipped with the absolute value. On the other hand, every ordinary metric space embeds isometrically into some $\ell ^{^{\infty}}_{\mathbb{R}}(I)$, the space of bounded families  $(x_i)_{i\in I}$ of real numbers, endowed with the sup-distance.   
 
 If the monoid operation on  $\mathbb{R}^{+} \cup \{ +\infty\}$ is the join and the involution is the identity,  distances are called \emph{ultrametric distances} and metric spaces are called \emph{ultrametric spaces} (see \cite{bayod-martinez}). The notion of ultrametric spaces has been generalized by several authors (see \cite{priess-crampe-ribenboim1, priess-crampe-ribenboim2}, \cite{ackerman},  \cite{braunfeld}). The general setting for the space of values is a join semilattice with a least element.
 
 A \emph{join-semilattice} is an ordered set in which two arbitrary elements $x$ and $y$  have  a join, denoted by $x\vee y$,  defined as the least element of the set of common upper bounds of $x$ and $y$.
 
 Let $\mathcal H$ be  a join-semilattice with a least element, denoted by $0$.
 A \emph{pre-ultrametric space} over $\mathcal H$ is a pair $\mathbf D:=(E,d)$ where $d$ is a map from $E\times E$ into $\mathcal H$ such that for all $x,y,z \in E$:
 \begin{equation} \label{eq:ultra1} 
 	d(x,x)=0,\; d(x,y)=d(y,x)  \text{~and } d(x,y)\leq d(x,z)\vee d(z, y).
 \end{equation}
 The map $d$ is an \emph{ultrametric distance} over $\mathcal H$ and $\mathbf D$ is an \emph{ultrametric space}  over $\mathcal H$ if $\mathbf D$ is a pre-ultrametric space and $d$ satisfies \emph{the separation axiom}:
 \begin{equation} \label{eq:ultra2}
 	d(x,y)=0\; \text{implies} \;  x=y .
 \end{equation}
 
 Any binary relational structure  $\mathbf M:=(E, (\mathcal E_i)_{i\in
 	I})$  in which each $\mathcal E_i$ is an equivalence relation  on  the  set $E$ can be viewed as a pre-ultrametric space on $E$. Indeed, 
 given a set $I$,  let $\powerset (I)$ be the power set of $I$. Then $\powerset (I)$, ordered by inclusion, is a join-semilattice (in fact a complete Boolean algebra)  in which the join is the union, and  $0$  the empty set.
 For $x,y\in E$, set $d_{\mathbf M}(x,y):=\{i\in I: (x,y)\not \in \mathcal E_i\}$. Then   the pair $\mathbf D_{\mathbf M}:=(E, d_{\mathbf  M})$ is a pre-ultrametric space over $\powerset (I)$.
 Conversely, let $\mathbf D:=(E,d)$ be a pre-ultrametric space over $\powerset (I)$. For every $i\in I$ set $\mathcal E_i:=\{(x,y)\in E\times E: i\not \in d(x,y) \}$ and let  $\mathbf M:=(E, (\mathcal E_i)_{i\in
 	I})$. Then $\mathcal E_i$ is an equivalence relation on $E$ and $d_{\mathbf M}=d$.
 Furthermore,  $\mathbf D_{\mathbf M}$ is an ultrametric space if and only if $\bigcap_{i\in I} \mathcal E_i= \Delta_E:=\{(x,x): x\in E\}.$  
 
 The congruences of an algebra form an important class of equivalence relations; they can be studied in terms of ultrametric spaces  (see Section \ref{section:further development} for an example).  If we suppose that our distributivity condition holds, which is for example the case if the set of values  is a finite distributive lattice, the study of these ultrametric spaces  fits in the study of metric spaces over a Heyting algebra. This case was particularly studied in \cite{PR}  and more recently in   \cite{ ackerman, braunfeld, pouzet2}. 
 \vskip2mm

 \subsection{Graphs and  digraphs}         
 A \emph{binary relation} on a set $E$ is a subset $\mathcal E$ of $E\times E$,  the set of ordered pairs $(x,y)$ of elements of $E$. The \emph{inverse} of $\mathcal E$ is the binary relation $\mathcal E^{-1}:= \{(x,y): (y,x)\in \mathcal E\}$. 
 The \emph{diagonal} of $E$ is the set $\Delta_{E}=\{(x,x):x\in E\}$. A \emph{directed graph} $\G$ is a pair $(E, \mathcal E)$ where $\mathcal E$ is a binary relation on $E$. We say that  $\G$ is  \emph{reflexive} if $\mathcal E$ is \emph{reflexive}, that is, contains the diagonal $\Delta_{E}$,  and that $G$ is \emph{oriented} if $\mathcal E$ is \emph{antisymmetric}, that is, $(x,y)$ and $(y,x)$ cannot be in $\mathcal E$ simultaneously except if $x =y$. If $\mathcal E$ is \emph{symmetric}, that is $\mathcal E=\mathcal E^{-1}$,  we identify it with a subset of pairs of $E$ and we say that the graph is \emph{undirected}. If  $\G:= (E, \mathcal E)$ and $\G':= (E', \mathcal E')$ are two directed graphs,  a \emph{homomorphism from $\G$ to $\G'$} is a map $h: E\rightarrow  E'$ such that $(h(x), h(y)) \in \mathcal E'$ whenever $(x,y)\in \mathcal E$ for every $(x,y)\in E\times E$.
 
 In the sequel, all graphs we consider will be reflexive.  Hence, graph-homomorphisms  can send edges or arcs on loops.  We refer to \cite{bondy-murty} for the terminology on graphs.
 
 \subsubsection{Reflexive graphs} 
 Let $\mathcal H$ be the complete lattice on three elements such that \lq\lq$0< \frac{1}{2}< 1$\rq\rq.
 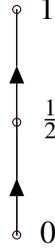
\begin{figure}[H]
 	\centering
 	\definecolor{ttqqqq}{rgb}{0.2,0.,0.}
 	\begin{tikzpicture}[line cap=round,line join=round,>=triangle 45,x=0.75cm,y=0.75cm]
 	\clip(1.,-1.5) rectangle (3.,3.5);
 	\draw [->,line width=0.5pt] (2.,-1.) -- (2.,0.);
 	\draw [->,line width=0.5pt] (2.,1.) -- (2.,2.);
 	\draw [line width=0.5pt] (2.,0.)-- (2.,1.);
 	\draw [line width=0.5pt] (2.,2.)-- (2.,3.);
 	\draw [line width=0.5pt] (2.,-1.)-- (2.,1.);
 	\draw (2.22,-0.56) node[anchor=north west] {$0$};
 	\draw (2.24,1.6) node[anchor=north west] {$\frac{1}{2}$};
 	\draw (2.2,3.4) node[anchor=north west] {$1$};
 	\begin{scriptsize}
 	\draw [color=ttqqqq] (2.,-1.) circle (1.5pt);
 	\draw [color=ttqqqq] (2.,1.) circle (1.5pt);
 	\draw [color=ttqqqq] (2.,3.) circle (1.5pt);
 	\end{scriptsize}
 	\end{tikzpicture}
 	\caption{The ordered monoid $\mathcal{H}$. }
 \end{figure}

 The monoid operation is defined by $ x\oplus y=\min\{x+y,1\}$ and  the involution is the identity.\\ 
 Every symmetric reflexive graph $\G:=(E,\mathcal{E})$ is a metric space over  $\mathcal{H}$. The distance $d:E\times E\longrightarrow \mathcal H$ is defined by:
 \begin{enumerate}
 	\item $d(x,y)=1 \;\;\text{if} \;\;(x,y) \notin \mathcal{E};$
 	\item   $d(x,y)= \frac{1}{2} \;\;\text{if} \;\;  (x,y)\in \mathcal{E}\; \text{and}\; x\neq y;$
 	\item $ d(x,y)=0 \;\;\text{if} \;\; x=y.$     
 \end{enumerate}    
 Conversely every metric space $\mathbf E:= (E,d)$ over $\mathcal H$ can be viewed as a symmetric reflexive graph; the vertices are the elements of $E$ and the set of edges $\mathcal{E}$ (including the loops) is defined as follows: 
 \begin{displaymath} (x,y)\in \mathcal{E} \Longleftrightarrow d(x,y)\leq \frac{1}{2}. 
 \end{displaymath}    
 Nonexpansive maps correspond to graph-homomorphisms (provided that  edges are  sended on edges or loops). 
 \begin{sloppypar}
 The distance $d_{\mathcal H}$ on the Heyting algebra takes value $\frac {1}{2}$ on the pairs $(x,y)\in \{ (0,\frac {1}{2}), (\frac {1}{2}, 0), (\frac {1}{2}, 1), (1,  \frac {1}{2})\}$, value $1$  on the pairs $(0,1)$ and $(1,0)$, and $0$ on the diagonal. The corresponding graph $\mathbf G_{\mathcal H}$ is the  path  $P_3$ on three vertices with $\frac {1}{2}$ as a middle  point.  
 \end{sloppypar}
 \subsubsection{Reflexive digraphs}
 Let $\mathcal H$ be the complete lattice on five elements $\{0,\frac{1}{2} , + , - , 1\}$, represented below:
 \begin{figure}[H]
 	\centering
 	\definecolor{ttqqqq}{rgb}{0.2,0.,0.}
 	\begin{tikzpicture}[line cap=round,line join=round,>=triangle 45,x=1cm,y=1cm]
 	\clip(0.,-3.5) rectangle (6.,3.75);
 	\draw [line width=0.5pt,color=ttqqqq] (3.,-3.)-- (3.,-1.);
 	\draw [line width=0.5pt,color=ttqqqq] (3.,-1.)-- (5.,1.);
 	\draw [line width=0.5pt,color=ttqqqq] (3.,-1.)-- (1.,1.);
 	\draw [line width=0.5pt,color=ttqqqq] (1.,1.)-- (3.,3.);
 	\draw [line width=0.5pt,color=ttqqqq] (5.,1.)-- (3.,3.);
 	\draw [->,line width=0.5pt,color=ttqqqq] (3.,-3.) -- (3.,-2.);
 	\draw [->,line width=0.5pt,color=ttqqqq] (3.,-1.) -- (2,0);
 	\draw [->,line width=0.5pt,color=ttqqqq] (3.,-1.) -- (4,0);
 	\draw [->,line width=0.5pt,color=ttqqqq] (1.,1.) -- (2,2);
 	\draw [->,line width=0.5pt,color=ttqqqq] (5.,1.) -- (4,2);
 	\draw (2.8,-3.1) node[anchor=north west] {$0$};
 	\draw (2.84,3.5) node[anchor=north west] {$1$};
 	\draw (0.4,1.2) node[anchor=north west] {$+$};
 	\draw (5.1,1.2) node[anchor=north west] {$-$};
 	\draw (3.26,-0.64) node[anchor=north west] {$\frac{1}{2}$};
 	\begin{scriptsize}
 	\draw [color=ttqqqq] (3.,-3.) circle (1.5pt);
 	\draw [color=ttqqqq] (3.,-1.) circle (1.5pt);
 	\draw [color=ttqqqq] (5.,1.) circle (1.5pt);
 	\draw [color=ttqqqq] (1.,1.) circle (1.5pt);
 	\draw [color=ttqqqq] (3.,3.) circle (1.5pt);
 	\end{scriptsize}
 	\end{tikzpicture}
 	\caption{The ordered monoid  $\mathcal{H}$. }
 \end{figure}
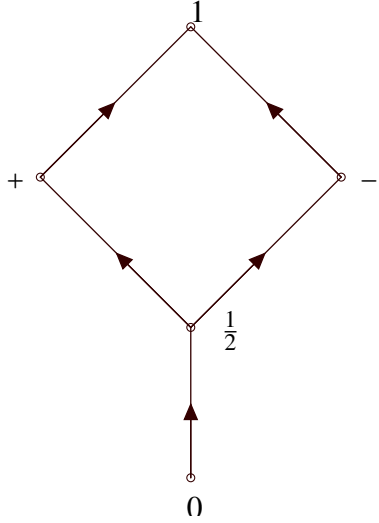
 
 The monoid operation is defined by
 $$\left\{
 \begin{array}{ll}
 x\oplus y=1  \;\;\text{if} \; x,y\geq \frac{1}{2};\\
 x\oplus y=\max(x,y)  \;\;\text{if not}.
 \end{array}  \right.$$
 The  involution exchanges $+$ and $-$  and fixes $0$, $\frac{1}{2}$ and $1$.\\
 If $\G:=(E,\mathcal{E})$ is a reflexive directed graph, the application   $d:E\times E\longrightarrow \mathcal{H}$ defined by 
 \begin{enumerate}
 	\item $d(x,y)=1  \;\;\text{if}\; (x,y)\notin \mathcal{E} \cup \mathcal{E}^{-1};$
 	\item  $ d(x,y)=+    \;\;\text{if}\;   (x,y)\in \mathcal{E} \backslash \mathcal{E}^{-1};$
 	\item $ d(x,y)=-    \;\;\text{if}\;  (x,y)\in \mathcal{E}^{-1}\backslash \mathcal{E};$
 	\item  $d(x,y)=\frac{1}{2}   \;\;\text{if}\;   (x,y)\in \mathcal{E} \bigcap \mathcal{E}^{-1} \backslash \Delta_{E};$
 	\item $d(x,y)=0    \;\;\text{if}\;   (x,y)\in  \Delta_{E},$     
 \end{enumerate}
 is a distance  on $E$.\\
 Conversely every metric space $(E,d)$ over $\mathcal H$ can be viewed as a  reflexive digraph; the vertices are the elements of $E$ and the set of arcs $\mathcal{E}$ is defined as follows:  \begin{displaymath} (x,y)\in \mathcal{E} \Longleftrightarrow d(x,y)\leq +. \end{displaymath}

 \subsubsection{The graphic distance}
 A graph $\mathbf P$ is a {\it path} if we can enumerate the vertices in a
 non-repetitive sequence
 $(x_{i})_{i \in  I}$ where $I$ is of the form $\{0,1,...,n\}$ or $\N$,
 the set of non-negative integers, or $\Z$,
 the set of relative integers, such that $(x_{i},x_{j})$ forms an edge
 if and only if $|j - i| \leq 1$;
 the path $\mathbf P$ is {\it finite} if $I = \{0,1,...,n\}$ and in this case $n$
 is its {\it length} whereas $\mathbf P$
 is {\it infinite} if $I = \N$, and  {\it doubly infinite} if  $I = \Z$. If $\G:=(V,\mathcal{E})$ is an (undirected) graph, the  {\it graphic distance} is the map $d_{G}:V\times V\longrightarrow \N\cup\{+\infty\} $ for which  $d_{G}(x,y)$ is the length of the shortest path connecting
 $x$ to $y$ if there is a  such a path and $+\infty$ otherwise. This is a distance on $\mathcal{H}:= (\N\cup\{+\infty\}, \oplus)$ where $\oplus$ is the ordinary sum. The distance on $\mathcal{H}$ defined by means of  Formula (\ref{eq:distance})   is the graphic distance associated with the graph $\G_{\mathcal H}$ made of a one way infinite path and an isolated vertex. Not every metric space over $\mathcal H$ comes from a graph. Still, with the fact that $\G_{\mathcal H}$ embeds isometrically into an infinite  product of finite paths,  it follows from  Theorem \ref{thm:embedding} that   \emph{every graph embeds into a product of finite paths}, a result due to Nowakowski-Rival \cite{NoRi} and Quilliot \cite{Qu1}.

 \subsubsection{The zigzag distance} \label{subsection:the zigzag distance} 
 A \emph{reflexive zigzag} is a reflexive graph  $\mathbf L$ such that the symmetric hull is a path. If  $\mathbf L:= (L, \mathcal L)$ is a finite reflexive oriented zigzag, we may enumerate the vertices in a non-repeating sequence $v_0:= x, \dots,  v_{n}:= y$ and to this enumeration we may  associate the finite sequence $ev(\mathbf L):= \alpha_0\cdots \alpha_i \cdots\alpha_{n-1}$ of $+$ and $-$,   where $\alpha_i:= +$ if $(v_i,v_{i+1})\in \mathcal L$  and $\alpha_i:= -$ if $(v_{i+1},v_{i})\in \mathcal L$. We call  such a sequence a \emph{word} over the \emph{alphabet} $\Lambda:= \{+,-\}$.  If the path has just one vertex, the corresponding  word  is   the empty word, that we denote by $\Box$.  Conversely, to a finite word  $u:= \alpha_0\cdots \alpha_i \cdots\alpha_{n-1}$ over $\Lambda$ we may associate the reflexive oriented zigzag  $\mathbf L_u:= (\{0, \dots n\}, \mathcal L_{u})$ with end-points $0$ and $n$ (where $n$ is the length $\mid u\mid$ of $u$) such that  $$\mathcal L_{u}= \{(i,i+1): \alpha_i=+\}\cup \{(i+1, i): \alpha_i=-\}\cup \Delta_{\{0, \dots, n\}}.$$
 
 \begin{figure}[H]
 	\begin{center}
 		\unitlength=1cm
 		\begin{tikzpicture}
 		
 		\draw (-0.2,2.1) circle (.2);
 		\put(0,2){\circle*{.15}}
 		\put(-1,1){\circle*{.15}}
 		\draw (-1.1,1.2) circle (.2);
 		\put(1,3){\circle*{.15}}
 		\draw (0.9,3.2) circle (.2);
 		\put(2,4){\circle*{.15}}
 		\draw (1.9,4.2) circle (.2);
 		\put(3,3){\circle*{.15}}
 		\draw (3.1,3.2) circle (.2);
 		\put(4,2){\circle*{.15}}
 		\draw (4,1.8) circle (.2);
 		\put(5,3){\circle*{.15}}
 		\draw (5,3.2) circle (.2);
 		\put(6,2){\circle*{.15}}
 		\draw (6.1,2.2) circle (.2);
 		\put(7,1){\circle*{.15}}
 		\draw (7.1,1.2) circle (.2);
 		\put(0,2){\vector(1,1){.5}}
 		\put(-1,1){\vector(1,1){.5}}
 		\put(-0.5,1.5){\line(1,1){.5}}
 		\put(.5,2.5){\vector(1,1){1}}
 		\put(2,4){\line(-1,-1){.5}}
 		
 		\put(4,2){\vector(-1,1){.5}}
 		\put(3.5,2.5){\line(-1,1){.5}}
 		\put(3,3){\vector(-1,1){.5}}
 		\put(2.5,3.5){\line(-1,1){.5}}
 		\put(4,2){\vector(1,1){0.5}}
 		\put(7,1){\vector(-1,1){.5}}
 		\put(4.5,2.5){\line(1,1){0.5}}
 		\put(6.5,1.5){\line(-1,1){0.5}}
 		\put(6,2){\vector(-1,1){0.5}}
 		\put(5.5,2.5){\line(-1,1){0.5}}

 		\end{tikzpicture}
 		\caption{A reflexive oriented zigzag.}
 	\end{center}
 \end{figure}
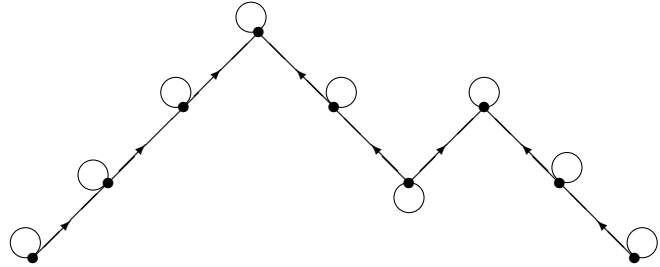

 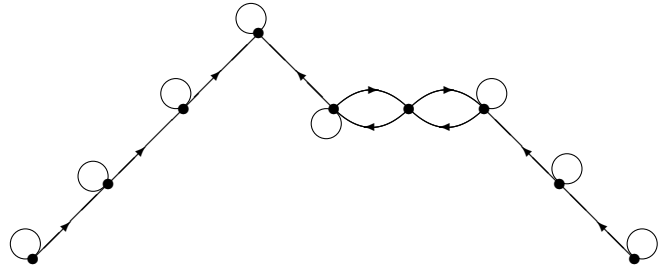
\begin{figure}[H]
 	\begin{center}
 		\unitlength=1cm
 		\begin{tikzpicture}
 		
 		\draw (-0.2,2.1) circle (.2);
 		\put(0,2){\circle*{.15}}
 		\put(-1,1){\circle*{.15}}
 		\draw (-1.1,1.2) circle (.2);
 		\put(1,3){\circle*{.15}}
 		\draw (0.9,3.2) circle (.2);
 		\put(2,4){\circle*{.15}}
 		\draw (1.9,4.2) circle (.2);
 		\put(3,3){\circle*{.15}}
 		\draw (2.9,2.8) circle (.2);
 		\put(4,3){\circle*{.15}}
 		\put(5,3){\circle*{.15}}
 		\draw (5.1,3.2) circle (.2);
 		\put(6,2){\circle*{.15}}
 		\draw (6.1,2.2) circle (.2);
 		\put(7,1){\circle*{.15}}
 		\draw (7.1,1.2) circle (.2);
 		\put(0,2){\vector(1,1){.5}}
 		\put(-1,1){\vector(1,1){.5}}
 		\put(-0.5,1.5){\line(1,1){.5}}
 		\put(.5,2.5){\vector(1,1){1}}
 		\put(2,4){\line(-1,-1){.5}}
 		\bezier{200}(3,3)(3.5,3.5)(4,3)
 		\put(3.5,3.25){\vector(1,0){.1}}
 		
 		\bezier{200}(4,3)(4.5,3.5)(5,3)
 		\put(4.5,3.25){\vector(1,0){.1}}
 		
 		\bezier{200}(3,3)(3.5,2.5)(4,3)
 		\put(3.5,2.75){\vector(-1,0){.1}}
 		
 		\bezier{200}(4,3)(4.5,2.5)(5,3)
 		\put(4.5,2.75){\vector(-1,0){.1}}
 		\put(3,3){\vector(-1,1){.5}}
 		\put(2.5,3.5){\line(-1,1){.5}}
 		
 		\put(7,1){\vector(-1,1){.5}}
 		
 		\put(6.5,1.5){\line(-1,1){0.5}}
 		\put(6,2){\vector(-1,1){0.5}}
 		\put(5.5,2.5){\line(-1,1){0.5}}

 		\end{tikzpicture}
 		\caption{A reflexive directed zigzag.}
 	\end{center}
 	
 \end{figure}

 Let $\G:= (E, \mathcal E)$ be a reflexive directed graph. For each pair $(x,y)\in E\times E$, the \emph{zigzag distance}  from $x$ to $y$ is the set $d_{\mathbf G}(x,y)$ of words $u$ such that there is a nonexpansive map $h$ from $\mathbf L_u$ into $\G$ which sends $0$ on $x$ and $\mid u\mid$ on $y$.
  
 Because of the reflexivity of $\G$, every word obtained from a word belonging to $d_{\G}(x,y)$ by inserting letters  will also be into $d_{\G}(x,y)$. This leads to  the following framework.
 
 Let $\Lambda^*$ be  collection of words over the alphabet  $\Lambda:= \{+,-\}$. Extend the involution on $\Lambda$ to $\Lambda^*$ by setting $\overline \Box:= \Box$ and $\overline {u_0\cdots u_{n-1}}:= \overline{u_{n-1}}\cdots \overline{u_{0}}$ for every word in $\Lambda^*$. Order  $\Lambda^{\ast}$ by the \emph{subword ordering}, denoted by $\leq$. If  $u: = \alpha_{1} \alpha_{2} \ldots \alpha_{m},  v: = \beta_{1}
 \beta_{2} \ldots \beta_{n}\in \Lambda^{*}$ set:
 \begin{equation}
 	u\leq v \;
 	\text{if and only if }\;
 	\alpha_{j}  =\beta_{i_{j}}\ {\rm for\ all}\ j = 1, \ldots m\; \text{with some}\; 1\leq j_1<\dots j_m\leq n.
 \end{equation}
 Let $\mathbf {F}(\Lambda^*)$ be  the set of final segments of $\Lambda^*$, that is subsets $F$ of $\Lambda^{\ast}$ such that $u\in F$ and $u\leq v$ imply $v\in F$.
 Setting $\overline X:= \{\overline u: u\in X\}$ for a set $X$ of words, we observe that $\overline X$ belongs to  $\mathbf {F}(\Lambda^*)$.
 Order $\mathbf {F}(\Lambda^*)$  by reverse of the inclusion, denote by   $0$ its least element (observe that it  is $\Lambda^*$, the final segment generated by the empty word), set $uv$ for the concatenation of two words $u,v \in \Lambda^*$, $X\oplus Y$ for the concatenation $XY:= \{uv: u\in X, v\in Y\}$.
 Then,  one  sees that $\mathcal H_{\Lambda}:= (\mathbf {F}(\Lambda^*), \oplus, \supseteq,  0, -)$ is an involutive Heyting algebra. This leads us to consider distances and metric spaces over $\mathcal H_{\Lambda}$.
 
 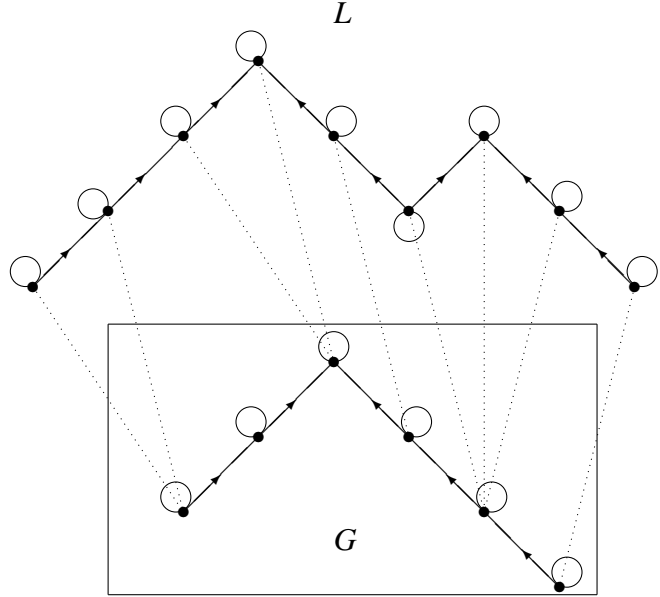
\begin{figure}[H]
 	\begin{center}
 		\unitlength=1cm
 		\begin{tikzpicture}
 		\draw (-0.2,2.1) circle (.2);
 		\put(0,2){\circle*{.15}}
 		\put(-1,1){\circle*{.15}}
 		\draw (-1.1,1.2) circle (.2);
 		\put(1,3){\circle*{.15}}
 		\draw (0.9,3.2) circle (.2);
 		\put(2,4){\circle*{.15}}
 		\draw (1.9,4.2) circle (.2);
 		\put(3,3){\circle*{.15}}
 		\draw (3.1,3.2) circle (.2);
 		\put(4,2){\circle*{.15}}
 		\draw (4,1.8) circle (.2);
 		\put(5,3){\circle*{.15}}
 		\draw (5,3.2) circle (.2);
 		\put(6,2){\circle*{.15}}
 		\draw (6.1,2.2) circle (.2);
 		\put(7,1){\circle*{.15}}
 		\draw (7.1,1.2) circle (.2);
 		\put(0,2){\vector(1,1){.5}}
 		\put(-1,1){\vector(1,1){.5}}
 		\put(-0.5,1.5){\line(1,1){.5}}
 		\put(.5,2.5){\vector(1,1){1}}
 		\put(2,4){\line(-1,-1){.5}}
 		
 		\put(4,2){\vector(-1,1){.5}}
 		\put(3.5,2.5){\line(-1,1){.5}}
 		\put(3,3){\vector(-1,1){.5}}
 		\put(2.5,3.5){\line(-1,1){.5}}
 		\put(4,2){\vector(1,1){0.5}}
 		\put(7,1){\vector(-1,1){.5}}
 		\put(4.5,2.5){\line(1,1){0.5}}
 		\put(6.5,1.5){\line(-1,1){0.5}}
 		\put(6,2){\vector(-1,1){0.5}}
 		\put(5.5,2.5){\line(-1,1){0.5}}
 		
 		\put(1,-2){\circle*{.15}}
 		\draw (0.9,-1.8) circle (.2);
 		\put(2,-1){\circle*{.15}}
 		\draw (1.9,-0.8) circle (.2);
 		\put(3,0){\circle*{.15}}
 		\draw (3,0.2) circle (.2);
 		\put(4,-1){\circle*{.15}}
 		\draw (4.1,-0.8) circle (.2);
 		\put(5,-2){\circle*{.15}}
 		\draw (5.1,-1.8) circle (.2);
 		\put(6,-3){\circle*{.15}}
 		
 		\draw (6.1,-2.8) circle (.2);
 		\put(1,-2){\vector(1,1){0.5}}
 		\put(1.5,-1.5){\line(1,1){0.5}}
 		\put(2,-1){\vector(1,1){0.5}}
 		\put(2.5,-0.5){\line(1,1){0.5}}
 		\put(6,-3){\vector(-1,1){0.5}}
 		\put(5.5,-2.5){\line(-1,1){0.5}}
 		\put(5,-2){\vector(-1,1){0.5}}
 		\put(4.5,-1.5){\line(-1,1){0.5}}
 		\put(4,-1){\vector(-1,1){0.5}}
 		\put(3.5,-0.5){\line(-1,1){0.5}}
 		\draw[dotted] (-1,1)to (1,-2);
 		\draw[dotted] (-0,2)to (1,-2);
 		\draw[dotted] (2,4)to (3,0);
 		\draw[dotted] (1,3)to (3,0);
 		\draw[dotted] (7,1)to (6,-3);
 		\draw[dotted] (3,3)to (4,-1);
 		\draw[dotted] (4,2)to (5,-2);
 		\draw[dotted] (5,3)to (5,-2);
 		\draw[dotted] (6,2)to (5,-2);
 		\put(0,-3.1){\line(0,1){3.6}}
 		\put(0,0.5){\line(1,0){6.5}}
 		\put(6.5,.5){\line(0,-1){3.6}}
 		\put(6.5,-3.1){\line(-1,0){6.5}}
 		\put(3,4.5){$L$}
 		\put(3,-2.5){$G$}
 		\end{tikzpicture}

 		\caption{A morphism of an oriented zigzag $L$ into a directed graph $G$. }
 	\end{center}
 \end{figure}
 
 There are two simple and crucial facts about the consideration of the zigzag distance  (see \cite {JaMiPo}).
 \begin{lemma}\label{graphmorphisms} A map from a reflexive directed graph $\G$ into an other is a graph-homomorphism iff it is nonexpansive.
 \end{lemma}
 \begin{lemma}\label{lem:connexity}
 	The distance $d$ of  a metric space $\mathbf E:= (E,d)$ over $\mathcal H_{\Lambda}$ is the zigzag distance  of a reflexive directed graph $\G:= (E, \mathcal E)$ iff  it satisfies the following property for all  $x,y,z \in E$, $u, v\in \Lambda^*$:
 	$uv \in d(x,y)$ implies $u\in d(x,z)$ and $v\in d(z,y)$ for some $z\in E$. When this condition holds, $(x,y)\in \mathcal E$ iff $+\in d(x,y)$.
 \end{lemma}
 On account of Lemma \ref{lem:connexity}, the various metric spaces mentioned in the introduction (injective, absolute retracts, etc.) are graphs equipped with the zigzag distance;  in particular, the distance $d_{\mathcal H_{\Lambda}}$  defined on $\mathcal H_{\Lambda}$ is the zigzag distance of some graph,  say $\mathbf G_{\mathcal H_{\Lambda}}$. According to  Theorem \ref{thm:embedding}, every graph embeds isometrically into some power of $\mathbf G_{\mathcal H_{\Lambda}}$. This graph   is countably infinite (this follows from Higman's theorem on words \cite {Hi}) but it is not easy to describe. From the study of hyperconvexity (see Section \ref{section:hyperconvex}) it follows that it embeds isometrically (w.r.t. the zigzag distance) into a product of its restrictions to principal initial segments of  $\mathcal H_{\Lambda}$. Hence every graph isometrically  embeds into a product of these finite graphs.  This later  fact leads to a fairly precise description of absolute retracts  in the category of reflexive directed graphs (see \cite{KP2}).\\
 The notion of zigzag distance  is due to Quilliot \cite{Qu1, Qu2}. He considered reflexive directed graphs, not necessarily oriented  and, in defining the distance, considered only  oriented paths. The consideration of the set of values of the distance, namely $\mathcal H_{\Lambda}$,  is in \cite {pouzet}. A general study is presented in \cite{JaMiPo}; some developments appear in \cite{Sa} and \cite{KP2}.
 \subsection{Ordered sets}  
 Let $\mathcal H$ be the following structure. The domain is the set  $\{0, +,-, 1\}$. The order is $0\leq +,-\leq 1$ with $+$ incomparable to $-$; the involution exchanges $+$ and $-$ and fixes $0$ and $1$; the operation $\oplus$ is defined by $p\oplus q:= p\vee q$ for every $p,q \in V$. As it is easy to check, $\mathcal H$ is an involutive Heyting algebra.
 If $(E, d)$ is metric space over $\mathcal H$, then $\mathbf P_{d}:=(E, \delta_{+})$,   where $\delta_{+}:= \{(x,y): d(x,y)\leq +\}$,  is an ordered set. Conversely,  if $\mathbf P:= (E, \leq )$ is an ordered set, then the map $d:E\times E\rightarrow \mathcal H$ defined by $d(x,y):= 0$ if $x=y$,  $d(x,y):= +$ if $x<y$, $d(x,y):=-$ if $y<x$ and $d(x,y):= 1$ if $x$ and $y$ are incomparable is a distance over $\mathcal H$. Clearly, if $\mathbf E:=(E,d)$ and $\mathbf E':= (E',d')$ are two metric spaces over $\mathcal H$, a map $f:E\rightarrow E'$ is nonexpansive from $\mathbf E$ into $\mathbf E'$ iff it is order-preserving from $\mathbf P_{d}$ into $\mathbf P_{d'}$.
 Depending on the value of their radius $v\in \mathcal H$, a  metric space  over $\mathcal{H}$ has four types of balls: singletons, corresponding to  $v=0$, the full space, corresponding to   $v=1$, the principal  final segments,  $\uparrow x:= \{y\in E: x\leq y\}$, corresponding to balls  $B(x, +)$, and principal initial segments, $\downarrow  x:= \{y\in E : y\leq x\}$, corresponding to  balls   $B(x, -)$.
 The set  $\mathcal H$ can be equipped with the distance $d_\mathcal H$ given  by means of the formula (\ref{eq:distance}). The corresponding poset is the four element lattice $\{-, 0, 1, +\}$ with $0< -, +< 1$. The  retracts of powers of this lattice are all complete lattices. 
 
 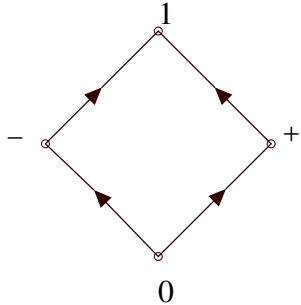
\begin{figure}[H]
 	\centering
 	\definecolor{ttqqqq}{rgb}{0.2,0.,0.}
 	\begin{tikzpicture}[line cap=round,line join=round,>=triangle 45,x=0.75cm,y=0.75cm]
 	\draw [->,line width=0.5pt,color=ttqqqq] (3.,-1.) -- (4.2,0.2);
 	\draw [line width=0.5pt,color=ttqqqq] (4.2,0.2)-- (5.,1.);
 	\draw [->,line width=0.5pt,color=ttqqqq] (5.,1.) -- (4,2);
 	\draw [line width=0.5pt,color=ttqqqq] (4,2)-- (3.,3.);
 	\draw [->,line width=0.5pt,color=ttqqqq] (1.,1.) -- (2.,2.);
 	\draw [line width=0.5pt,color=ttqqqq] (2,2)-- (3.,3.);
 	\draw [->,line width=0.5pt,color=ttqqqq] (3.,-1.) -- (1.86,0.18);
 	\draw [line width=0.5pt,color=ttqqqq] (1.86,0.18)-- (1.,1.);
 	\draw (2.8,-1.22) node[anchor=north west] {0};
 	\draw (2.8,3.7) node[anchor=north west] {1};
 	\draw (5,1.5) node[anchor=north west] {$+$};
 	\draw (0.1,1.44) node[anchor=north west] {$-$};
 	\begin{scriptsize}
 	\draw [color=ttqqqq] (3.,-1.) circle (1.5pt);
 	\draw [color=ttqqqq] (5.,1.) circle (1.5pt);
 	\draw [color=ttqqqq] (1.,1.) circle (1.5pt);
 	\draw [color=ttqqqq] (3.,3.) circle (1.5pt);
 	\end{scriptsize}
 	\end{tikzpicture}
 	\caption{The ordered monoid $\mathcal{H}$.}
 \end{figure}
 The fact, due to Birkhoff,  that every poset embeds into a power of the two-element chain $\underline{\mathbf {2}}:=\lbrace 0,1\rbrace $ is the translation in terms of posets of Theorem \ref{thm:embedding}.
 \vskip2mm
 \subsection{The fence distance on posets}\label{subsection:fencedistance}
 If we view an ordered set as a directed graph,  we may associate its  zigzag distance. 
 In this case, the reflexive oriented zizags defined at the begining of  Subsubsection \ref{subsection:the zigzag distance} reduce to \emph{fences}. Indeed, 
 a \emph{fence} is a poset whose comparability graph is a path. For example, a two-element chain is a fence. Each larger fence has two orientations, for example on the three vertices path, these orientations yield the $\bigvee$ and the $\bigwedge$.  The  $\bigvee$ is the  $3$-element poset consisting of $0, +, -$ with $0< +, -$ and $+$ incomparable to $-$. The $\bigwedge $ is its dual.  More generally, for each integer $n$, there are two fences of length $n$: the \emph{up-}  and the \emph{down-fence}. The first one starts with $x_0<x_1>...$, the second with $x_0>x_1<..$. For $n:=2$ one  get $\bigwedge $ and $\bigvee$ respectively.

 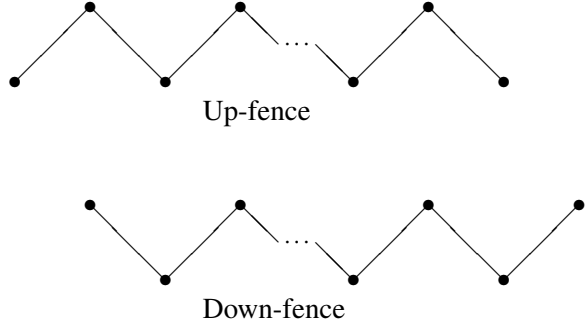
\begin{figure}[H]
 	\unitlength=1cm
 	\begin{center}
 		
 		\begin{picture}(8,2)
 		
 		\put(5.5,1){\circle*{.15}}
 		\put(6.5,0){\circle*{.15}}
 		\put(0,0){\circle*{.15}}
 		\put(0,0){\line(1,1){1}}
 		
 		\put(1,1){\circle*{.15}}
 		\put(2,0){\line(-1,1){1}}
 		
 		\put(2,0){\circle*{.15}}
 		\put(2,0){\line(1,1){1}}
 		
 		\put(3,1){\circle*{.15}}
 		\put(3,1){\line(1,-1){0.5}}
 		\put(3.625,0.5){\circle*{.05}}
 		\put(3.75,0.5){\circle*{.05}}
 		\put(3.875,0.5){\circle*{.05}}
 		\put(4,0.5){\line(1,-1){0.5}}
 		
 		\put(4.5,0){\circle*{.15}}
 		\put(4.5,0){\line(1,1){1}}
 		
 		\put(6.5,0){\line(-1,1){1}}
 		 		
 		\put(2.5,-0.5){\small{Up-fence}}
 		\end{picture}
 		\vskip6mm
 		
 		\begin{picture}(8,2)
 		
 		\put(5.5,1){\circle*{.15}}
 		\put(6.5,0){\circle*{.15}}

 		\put(1,1){\circle*{.15}}
 		\put(2,0){\line(-1,1){1}}
 		
 		\put(2,0){\circle*{.15}}
 		\put(2,0){\line(1,1){1}}
 		
 		\put(3,1){\circle*{.15}}
 		\put(3,1){\line(1,-1){0.5}}
 		\put(3.625,0.5){\circle*{.05}}
 		\put(3.75,0.5){\circle*{.05}}
 		\put(3.875,0.5){\circle*{.05}}
 		\put(4,0.5){\line(1,-1){0.5}}
 		
 		\put(4.5,0){\circle*{.15}}
 		\put(4.5,0){\line(1,1){1}}
 		
 		\put(6.5,0){\line(-1,1){1}}
 		\put(6.5,0){\line(1,1){1}}
 		\put(7.5,1){\circle*{.15}}

 		\put(2.5,-0.5){\small{Down-fence}}
 		\end{picture}
 		\vskip1cm
 		
 		\caption{ Up-fence and Down-fence.}
 	\end{center}
 \end{figure}

 Let  $\mathbf P:=(E, \leq)$ be  a poset. If two vertices $x$ and $y$ are  connected in the comparability graph of $\mathbf P$, one may map some fence  into $\mathbf P$ by an order-preserving map sending the extremities of the fence onto $x$ and $y$. One can then define the distance $d_{\mathbf P}(x,y)$ between  $x$ and $y$ as the pair $(n,m)$ of integers such that $n$, resp. $m$,  is the  shortest length of an up-fence, resp. a down fence,  whose extremities can be mapped onto  $x$  and  $y$. If $x$ and $y$ are not connected in the comparability graph of $\mathbf P$, one sets  $d_{\mathbf P}(x,y)=+\infty$. For example, if $x<y$ then $d_{\mathbf P}(x,y)= (1, 2)$. This distance is defined  in \cite{nevermann-rival}, an alternative definition is in \cite{JaMiPo}.    
 
 Let $\mathcal H:= \{(n,m)\in (\N\setminus \{0\})^2: \vert n-m\vert\leq 1\}\cup \{(0,0), +\infty\} \setminus \{(1,1)\}$,  the pairs being ordered componentwise and  $+\infty$ being at the top. The involution transforms $(n,m)$ into $(m,n)$. The sum $(n,m)\oplus (n',m')$ is $(n\oplus n', m\oplus m')$ where $n\oplus n'$ is $n+n'-1$ if $n$ is odd and $n+n'$ otherwise. With this operation,  $\mathcal H$ forms a Heyting algebra. If $\mathbf P:= (E, \leq)$ is a poset  then $d_{\mathbf P}: E\times E\rightarrow \mathcal H$ is a distance over $\mathcal H$.  According to Theorem \ref{thm:embedding},  this Heyting algebra has a metric structure $\mathbf H$ and every metric space over $\mathcal H$ embeds isometrically into a power of $\mathbf H$. It turns out that  $\mathbf H$ is the metric space associated to a poset $\mathbf P_{\mathcal H}$ (to see it, set  $x\leq y$ if $x=y$ or $1$ is the first component of $d_\mathcal H(x,y)$). This poset is represented below. Hence every poset embeds isometrically into a power of $\mathbf P_{\mathcal H}$. From the study of hyperconvexity in Section \ref{section:hyperconvex}, this poset embeds isometrically  into a product of fences, hence  every poset embeds isometrically  into a retract of fences (\cite {Qu1}). For more, see Nevermann-Rival, 1985 and Jawhari-al 1986. 
 
 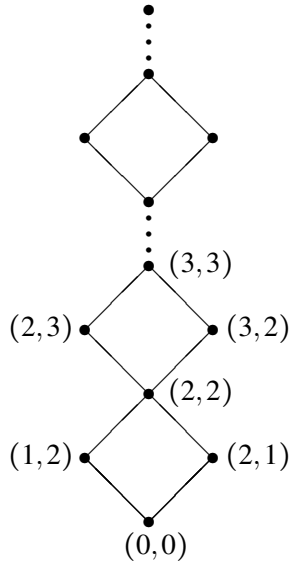
\begin{figure}[H]
 	\begin{center}
 		\unitlength=0.85cm
 		\begin{picture}(1,8)
 		\put(-1,1){\circle*{.15}}
 		\put(1,1){\circle*{.15}}
 		\put(0,0){\circle*{.15}}
 		\put(0,0){\line(1,1){1}}
 		\put(0,0){\line(-1,1){1}}
 		\put(-1,1){\line(1,1){1}}
 		\put(1,1){\line(-1,1){1}}
 		\put(0,2){\circle*{.15}}
 		\put(0,2){\line(1,1){1}}
 		\put(0,2){\line(-1,1){1}}
 		
 		\put(1,3){\circle*{.15}}
 		\put(1,3){\line(-1,1){1}}
 		\put(-1,3){\line(1,1){1}}
 		
 		\put(-1,3){\circle*{.15}}
 		
 		\put(0,4.25){\circle*{.08}}
 		\put(0,4.5){\circle*{.08}}
 		\put(0,4.75){\circle*{.08}}
 		
 		\put(0,4){\circle*{.15}}
 		\put(0,0){\line(1,1){1}}
 		\put(0,0){\line(-1,1){1}}
 		\put(0,5){\circle*{.15}}
 		\put(1,6){\circle*{.15}}
 		\put(-1,6){\circle*{.15}}
 		\put(0,7){\circle*{.15}}
 		\put(0,8){\circle*{.15}}
 		\put(0,5){\line(-1,1){1}}
 		\put(0,5){\line(1,1){1}}
 		
 		\put(1,6){\line(-1,1){1}}
 		\put(-1,6){\line(1,1){1}}
 		\put(0,7.25){\circle*{.08}}
 		\put(0,7.5){\circle*{.08}}
 		\put(0,7.75){\circle*{.08}}
 		\put(-0.4,-0.5){\small{$(0,0)$}}
 		
 		\put(1.2,0.9){\small{$(2,1)$}}
 		\put(-2.2,0.9){\small{$(1,2)$}}
 		\put(0.3,1.9){\small{$(2,2)$}}
 		
 		\put(1.2,2.9){\small{$(3,2)$}}
 		\put(-2.2,2.9){\small{$(2,3)$}}
 		\put(0.3,3.9){\small{$(3,3)$}}
 		\end{picture}\vskip0.5cm

 		\caption{The ordered monoid $\mathcal{H}$.}
 	\end{center}
 \end{figure}

 \begin{figure}[H]
 	\unitlength=1cm
 	\begin{center}
 		\begin{tikzpicture}
 		\put(0,0){\scriptsize{$(0,0)$}}
 		\put(1,0){\circle*{.10}}
 		\put(1.35,1.25){\scriptsize{$(1,2)$}}
 		\put(1.35,-1.4){\scriptsize{$(2,1)$}}
 		\put(1.75,1){\circle*{.10}}
 		\put(1.75,-1){\circle*{.10}}
 		\draw (1,0) to (1.75,1);
 		\draw (1,0) to (1.75,-1); 
 		\draw (2.5,0) to (1.75,1);
 		\draw (2.5,0) to (1.75,-1);
 		\put(2.10,0.56){\scriptsize{$(2,2)$}}
 		\put(2.5,0){\circle*{.10}}
 		\put(2.85,1.25){\scriptsize{$(3,2)$}}
 		\put(2.85,-1.4){\scriptsize{$(2,3)$}}
 		\put(3.25,1){\circle*{.10}}
 		\put(3.25,-1){\circle*{.10}}
 		\draw (2.5,0) to (3.25,1);
 		\draw (2.5,0) to (3.25,-1); 
 		\draw (4,0) to (3.25,1);
 		\draw (4,0) to (3.25,-1);
 		
 		\put(3.60,0.56){\scriptsize{$(3,3)$}}
 		\put(4,0){\circle*{.10}}
 		\put(4.35,1.25){\scriptsize{$(3,4)$}}
 		\put(4.35,-1.4){\scriptsize{$(4,3)$}}
 		\put(4.75,1){\circle*{.10}}
 		\put(4.75,-1){\circle*{.10}}
 		\draw (4,0) to (4.75,1);
 		\draw (4,0) to (4.75,-1); 
 		\draw (5.5,0) to (4.75,1);
 		\draw (5.5,0) to (4.75,-1);
 		\put(5.25,0.35){\scriptsize{$(4,4)$}}
 		\put(5.5,0){\circle*{.10}}
 		\put(5.7,0){\circle*{.08}}
 		\put(5.9,0){\circle*{.08}}
 		\put(6.1,0){\circle*{.08}}

 		\put(5.3,-0.9){\scriptsize{$(2n,2n)$}}
 		\draw[->] (5.9, -0.65) to (6.3,-0.1);
 		\put(6.3,0){\circle*{.10}}
 		\put(6.1,1.25){\scriptsize{$(2n+1,2n)$}}
 		\put(6.1,-1.4){\scriptsize{$(2n,2n+1)$}}
 		\put(7.05,1){\circle*{.10}}
 		\put(7.05,-1){\circle*{.10}}
 		\draw (6.3,0) to (7.05,1);
 		\draw (6.3,0) to (7.05,-1); 
 		\draw (7.8,0) to (7.05,1);
 		\draw (7.8,0) to (7.05,-1);
 		
 		\put(7,2){\scriptsize{$(2n+1,2n+1)$}}
 		\put(7.8,0){\circle*{.10}}
 		\put(8.2,1.25){\scriptsize{$(2n+1,2n+2)$}}
 		\put(8.2,-1.4){\scriptsize{$(2n+2,2n+1)$}}
 		\put(8.55,1){\circle*{.10}}
 		\put(8.55,-1){\circle*{.10}}
 		\draw[->] (8,1.9) to (7.8,0.1);
 		\draw (7.8,0) to (8.55,1);
 		\draw (7.8,0) to (8.55,-1); 
 		\draw (9.3,0) to (8.55,1);
 		\draw (9.3,0) to (8.55,-1);
 		\put(7.8,0){\circle*{.10}}
 		\put(10.2,0){\circle*{.20}}
 		\put(10.38,-0.07){\scriptsize{$\infty$}}
 		
 		\end{tikzpicture}
 	\end{center}\vskip3mm
 	\caption{The poset $P_{\mathcal{H}}$.}
 \end{figure}
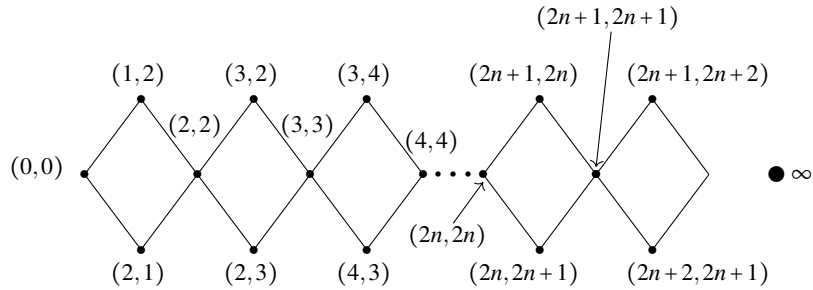
 \subsection{Transitions systems} 
 The zigzag distance  is a special case of distance  defined on transition systems. Indeed, it $\M$ is a transition system on an alphabet $\Lambda$, we may define the distance $d_{\M}(x,y)$ from a state $x$ to  a state $y$ as the language accepted by the automaton  $\mathcal A_{x,y}:=(\M, \{x\}, \{y\})$ whose initial state is $x$ and final state $y$. Once the alphabet is equipped with an involution,  this distance takes value in  a Heyting algebra in which the neutral element is no longer the least element  and satisfies conditions (\ref{defdistance}) of our introduction. As it turn out, if we  view   a reflexive graph as a transition system of a special form, the zigzag distance  is the distance on that transition system. Here are the details. 
 
 Let $\Lambda $ be a set. Consider  $\Lambda$ as an \textit{alphabet} whose
 members are \textit{letters} and extend to $\Lambda$ what we did for the two-letter alphabet.  We write a word $\alpha $ with a
 mere juxtaposition of its letters as $\alpha
 =a_{0}\dots a_{n-1}$ where $a_{i}$
 are letters from $\Lambda$ for 0 $\leq i\leq n-1.$ The integer $n$ is the \textit{%
 	length} of the word $\alpha $ and we denote it $\left| \alpha \right| $.
 Hence we identify letters with words of length 1. We denote by $\Box $ the
 empty word, which is the unique  word of length zero.
 The \emph{concatenation} of two word $\alpha:= a_{0}\cdots a_{n-1}$ and $\beta:=b_{0}\cdots b_{m-1}$ is the word $\alpha \beta:=a_{0}\cdots a_{n-1}b_{0}\cdots b_{m-1}$. We denote by $\Lambda^{\ast }$ the set of all words on the alphabet $\Lambda$. Once equipped with the
 concatenation of words, $\Lambda^{\ast }$ is a monoid, whose neutral element is the empty word, in fact $\Lambda^{\ast}$ is the \textit{free
 	monoid} on $\Lambda$.   A \emph{language} is any subset $X$ of $\Lambda^{\ast}$. We denote by $\powerset(\Lambda^{\ast})$ the set  of languages. We will use capital letters for languages. If $X, Y \in \powerset (\Lambda^{\ast})$ the \emph{concatenation} of $X$ and $Y$ is the set  $XY:= \{\alpha\beta: \alpha\in X, \beta\in Y\}$ (and we will use  $Xy$ and $xY$instead of $X\{y\}$ and $\{x\}Y$).  This  operation  extends the concatenation operation on $\Lambda^{\ast}$; with it, the set $\powerset (\Lambda^{\ast})$ is a monoid whose  neutral element is the set $\{ \Box \}$.

 Ordered by inclusion, this is a (join) lattice ordered monoid. Indeed, concatenation distributes over arbitrary union, namely:
 \begin{center}
 	$( \underset{i\in I}{\bigcup }X_{i})Y=\underset{i\in I}{ \bigcup }X_{i} Y.$
 \end{center}
 But concatenation does not distribute over intersection (for a simple example, let $\Lambda:= \{a,b,c\}$, $I:=\{1, 2\}$, $X_1:=\{ab\}$, $X_2:=\{a\}$, $Y:=\{c, bc\}$, then $\emptyset =(X_1\cap X_2) Y\not = X_1Y\cap X_2Y=\{abc\}$).
 Ordered by reverse of the inclusion, the monoid $\powerset ( \Lambda^{\ast})$ becomes a  Heyting algebra (while ordered by inclusion it is not) in the sense that it satisfies the distributivity condition (\ref{heyting}). If $-$ is an involution on $\Lambda$, it extends  to an involution on $\Lambda^*$, by setting $\overline \Box:= \Box$, and $\overline{\alpha}=\overline {a_{n-1}}\dots\overline{a_0}$ if $\alpha= a_{0}\dots a_{n-1}$.  This  involution reverses the concatenation of words. Extended  to $\powerset (\Lambda^{\ast})$ by setting $\overline X:= \{\overline {\alpha }: \alpha\in X\}$, it reverses the concatenation of languages and preserves the inclusion order on languages.
 The set  $\powerset (\Lambda^{\ast})$, with the concatenation of languages as  a monoid operation,  the reverse of the inclusion order  and the extension of the involution is  a Heyting algebra. But in this Heyting algebra,  the neutral element (namely $\{ \Box \}$), is not the least element.

 We suppose from now that the alphabet $\Lambda$ is ordered.

 We order  $\Lambda^{\ast }$ with the Higman ordering \cite{Hi} that is, if $\alpha $ and
 $\beta $ are two elements in $\Lambda^{\ast }$ such $\alpha: =a_{0}\cdots
 a_{n-1}$ and $\beta: =b_{0}\cdots  b_{m-1}$ then $\alpha \leq \beta$
 if there is an injective and increasing map $h$ from $\left\{
 0,...,n-1\right\} $ to $\left\{ 0,...,m-1\right\}$ such that for each $i$,  $
 0\leq i\leq n-1$, we have $a_{i}\leq b_{h\left( i\right) }$. Then
 $\Lambda^{\ast }$ is an ordered monoid with respect to the concatenation
 of words.  A \emph{final segment} of $\Lambda^{\ast}$ is any subset $F\subseteq \Lambda^{\ast}$ such that $\alpha \leq
 \beta,\alpha \in F$ implies $\beta \in F$.  Initial segments are defined dually.  

 Let $\F\left( \Lambda^{\ast }\right) $ be the collection of
 final segments of $\ \Lambda^{\ast }$.
 The set $\F\left( \Lambda^{\ast }\right)$ is stable w.r.t. the concatenation of languages:  if $ X,Y\in \F\left(\Lambda^{\ast }\right)$,  then $XY\in \F(\Lambda^{\ast})$ (indeed, if $u,v, w\in \Lambda^{\ast}$ with $uv\leq w$ then $w= u'v'$ with $u\leq u'$ and $v\leq v'$).   Clearly,  the neutral element
 is $\Lambda^{\ast }$.  The set $\F\left( \Lambda^{\ast }\right) $ ordered by
 inclusion is a complete lattice (the join is the union, the meet
 is the intersection).  Concatenation distributes over union.
 If we order $\F\left( \Lambda^{\ast }\right) $ by reverse of the
 inclusion, denoting $X\leq Y$ instead of $X\supseteq Y$, and we set ${\bf{1}}:= \Lambda^{\ast}$, we have the exact generalization obtained for a two-letter alphabet. 
 
 \begin{lemma}\label{fact:heyting}The set $\mathcal H_{\Lambda}:= (\mathbf {F}(\Lambda^*), \oplus, \supseteq,  \bf 1, -)$, where $\oplus$ denotes the concatenation of languages,  is a Heyting algebra and  ${\bf{1}}$ is the least element.
 \end{lemma}
 Contrarily to the case of the power set, in $\F( \Lambda^{\ast})$ concatenation  distributes over intersection:
 \begin{lemma}
 	$( \underset{i\in I}{\bigcap }X_{i})Y=\underset{i\in I}{ \bigcap }X_{i} Y$ for all final segments $X_i$ and $Y$ of $\Lambda^{\ast}$.
 \end{lemma}
 \begin{proof}
 	The inclusion $( \underset{i\in I}{\bigcap }X_{i})Y\subseteq \underset{i\in I}{ \bigcap }X_{i} Y$ is obvious. For the proof of the reverse inclusion, let $z\in \underset{i\in I}{ \bigcap }X_{i} Y.$ For every $i\in I$ there are $x_i\in X_{i}$ and $y_i\in Y$ such that $z=x_iy_i$. Let $y$ be the shortest  suffix of $z$ such that $y=y_{i_0}$ for some $i_0\in I$ and let $x\in \Lambda^{\ast}$ such that $z= xy$. We claim that $x\in \underset{i\in I}{ \bigcap }X_{i}$. Indeed, let $j\in I$. We have  $z= x_jy_j$ and $z=x_{i_0}y_{i_0}$. By minimality of $y_{i_0}$, we have $x_j\leq x_{i_0}=x$, hence $x\in X_j$ since $X_j$ is a final segment of $\Lambda^{\ast}$. This proves our claim. Since $z=xy$, $z\in  (\underset{i\in I}{\bigcap }X_{i})Y $,  as required.
 \end{proof}
 
 We refer to \cite{sakarovitch} for the terminology about transition systems.
 A \textit{transition system} on the \emph{alphabet} $\Lambda$ is a pair $\M:=(Q$,
 $T)$ where $T\subseteq Q\times \Lambda\times Q.$ The elements of $Q$ are
 called \textit{states} and those of $T$ \textit{transitions}. Let
 $\M:=\left(Q,T\right) $ and $\M^\prime:=\left( Q^{\prime
 },T^{\prime}\right) $ be two transition systems on the alphabet
 $\Lambda$. A map $f:Q\longrightarrow Q^{\prime }$ is a
 \textit{morphism} of transition systems if for every transition
 $(p,\alpha,q)\in T$, we have $(f(p),\alpha
 , f\left( q\right)) \in T^{\prime}$. When $f$ is bijective and $%
 f^{-1} $ is a morphism from $\M^{\prime }$ to
 $\M$, we say that $f$ is an \textit{isomorphism}.

 An \textit{automaton} $\mathcal A$ on the alphabet $\Lambda$ is given by a transition system
 $\M:=\left( Q,T\right)$ and two subsets $I,$ $F$ of $Q$ called the set
 of \textit{initial} and \textit{final states}. We denote the
 automaton  as a triple $\left(\M,I,F\right)$.  A \textit{path} in
 the automaton  $\mathcal{A}:=\left(\M,I,F\right)$ is a sequence
 $c:=\left( e_{i}\right) _{i<n}$ of consecutive
 transitions, that is  of transitions $e_{i}:=(q_{i},a _{i},q_{i+1})$.
 The  word $\alpha:=a_{0}\cdots a_{n-1}$ is the \textit{label} of the path, the
 state $q_{0}$ is its \textit{origin} and the state $q_{n}$ its \textit{end}.
 One agrees to define for each state $q$ in $Q$ a unique null path
 of length $0$ with origin and end $q$. Its label is the empty word
 $\Box$. A path is \textit{successful} if its origin is in $I$ and
 its end is in $F$. Finally,  a word $\alpha$ on the alphabet $\Lambda$ is
 \textit{accepted} by the automaton 
 $\mathcal{A}$ if it is the label of some successful path. The \textit{%
 	language accepted} by the automaton  $\mathcal{A}$,  denoted by
 $L_{\mathcal{A}}$,  is the set of all words accepted by
 $\mathcal{A}$. Let $\mathcal{A}:=\left(\M,I,F\right) $ and
 $\mathcal{A}^{\prime }:= \left(\M^{\prime },I^{\prime },F^{\prime
 }\right) $ be two automata. A {\it morphism} from $\mathcal{A}$ to $\mathcal{A}'$ is a map $f:Q\longrightarrow Q^{\prime }
 $ satisfying the two conditions:
 \begin{enumerate}
 	\item  $f$ is  morphism from $\M$ to $\M^{\prime }$;
 	\item  $f$ $(I)$ $\subseteq I^{\prime }$ and $f(F)\subseteq F^{\prime}$.
 \end{enumerate}
 If, moreover, $f$ is bijective, $f(I)=I^{\prime },f(F)=F^{\prime }$ and $%
 f^{-1}$ is also a  morphism from $\mathcal{A}^{\prime }$ to $%
 \mathcal{A}$, we say that $f$ is an  \textit{isomorphism} and that
 the two automata $\mathcal{A}$ and $\mathcal{A}^{\prime }$ are \textit{%
 	isomorphic}.

 To a metric space $\mathbf E:= \left( E,d\right) $ over
 $\mathcal{H}_{\Lambda}:= \F(\Lambda^{\ast })$, we may associate the transition system
 $\M:=\left( E,T\right) $ having $E$ as   set of states and
 $T:=\left\{ \left( x,a,y\right) :a\in d\left( x,y\right) \cap
 \Lambda\right\} $ as   set of transitions. Notice that such a
 transition system has the following properties: for all $x,y\in E$
 and every $a,b\in \Lambda$ with $b\geq a$:\\
 1) $\left( x,a,x\right) \in T$; \\
 2) $\left( x,a,y\right) \in T$ implies $\left(
 y,\overline{a},x\right) \in T$; \\
 3) $\left( x,a,y\right) \in T$ implies  $\left( x,b,y\right) \in
 T.$\\
 We say that a transition system satisfying these properties is \textit{%
 	reflexive} and \textit{involutive} (cf. \cite{Sa}, \cite{KP2}). Clearly \ if $%
 \M:=\left( Q,T\right) $ is such a transition system, the map
 $d_{\M}:Q\times Q\longrightarrow \mathcal{H}_{\Lambda}$, where $d_{\M}\left(
 x,y\right) $ is the language accepted by the automaton  $\left(
 \M,\left\{ x\right\},\left\{ y\right\} \right) $, is a distance.  We have the following:
 \begin{lemma}\label{lem:metric-transition}
 	Let $\mathbf E:= ( E,d)$ be a metric space over
 	$\mathcal{H}_{\Lambda}:= \F\left( \Lambda^{\ast }\right)$.  The following properties are
 	equivalent:
 	\begin{enumerate}
 		\item The map $d$ is of the form $d_{\M}$ for some
 		reflexive and involutive transition system $\M:=(E,T)$;
 		\item  For all $\alpha,\beta \in \Lambda^{\ast }$ and  $x$,
 		$y$ $\in E$, if $\alpha\beta \in d\left( x,y\right)$, then there
 		is some $z\in E$ such that $\alpha \in d\left( x,z\right)$ and
 		$\beta \in d\left( z,y\right)$.
 	\end{enumerate}
 \end{lemma}

 \begin{lemma}\label{fact:morphism}  Let $\M_{i}:=\left( Q_{i},T_{i}\right) \left( i=1,2\right) $
 	be two reflexive and involutive transition systems. A map $ f:Q_{1}\longrightarrow Q_{2}$ is a morphism from $\M_{1}$ to $\M_{2}$ if only if $f$ is a nonexpansive map  from $(Q_{1},d_{\M_{1}})$ to $(Q_{2},d_{\M_{2}}).$
 \end{lemma}

 From Lemma \ref{fact:morphism},   the category of reflexive and involutive transition
 systems with the morphisms defined above can be identified  to a subcategory of the
 category having as objects the metric spaces and morphisms the nonexpansive maps. %
 
 As with directed graphs,  Lemma \ref{lem:metric-transition} ensures that the various metric spaces mentioned in the introduction (injective, absolute retracts, etc.) come from  transition systems.  In particular, the distance $d_{\mathcal H_{\Lambda}}$  defined on $\mathcal H_{\Lambda}$ is the distance of some transition system,  say $\mathbf M_{\mathcal H_{\Lambda}}$. According to  Theorem \ref{thm:embedding}, every reflexive involutive transition systems embeds isometrically into some power of $\mathbf M_{\mathcal H_{\Lambda}}$. As in the case of graphs, this transition system is countably infinite (for more, see \cite{KP2,KPR, kabil-pouzet}).

 \section{A categorical approach of generalized metric spaces}
 Let $\mathcal C$ be  a category, with objects, say $\mathbf P$, $\mathbf Q$
 , ... and morphisms $f$, $g$,.... We say that the object $\mathbf P$ is a \emph{retract} of the object $\mathbf Q$ and we note $\mathbf P\vartriangleleft \mathbf Q$ if there are morphisms $f:\mathbf P\longrightarrow \mathbf Q$ and $g:\mathbf Q\longrightarrow \mathbf \mathbf P$ such that $g\circ f= \id_{\mathbf P}$, where $\id_{\mathbf P}$ is the identity map on $\mathbf P$. 
 
 Two examples:
 \begin{enumerate}
 	\item The objects of the category are the posets and the morphisms are the order-preserving maps (i.e. the maps $f$ such that $x\leq y$ implies $f(x)\leq f(y)$. 
 	\begin{figure}[H]
 		\unitlength=1cm
 		\begin{center}
 			\begin{tikzpicture}
 			\put(0.6,0.5){\scriptsize{$\mathbf P$}}
 			\put(1,0){\circle*{.15}}
 			\put(2,1){\circle*{.15}}
 			\put(3,0){\circle*{.15}}
 			\draw (1,0) to (2,1);
 			\draw (2,1) to (3,0); 
 			
 			\put(4.6,0.5){\scriptsize{$\mathbf Q$}}
 			\put(5,0){\circle*{.15}}
 			\put(6,1){\circle*{.15}}
 			\put(7,0){\circle*{.15}}
 			\put(8,1){\circle*{.15}}
 			\put(9,0){\circle*{.15}}
 			\put(10,1){\circle*{.15}}
 			\draw (5,0) to (6,1);
 			\draw (6,1) to (7,0); 
 			\draw (7,0) to (8,1);
 			\draw (8,1) to (9,0); 
 			\draw (9,0) to (10,1);
 			\draw (5,0) to (10,1); 	
 			\end{tikzpicture}
 		\end{center}
 		\caption{$\mathbf P$ is retract of $\mathbf Q$.}
 	\end{figure}
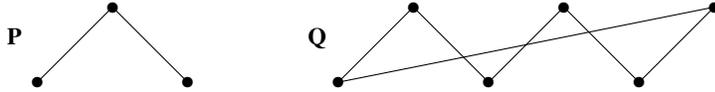
 	
 	\item The objects of the category are all reflexive graphs (which are the undirected graphs with a loop at every vertex, or, equivalently, the reflexive and symmetric binary relations)  and the morphisms are  all edge-preserving maps (note that an edge joining two vertices can be mapped on a loop).
 	
 	\begin{figure}[H]
 		\begin{center}
 			\definecolor{ttqqqq}{rgb}{0.2,0.,0.}
 			\begin{tikzpicture}[line cap=round,line join=round,>=triangle 45,x=1cm,y=1cm]
 			
 			\draw [line width=0.5pt] (0.,0.5)-- (0.,1.5);
 			\draw [line width=0.5pt] (1.,1.5)-- (1.,0.5);
 			\draw [line width=0.5pt] (0.,0.5)-- (1.,0.5);
 			\draw [line width=0.5pt] (1.,1.5)-- (0.,1.5);
 			
 			\begin{scriptsize}
 			\draw [color=ttqqqq] (0.,0.5) circle (3pt);
 			\draw [color=ttqqqq] (1.,1.5) circle (3pt);
 			\draw [color=ttqqqq] (1.,0.5) circle (3pt);
 			\draw [color=ttqqqq] (0.,1.5) circle (3pt);
 			
 			\draw (-0.5,1) node{$\mathbf G$};
 			\end{scriptsize}
 		
 			\end{tikzpicture}
 		
 			\;\;\;\;\;\;\;\;\;\;\;\;\;\;\;\;
 		
 			\definecolor{ttqqqq}{rgb}{0.2,0.,0.}
 			\begin{tikzpicture}[line cap=round,line join=round,>=triangle 45,x=1cm,y=1cm]
 			\draw [line width=0.5pt] (0.,0.)-- (0.,1.);
 			\draw [line width=0.5pt] (1.,1.)-- (1.,0.);
 			\draw [line width=0.5pt] (0.,0.)-- (1.,0.);
 			\draw [line width=0.5pt] (1.,1.)-- (0.,1.);
 			
 			\draw [line width=0.5pt] (0.5,1.5)-- (0.,1.);
 			\draw [line width=0.5pt] (0.5,1.5)-- (1.,1.);
 		
 			\draw [line width=0.5pt] (-0.5,0.5)-- (0.,0.);
 			\draw [line width=0.5pt] (-0.5,0.5)-- (0.,1.);
 			
 			\draw [line width=0.5pt] (0.5,-0.5)-- (1.,0.);
 			\draw [line width=0.5pt] (0.5,-0.5)-- (0.,0.);
 			
 			\draw [line width=0.5pt] (1.5,0.5)-- (1.,1.);
 			\draw [line width=0.5pt] (1.5,0.5)-- (1.,0.);
 			\begin{scriptsize}
 			\draw [color=ttqqqq] (0.,0.) circle (3pt);
 			\draw [color=ttqqqq] (1.,1.) circle (3pt);
 			\draw [color=ttqqqq] (1.,0.) circle (3pt);
 			\draw [color=ttqqqq] (0.,1.) circle (3pt);
 		
 			\draw [color=ttqqqq] (0.5,1.5) circle (3pt);
 			\draw [color=ttqqqq] (1.5,0.5) circle (3pt);
 			\draw [color=ttqqqq] (0.5,-0.5) circle (3pt);
 			\draw [color=ttqqqq] (-0.5,0.5) circle (3pt);
 			\draw (-1,0.5) node{$\mathbf K$};
 			\end{scriptsize}
 			\end{tikzpicture}
 			\caption{$\mathbf G$ is retract of $\mathbf K$.}
 		\end{center}
 	\end{figure}
 \end{enumerate}
 The central question about retraction  is to decide, for two given objects $\mathbf P$ and $\mathbf Q$, whether $\mathbf P$ is a retract of $\mathbf Q$ or not. A related question is to decide whether a given morphism $f: \mathbf P \longrightarrow \mathbf Q$ has a companion $g: \mathbf Q \longrightarrow \mathbf P$ such that $g \circ f= \id_{\mathbf P}$; if this is the case, $f$ is said to be \emph{coretraction}  and its companion  is a \emph{retraction}.
 In fact, these questions are still largely unsolved, even for very simple categories like those of posets and graphs. Neverthlesss a fruitful approach of a solution is this:\\
 Identify a general property, say (p), that the coretractions  enjoy in the category considered; for example, in the above category of posets each coretraction is an order-embeding (that is a map $f$ such that $x\leq y$ is equivalent to $f(x)\leq f(y)).$ Now looking at (p) as an \textit{approximation} of the coretractions, then characterize the objects $\mathbf P$ for which this approximation is accurate, that is  for which every morphism of source $\mathbf P$ and with property (p) is a coretraction. These $\mathbf P$ are commonly called the \textit{absolute retracts} (briefly $AR$); (a terminology not perfectly adequate, since these objects depend upon the approximation, but commonly used in the field), we will rather say $AR$ with respect to the approximation (p). 
 In  the category of metric spaces with nonexpansive mappings we are lead  to the following definitions :
 \subsection{Retraction, coretraction, absolute retract}  
 Let $\mathbf E$ and $\mathbf F$ be two metric spaces over a Heyting algebra $\mathcal H$. The space $\mathbf E$ is a
 {\it retract} of $\mathbf F$, in symbols $\mathbf E\lhd \mathbf F$ if there are nonexpansive 
 maps
 $f: \mathbf E \to \mathbf F$ and $g : \mathbf F \to \mathbf E$ such that $g \circ f = \id_{\mathbf E}$. If this is the case,  $f$ is said to be
 {\it coretraction}  and $g$ a {\it retraction}. If $\mathbf E$ is a subspace of $\mathbf F$,
 then $\mathbf E$ is a retract of $\mathbf F$ if there is a nonexpansive map from
 $\mathbf F$ to $\mathbf E$ such that $g(x) = x$ for all $x\in E$, where $E$ is the domain of $\mathbf E$. We can easily see that
 every coretraction  is an isometry. A metric space is an
 {\it absolute retract} if it is a retract of every isometric extension.
 \subsubsection{Injectivity  and extension property} A metric space $\mathbf E$ is said to be {\it injective}  if for all spaces $\mathbf F$ and 
 $\mathbf E'$, each nonexpansive mapping $f : \mathbf F \to \mathbf E$, and every isometry
 $g : \mathbf F \to \mathbf E'$ there is a nonexpansive mapping $ h : \mathbf E' \to \mathbf E$ such that
 $ h\circ g =f$. 
 
 A metric space $\mathbf E$ has the \emph{one-point extension property} if for every space $\mathbf E':= (E', d')$ and every subset $F$ of $E'$,  every nonexpansive map $f: \mathbf E'_{\restriction F}\rightarrow \mathbf E$ extends to some $x'\in E'\setminus F$  (if any) to a nonexpansive map from $\mathbf E'_{\restriction F \cup \{x'\}}$ into $\mathbf E$. 
 
 Using Zorn's lemma one has immediately:
 
 \begin{lemma} A metric space $\mathbf E:= (E, d)$ over $\mathcal H$ is injective iff it has the one-point extension property. 
 \end{lemma}	
 
 \begin{proof}Trivially, injectivity implies the one-point extension property. For the converse, let $\mathbf E':= (E' d')$, $F\subseteq E'$ and $ f : F \to E$ be a
 	nonexpansive map from $\mathbf E'_{\restriction F}$ into $E$. Consider the collection of nonexpansive  maps
 	$f': F'\to E$ which extend $f$. This collection of maps is
 	inductive. From Zorn's lemma, it has a maximal element $g$. The domain $F''$ of $g$ is $E'$, otherwise, pick $x\in E'\setminus E''$; since $\mathbf E$ has the one-point extension, $g$ extends to $x$, a contradiction. 
 \end{proof}
 
 As it will become apparent in Theorem \ref{thm: caracterisation-hyperconvexity}, we may replace "for some $x'$" by "every $x'$" in the definition above.

 \subsubsection{Hyperconvexity}  We say that a space $\mathbf E$ is {\it hyperconvex} if the
 intersection of every family
 of balls $\left(B_{\mathbf E} (x_i,r_i)\right)_{i\in I}$ is non-empty whenever
 $d(x_i,x_j) \leq r_i \oplus \overline {r_j}$ for all $i, j \in I$.

 Hyperconvexity is equivalent to the conjunction of the following 
 conditions:\\
 1) {\it Convexity} : for all $x,y \in E$ and $p,q \in \mathcal H$ such
 that $d(x,y) \leq p \oplus q$ there is $z \in E$ such that $d(x,z)\leq p$
 and $d(z,y)\leq q$.\\
 2) The {\it 2-Helly property},  also called the
 {\it 2-ball intersection property} : The intersection of every set
 (or, equivalently, every family) of balls is non-empty provided
 that their pairwise intersections are all non-empty.

 \subsection{A description of hyperconvex metric spaces}\label{section:hyperconvex}

 As it is easy to see, the collection of hyperconvex spaces over a Heyting algebra is stable under (non-empty) products and retracts. Thus, in the terminology of Duffus and Rival \cite{DuRi}, it forms a variety. 
 A less trivial property is this: 

 \begin{theorem}\cite{JaMiPo}\label{thm: space values-hyperconvex}
 	The metric space $\mathbf H:=(\mathcal H , d_{\mathcal H})$ is hyperconvex.
 \end{theorem}
 
 \begin{proof} We just give the idea, we  defer the reader to \cite{JaMiPo} for details.
 	
 	One shows first that $\mathbf H$ is convex. Indeed, let $x,y \in \mathbf H$ and $p,q \in \mathcal H$ such
 	that $d_{\mathcal H}(x,y) \leq p \oplus q$. Set $z:= (x \oplus p) \vee (y \oplus \overline q)$  and check that   $d_{\mathcal H}(x,z)\leq p$
 	and $d_{\mathcal H}(z,y)\leq q$.
 	
 	Next, one shows that balls in $\mathbf H$ are intervals of $\mathcal H$. More precisely, any ball $B_{\mathbf H}(x, r)$ of $\mathbf H$ is the closed interval $[q, p]:=\{y\in \mathcal H:  q \leq y \leq r\}$ where $q:= \bigwedge B_{\mathbf H}(x, r)$ and $p:= \bigvee B_{\mathbf H}(x, r)$.  
 	
 	Finally, to conclude,  observe that the closed intervals of  a complete lattice have the $2$-Helly property. 
 \end{proof}

 \vskip2mm

 We recall the notions of metric forms. 
 
 Let $\mathbf E:= (E,d)$ be a metric space over a Heyting algebra $\mathcal H$. A \emph {weak metric form}  is every map $
 f:E\longrightarrow \mathcal H$ satisfying 
 
 \begin{equation}\label{eq:weakmetricform}
 	d(x,y) \leq f(x) \oplus \overline {f(y)} 
 \end{equation}
 for all $
 x,y\in E.$    
 
 This is a \emph{metric form}  if it is a weak metric form satisfying:  
 
 \begin{equation}\label{eq:metricform}
 	f(x) \leq  d(x,y) \oplus f(y)
 \end{equation}
 for all $x,y\in E.$

 We denote by $\mathcal C(\mathbf E)$, resp. $\mathcal L(\mathbf E)$,  the set of  weak metric form, resp. metric forms over $\mathbf E$. We equip these sets by the distance induced from the sup-distance on the power $\mathbf H^E$.

 \begin{lemma}\label{lem:metric form} 
 	Let $\mathbf E:= (E,d)$ be a metric space over $\mathcal H$, and $f:E\to \mathcal H$. The 
 	following
 	properties are equivalent:
 	
 	\begin{enumerate}[{(i)}]
 		
 		\item $f$ is a metric form; 
 		\item $f$ satisfies 
 		\begin{equation}\label{eq:metricform2}
 			d_{\mathcal H}(d(x,y),f(x))\leq f(y) 
 		\end{equation}
 		for all $x, y\in E$; 
 		\item In the product space ${\mathbf H}^{E} $ equipped with the "sup"
 		distance, $d(\bar {\delta}(y),f) = f(y)$ for all $y\in E$; 
 		\item There is some isometric extension $\mathbf E':= (E',d')$ of $\mathbf E$ and $u\in E'$
 		such that $f(y) =d'(y,u)$ for all $y\in E$.
 		
 	\end{enumerate}
 \end{lemma}
 \begin{proof}
 	$(i) \Rightarrow (ii)$ According to the definition of the distance $d_{\mathcal H}$, conditions
 	(\ref {eq:weakmetricform}) and (\ref {eq:metricform}) amount  to $d_{\mathcal H}\left(d(x,y),f(x)\right) \leq f(y)$, that is  condition (\ref {eq:metricform2}).
 	
 	$(ii)\Rightarrow (iii)$ 
 	
 	According to formula (\ref{eq:sup-distance}):
 	
 	$$d(\bar \delta (y),f) =
 	\bigvee _{x\in E}d_{\mathcal H}\left(d(x,y),f(x)\right) \leq f(y).$$
 	Now, taking $x=y$, we get  $d(x,y)=0$, and
 	$d_{\mathcal H}\left(0,f(y)\right) = f(y)$, thus the supremum in the inequality
 	above is $f(y)$.
 	
 	$(iii) \Rightarrow (iv)$ Since $\bar \delta$ is an isometric embedding 
 	from $\mathbf E$ into ${\mathbf H}^{E} $, it suffices to take $\mathbf E' := {\mathbf
 		H}^{E}$ and
 	$u := f$.
 	
 	$(iv) \Rightarrow (i)$ Obvious from the triangular inequality.
 \end{proof}
 
 \begin{corollary}\label{cor:delta}The image of  $\bar \delta$ is included into $\mathcal L(\mathbf E)$,  hence, $\bar \delta$  is an isometry of $\mathbf E$ into $\mathcal L(\mathbf E)$.  
 \end{corollary}
 \begin{proof}
 	Let $u\in E$. We check that $\bar \delta(u)$ is a weak metric form  for every $u\in E$. For that we show that inequality (\ref{eq:metricform2}) holds with $f:= \bar\delta(u)$. Indeed, we have
 	$d_{\mathcal H}(d(x,y),\bar\delta(u)(x))=d_{\mathcal H}(d(x,y), d(x, u)) \leq d(y, u):= \bar\delta(u)(y)$. 
 \end{proof}
 
 We recall Lemma II-4.4 of \cite{JaMiPo}. 
 
 \begin{lemma}\label{lemmakey}Let $\mathbf E:=(E, d)$ be a metric space over $\mathcal H$. 
 	For every weak metric form $f$,  the map $f_M: E\rightarrow \mathcal{H}$ defined by $f_M(x):= \bigwedge\{d(x,y)\oplus  f(y): y\in E\}$ is the largest metric form below $f$ and  $\bigcap \{B(x, f(x)): x\in E\}=\bigcap \{B(x, f_M(x)): x\in E\}$. Furthermore, the map $f \mapsto f_M$ is a retraction from ${\mathcal C}(\mathbf E)$ onto
 	${\mathcal L}(\mathbf E)$.
 \end{lemma}
 \begin{proof}
 	The verification is routine (the difficulty was to discover the formulation). 
 	
 	One proves first that if $g\in \mathcal L(\mathbf E)$ and $g\leq f$ then $g\leq  f_M$.
 	Indeed, since $g$ is a metric form then,  for every $x,y\in E$,  one  has $g(x)\leq d(x,y)\oplus g(y)$ and since $g\leq f$, one has $g(y)\leq f(y)$, thus $g(x)\leq d(x,y)\oplus f(y)$ hence $g(x)\leq \bigwedge\{d(x,y)\oplus  f(y): y\in E\}=:  f_M(x)$. 
 	
 	Next, one proves that $f_M$  is 
 	a metric form,  that is $d(x,y) \leq f_{M}(x) \oplus \overline {f_M(y)}$ and $f_M(x)\leq d(x,y)\oplus f_M(y)$ for all $x,y \in E$. The right hand side of the   first inequality $ f_M(x)\oplus {\overline {f_M(y)}}$ is equal to $\bigwedge\{d(x,z)\oplus  f(z): z\in E\} \oplus \bigwedge\{\overline {f(t)}\oplus d(t,y): t\in E\}$. Using the distributivity condition on $\mathcal H$, this yields $\bigwedge \{d(x,z)\oplus  f(z)\oplus \overline {f(t)}\oplus d(t,y): z,t\in E\}$.  From the triangular inequality and the fact that $f(z)\oplus \overline f(t)\geq d(z,t)$ we get $d(x,z)\oplus  f(z)\oplus \overline {f(t)}\oplus d(t,y)\geq d(x,y)$, hence$ f_M(x)\oplus {\overline {f_M(y)}}\geq d(x,y)$. For the second inequality, we have $d(x,z)\oplus f(z)\leq d(x,y)\oplus d(y,z) \oplus f(z)$ for all $z\in E$, hence $f_M(x):= \bigwedge\{d(x,z)\oplus  f(z): z\in E\}\leq \bigwedge\{d(x,y)\oplus d(y,z) \oplus f(z): z\in E\}=d(x,y)\oplus \{\bigwedge d(y,z) \oplus f(z): z\in E\}=:f_M(z)$.
 	
 	From these two fact follows that $f_M$ is the largest metric form  below $f$
 	
 	For the equality of the intersections of ball, note that the inclusion $\bigcap \{B(x, f_M(x)): x\in E\}\subseteq \bigcap \{B(x, f(x)): x\in E\}$ follows immediately from the fact that $f_M\leq f$.
 	For the reverse inclusion, pick $t\in \bigcap \{B(x, f(x)): x\in E\}$ that is $\overline {\delta} (t)(x)=d(x,t)\leq f(x)$  for every $x\in E$ or equivalently $\overline {\delta}(t)\leq f$. Since $\overline \delta(t)$ is a metric form  and $f_M$ is the largest metric form  below $f$, we have  $\overline \delta(t)\leq f_M$ amounting to $t\in \bigcap \{B(x, f_M(x)): x\in E\}$.
 	
 	Finally, one  checks that the map $f \mapsto f_M$ is a retraction  from ${\mathcal C}(\mathbf E)$ onto
 	${\mathcal L}(\mathbf E)$. 
 	
 	Since $f_M$ is the largest metric form below $f$, this map fixes ${\mathcal L}(\mathbf E)$ pointwise.  To conclude, it suffices to prove that this map is nonexpansive that is $d(f_M, g_M)\leq d(f,g)$ for all $f,g\in {\mathcal C}(\mathbf E)$. Let $f,g\in {\mathcal C}(\mathbf E)$. By definition of the distance on ${\mathcal C}(\mathbf E)$, we have  $f(y)\leq g(y)\oplus \overline {d(f,g)}$ hence $d(x,y)\oplus f(y)\leq d(x,y)\oplus g(y) \oplus \overline {d(f,g)}$ for all $x, y\in E$. 
 	This yields $f_M(x):= \bigwedge \{d(x,y)\oplus f(y) : y\in E\}\leq \bigwedge\{d(x,y)\oplus g(y) \oplus \overline {d(f,g)}: y\in E\}= \bigwedge\{d(x,y)\oplus g(y):y\in E\} \oplus \overline {d(f,g)}=:g_M(x)\oplus \overline {d(f,g)}$, that is $f_M(x)\leq g_M(x)\oplus \overline {d(f,g)}$. The same argument shows that $g_M(x)\leq f_M(x)\oplus d(f, g)$. Consequently, $d_\mathcal {H}(f_M(x), g_M(x))\leq d(f,g)$. Since this holds  for every $x\in E$, $d(f_M, g_M)\leq d(f,g)$ as required. The proof of the lemma is then complete. 
 \end{proof}

 Lemma \ref{lemmakey} was obtained independently by Kat\v{e}tov \cite{katetov}. 
 It plays a key role in the description of hyperconvex spaces, of injective envelopes and of hole-preserving maps.

 We  obtain below the following test of hyperconvexity. 
 \begin{proposition}\label{prop: test hyperconvexity}
 	Let $\mathbf E:= (E,d)$ be a metric space over a Heyting algebra $\mathcal H$. The
 	following properties are equivalent:
 	\begin{enumerate}[{(i)}]
 		
 		\item $\mathbf E$ is hyperconvex;
 		
 		\item For every weak metric form $ f : E \to \mathcal H$, the intersection
 		of balls $B\left(x,f(x)\right)$ is non-empty; 
 		
 		\item For every isometric extension $\mathbf E^{'}:=(E', d')$ of $\mathbf  E$ and every
 		$u \in E^{'} \setminus E$, there is a retraction of $\mathbf E'_{\restriction  E\cup \{u\}}$ onto 
 		$\mathbf E$.
 	\end{enumerate}
 \end{proposition}
 \begin{proof}
 	$(iii) \Rightarrow (ii)$  Let $f : E \to \mathcal H$ be a weak metric form  and $f_M$  be the largest metric form  below $f$ given by Lemma \ref{lemmakey}. According to Corollary \ref{cor:delta}, $\bar \delta$ is an isometry of $\mathbf E$ into $\mathcal L(\mathbf E)$. Thus, setting $\mathbf E':=\mathcal L(\mathbf E)$, we may view $\mathbf E'$ as an isometric extension of $\mathbf E$.  Since $f_M$ is a metric form,  Lemma \ref{lem:metric form} ensures that $d_{\mathbf E'} (\bar {\delta}(y),  f_M) = f_M(y)$ for all $y\in E$. Thus $ f_M\in \bigcap_{x\in E} B_{\mathbf E'}(\bar \delta (x), f_M(x))$. Any retraction of $\mathbf E'_{\restriction  \bar\delta (E)\cup \{u\}}$ onto 
 	$\mathbf E$ will send $f_M$ into $\bigcap_{x\in E} B_{\mathbf E}(x, f_M(x))$. According to Lemma \ref{lemmakey}, this intersection is $\bigcap_{x\in E}  B_{\mathbf E}(x, f(x))$.

 	$(ii) \Rightarrow (i)$ Let $\big(B(x_i,r_i)\big)_{i\in I}$ be a family of 
 	balls of $\mathbf E$ such that 
 	
 	\begin{equation} \label{eq:hyperconvexity}
 		d(x_i,x_j)\leq r_i\oplus \overline {r_j}
 	\end{equation}
 	for all $i,j\in I$. 
 	
 	Define
 	$f : E \to \mathcal H$ as follows:  for each $x\in E$, set
 	$ f(x) = \displaystyle\bigwedge _{ i\in I, x_i=x}r_i$.
 	The distributivity condition on $\mathcal H$ ensures that $$d(x,y) \leq f(x) \oplus \overline {f(y)}$$ for all $x,y\in E$. Hence $f$ is a weak metric form. 
 	It follows that:
 	$$\emptyset \ne \displaystyle\bigcap _{x \in E} B\big(x,f(x)\big) \subseteq
 	\displaystyle\bigcap _ {i \in I} B(x_i,r_i).$$
 	$(i) \Rightarrow (iii)$ Let $\mathbf E':= (E',d')$ be an isometric extension of 
 	$\mathbf E$
 	and $u \in E' \setminus E$. For all $x,y \in E$, we have
 	$d(x,y) = d'(x,y) \le d'(x,u) \oplus d'(u,y)$. Since $\mathbf E$ is 
 	hyperconvex,
 	the set
 	$\displaystyle \bigcap _{x \in E} B\Big(x,d'(x,u)\Big)$ is non-empty.
 	Let $u'$ be an arbitrary element of this intersection. The map
 	$g : E\cup \{u\} \to E$ defined by $g(x) = x$ for every $x \in E$ and
 	$g(u) = u'$ is a retraction.
 \end{proof}
 
 We conclude this paragraph with a characterization theorem:

 \begin{theorem}\cite{JaMiPo}\label{thm: caracterisation-hyperconvexity}
 	Let $\mathcal H$ be an Heyting algebra. Then,  for a metric space $\mathbf E:= (E,d)$ over $\mathcal  H$,
 	the following conditions are equivalent:
 	\begin{enumerate} [{(i)}]
 		\item $\mathbf E$ is an absolute retract; 
 		\item $\mathbf E$ is injective; 
 		\item $\mathbf E$ is hyperconvex; 
 		\item $\mathbf E$ is a retract of a power of $\mathcal H$.
 	\end{enumerate}
 \end{theorem}
 
 \begin{proof}
 	We just give an hint (for a detailed proof, see \cite{JaMiPo}). 
 	
 	$(i) \Rightarrow (iv)$ According to theorem \ref{thm:embedding}, the space $\mathbf E$ 
 	isometrically embeds into a power of $\mathbf H:= (\mathcal H, d_{\mathcal H})$; since it is an absolute retract, it
 	must be a retract of such a power.
 	
 	$(iv) \Rightarrow (iii)$ The space $\mathbf H$ is hyperconvex and the class of
 	hyperconvex spaces is closed under product and retract, i.e, in our
 	terminology, forms a variety.  
 	
 	$(iii) \Rightarrow (ii)$ We prove that the one-point extension holds. Let $\mathbf E':= (E', d')$, $A'\subseteq E'$, $x'\in E'\setminus A'$  and $ f : A \to E$ be a
 	nonexpansive map. Let $\mathcal B:= (B_{\mathbf E} (f(a'), d'(a', x')))_{a'\in A'}$. Since $f$ is nonexpansive, this family of balls satisfies the hyperconvexity  condition, namely 
 	$$d(f(a'), f(a''))\leq d'(a',a'') \leq d'(a', x') \oplus \overline{d'(a'', x')}. $$ Hence,  it has a non-empty intersection. Pick $x$ into  this intersection and set $f(x'):= x$. 
 	
 	$(ii) \Rightarrow (i)$ Trivial.
 \end{proof}

 \subsection{Injective envelope}
 A nonexpansive map $f: \mathbf E\longrightarrow \mathbf E'$ is \textit{essential} it for
 every nonexpansive map  $g: \mathbf E'\longrightarrow \mathbf E''$,  the map $g\circ f$ is
 an isometry if and only if $g$ is isometry (note that, in
 particular, $f$ is an
 isometry). An essential nonexpansive map $f$ from $\mathbf E$ into an injective metric space  $\mathbf E'$ over $\mathcal{H}$ is called an {\it injective envelope}  of $\mathbf E$. We will rather say that $\mathbf E'$ is  an  \emph{injective envelope}  of  $\mathbf E$. Indeed, this says in substance that  an injective envelope of a metric space $\mathbf E$ is  a minimal
 injective metric space over $\mathcal H$ containing isometrically $\mathbf E$.
 
 The construction of
 injective envelopes is based upon the notion of \textit{
 	minimal metric form}, a notion borrowed to Isbell \cite{isbell} that he calls \emph{extremal}. 
 
 Let us recall that a (weak) metric form is  \textit{minimal} if there is no other  (weak) metric form $g$
 satisfying $g\leq f$ (that is $g(x)\leq f(x)$ for all $x\in E$). Since from Lemma \ref{lemmakey}, every weak metric form  majorizes a metric form, the two notions of minimality coincide.  Due to the distributivity condition and the completeness of $\mathcal H$,  we may apply Zorn's lemma to get the existence of a minimal metric form  below any weak metric form.  
 
 As shown in
 \cite{JaMiPo}, (cf. also theorem 2.2 of \cite{KP2}):

 \begin{theorem} \label{thm: env injective}
 	Every generalized metric space $\mathbf E$ over a Heyting algebra $\mathcal H$ has an injective envelope , namely the space $\mathcal N(\mathbf E)$ of
 	minimal metric forms.
 \end{theorem} 
 \begin{proof}
 	Let $\mathbf E$ be a metric space over the Heyting algebra $\mathcal H$. 
 	
 	\noindent One proves first that   the space $\mathcal L(\mathbf E)$ of metric forms  is an absolute retract. This means that every isometric extension $\mathbf E':= (E',d')$ can be retracted on  $\mathcal L(\mathbf E)$. This is  almost immediate. For every $u\in E'$,  let   
 	$\varphi_u: E\rightarrow \mathcal H$ be defined by setting $\varphi_{u} (x):= d'(\bar \delta (x), u)$. Since the map $\bar \delta: E \rightarrow \mathcal H$ is an isometry, $\varphi_{u}$ is a metric form.  To conclude, one proves that  the map $\varphi: u \mapsto  \varphi_u$ is a retraction of $\mathbf E'$ on $\mathcal L(\mathbf E)$. First, $\varphi$ is the identity on $\mathcal L(\mathbf E)$. Indeed, if $f\in \mathcal L(\mathbf E)$ then, according to  $(iii)$ of Lemma \ref{lem:metric form},  $\varphi_f (x)= d(\bar \delta  (x), f)=f(x)$ for every $x\in E$,  hence $\varphi_f= f$. Next, $\varphi$ is nonexpansive, that is, $d(\varphi_u, \varphi_v) \leq d'(u,v)$ for all $u, v\in \mathbf E'$.   From the triangular inequality, we have:
 	\begin{equation}\label{eq:metricform1}
 		d'(\bar \delta (x), u)\leq d'(\bar {\delta} (x), v)\oplus d'(v,u)
 	\end{equation}
 	and
 	\begin{equation}\label{eq:metricform3}
 		d'(\bar \delta (x), v)\leq d'(\bar {\delta} (x), u)\oplus d'(u, v)
 	\end{equation}
 	for every $x\in E$. 
 	
 	\noindent These inequalities translate to $\varphi_{u}(x) \leq \varphi_{v}(x)\oplus \overline {d'(u,v)}$ and  $\varphi_{v}(x)\leq \varphi_{u}(x)\oplus d'(u,v)$,   that is $d_{\mathcal H}(\varphi_u(x), 
 	\varphi_v(x)) \leq d'(u,v)$.  This yields  $d(\varphi_u, \varphi_v):= \bigvee _{x\in E}d_{\mathcal H}(\varphi_u(x), 
 	\varphi_v(x))\leq d'(u,v)$, as required. 
 	
 	Next, one proves that the space $\mathcal N(\mathbf E)$ of minimal metric forms  over $\mathbf E$ is hyperconvex. According to $(iii)$ of Proposition  \ref{prop: test hyperconvexity} this amounts to  prove that 
 	for every isometric extension $\mathbf E':=(E', d')$ of $\mathcal  N(\mathbf  E)$ and every
 	$u \in E' \setminus \mathcal N (\mathbf E)$, there is a retraction of $\mathbf E'_{\restriction  \mathcal N(\mathbf E)\cup \{u\}}$ onto 
 	$\mathcal N(\mathbf E)$. This amounts to the fact that 
 	the intersection
 	of balls $A:= \bigcap _{f\in \mathcal N(\mathbf E)} B_{\mathbf E'}(f,  d' (f, u))$ contains some element $\tilde u$ of $\mathcal N(\mathbf E)$. Let $\varphi_u: \mathbf E\rightarrow \mathcal H$  defined by setting $\varphi_{u} (x):= d'(\bar \delta (x), u)$. As illustrated above,  this is a metric form  on $\mathbf E$. Let $\tilde  u$ be a minimal metric form  on $\mathbf E$ below $u$. Let $\phi: \bar \delta (E)\cup \{u\}\rightarrow \mathcal L(\mathbf E)$ be the nonexpansive map sending   $u$ to $\tilde u$ and leaving fixed every other element. Since $\mathcal L(\mathbf  E)$ is an absolute retract, it is injective, hence $\phi$ extends to a nonexpansive map $\Phi$ from $\mathbf E'$ into $\mathcal L(\mathbf E))$. This map is the identity on $\mathcal N(\mathbf E)$. Indeed, let $f\in \mathcal N(\mathbf  E)$. Since $\Phi$ is nonexpansive, we have $d(\bar \delta (x), \Phi(f))\leq  d'(\bar \delta (x), f)$ for every $x\in E$, meaning $\Phi(f)(x)\leq f(x)$. Since $f$ is minimal, $\Phi(f)=f$. From this, it follows that $d(f, \tilde u)= d(\Phi(f), \Phi(u)\leq d'(f,u)$ for every $f\in \mathcal N(\mathbf E)$.  This proves that $\tilde u$ belongs to $A$. Hence $\mathcal N(\mathbf E)$ is hyperconvex. According to Theorem \ref{thm: caracterisation-hyperconvexity} it is injective. If $\mathbf E'$ is an injective space  between $\mathbf E$ and $\mathcal N(\mathbf E)$ then the identity map $\id$ on $\mathbf E$ extends to a  nonexpansive map $\Phi$ from $\mathcal N(\mathbf E)$ into $\mathbf E'$. As above,  for every $f\in \mathcal N(\mathbf E)$ we have $\Phi (f)\leq f$ hence $\Phi(f) =f$ since $f$ is minimal. It follows that $\mathbf E'= \mathcal N(\mathbf E)$. This proves that $\mathcal N(\mathbf E)$ is a minimal injective metric space containing $\mathbf E$.   \end{proof}

 A particularly useful fact is the following:
 \begin{lemma}\label{fact:fixe}
 	If a nonexpansive map from an injective envelope of $\mathbf E:= (E, d)$ into itself  fixes $E$ pointwise,  then  it is the identity map.
 \end{lemma}

 Note that  two injective envelopes of $\mathbf E$ are   isomorphic via an isomorphism which is the identity over $\mathbf E$. This allows to talk about "the" injective envelope of $\mathbf E$.  A particular injective envelope of $\mathbf E$, as $\mathcal N(\mathbf E)$,   will be called a \emph{representation} of the injective envelope.

 We describe the  injective envelopes of two-element metric spaces (see \cite{KP2} for proofs).
 Let $\mathcal H$ be a Heyting algebra and $v\in \mathcal H$. Let $\mathbf E:=(\{x,y\}, d)$ be a two-element metric space over $\mathcal H$ such that $d(x,y)=v$. We denote by $\tilde {\mathcal{N}}_v$  the  injective envelope of $\mathbf E$. We give two representations of it.  Let   $\mathcal{C}_{v}$ be the set of all pairs
 $(u_{1},u_{2})\in \mathcal H^{2}$ such that $v\leq u_{1}\oplus  \overline {u_{2}} $. Equip this
 set with the ordering induced by the product ordering on $\mathcal H^{2}$ and denote by $\mathcal{N}_{v}$ the set of its minimal elements. Each element of $\mathcal{N}_{v}$ defines a  minimal metric form. We equip $\mathcal H^{2}$
 with the supremum distance:
 $$d_{\mathcal H^{2}}\left( (u_{1},u_{2}), (u'_{1},u'_{2})\right)
 :=d_{\mathcal H}(u_{1}, u'_{1})\vee d_{\mathcal H}(u_{2}, u'_{2}).$$

 Let $v\in \mathcal H$ and $\mathcal{S}_{v}:=  \left\{ \lceil v - \beta \rceil
 :\beta \in \mathcal H\right\}$  be the subset of $\mathcal H$;  equipped with the ordering induced by
 the ordering over $\mathcal H$ this is  a complete lattice. According to lemma 2.5 of \cite{KP2},  $(x_1, x_2)\in \mathcal{N}_{v}$ iff  $x_1= \lceil v - x_2 \rceil$ and $\overline {x_2}=\lceil - x_1 \oplus v \rceil$. This yields a correspondence between   $ \mathcal {N}_{v}$ and  $\mathcal{S}_{v}$.
 \begin{lemma} (Lemma 2.3,  Proposition 2.7 of \cite {KP2})
 	The space $\mathcal{N}_{v}$ equipped with the supremum
 	distance and the set $\mathcal{S}_{v}$  equipped with
 	the distance induced by the distance over $\mathcal H$ are injective envelopes of   the two-element metric spaces $\{({0}, v), (v, {0})\}$ and  $\left\{{0},v\right\}$ respectively. These spaces are isometric to the injective envelope of $\mathbf E:= (\{x,y\}, d)$ where $d(x,y)=v$.
 \end{lemma}
 The reader will find more details in \cite{KP2} and in \cite{kabil-pouzet}, with a presentation in terms of Galois correspondence. An illustration is given in Section \ref{section-illustration}.

 \subsection{Hole-preserving maps}
 In this subsection, we introduce the notions of hole-preserving maps. A large part is borrowed from subsection II-4 of \cite{JaMiPo}.

Let $\mathbf E$ and $\mathbf F$ be two  metric spaces over a Heyting algebra $\mathcal{H}$. If $f$ is a nonexpansive map from $\mathbf F$ into  $\mathbf E$, and $h$ is a map from   $F$ into $\mathcal{H}$, the \emph{image} of $h$ is the map $h_f$ from $E$ into $\mathcal{H}$ defined by $h_f(x): \bigwedge \{h(y): f(y)=x\}$ (in particular $h_f(x)= 1$ where $1$ is the largest element of $\mathcal{H}$ for every $x$ not in the range of $f$).  A \emph{hole} of $\mathbf F$ is any map $h:F \rightarrow \mathcal{H}$ such that  the intersection of balls $B(x, h(y))$ of $F$ ($x\in F$) is empty. If $h$ is a hole of $\mathbf F$, the map $f$ \emph{preserves} $h$ provided that $h_f$ is a hole of $\mathbf E$. The map $f$ is \emph{hole-preserving} if the image of every hole is a hole.

As it is easy to see, coretractions preserve holes and  hole-preserving maps are isometries. One may then use hole-preserving maps as approximations of coretractions

We recall the following result of \cite{JaMiPo}.

\begin{theorem}\label{thm:hole-preserving} On an  involutive Heyting algebra $\mathcal{H}$,  the absolute retracts and the injectives w.r.t. hole-preserving maps coincide. The class of these objects is closed under products and retractions. Moreover, every metric space embeds into some member of this class by some hole-preserving map. \end{theorem}

The  proof relies on  the introduction of the  replete space
 $\mathscr H(\mathbf E)$ of a generalized metric space $\mathbf E$. The space  $\mathbf E$ is an absolute retract (w.r.t. the hole-preserving maps) or not depending whether  $\mathbf E$ is a retract of $\mathscr H(\mathbf E)$ or not. Furthermore, with the existence  of the replete space one may prove  the \emph{transferability}  of hole-preserving maps (Lemma II-4.6 of \cite{JaMiPo}), that is the fact that  for every  nonexpansive map $f: \mathbf F \rightarrow \mathbf E$, and every hole-preserving map $g: \mathbf F\rightarrow \mathbf G$  there is a hole-preserving map  $g': \mathbf G \rightarrow \mathbf E'$ and  a nonexpansive map $f': \mathbf G\rightarrow  \mathbf E'$ such that  $g'\circ f= f'\circ g$.  Indeed, one may choose $\mathbf E'= \mathscr H(\mathbf E)$. As it is well known among categorists, the transferability property implies that absolute retracts and injective objects coincide \cite{kiss-marki-prohle-tholen}.

In the sequel we define the replete space and give the proof of the transferability property. 

Proofs are borrowed from \cite{JaMiPo}.

\begin{figure}[H] \centering
	\begin{tikzcd} [row sep=5em, column sep=6.5em]
	F \arrow[Rightarrow]{r}{g} \arrow[right ,swap]{d}{f}
	& G \arrow[d, dashed, "f'"] & F \arrow[Rightarrow]{r}{g} \arrow[right ,swap]{d}{f}
	& G \arrow[d, dashed, "\bar{f}"] \\
	E \arrow[r, Rightarrow, dashed, swap , "g'" ]
	& E'
	& E \arrow[r, Rightarrow, dashed, swap , "\bar{\delta}" ]
	& \mathscr H(E)
	\end{tikzcd}
	\caption{Transferability.}  \label{iter-F3}
	
	\end{figure}
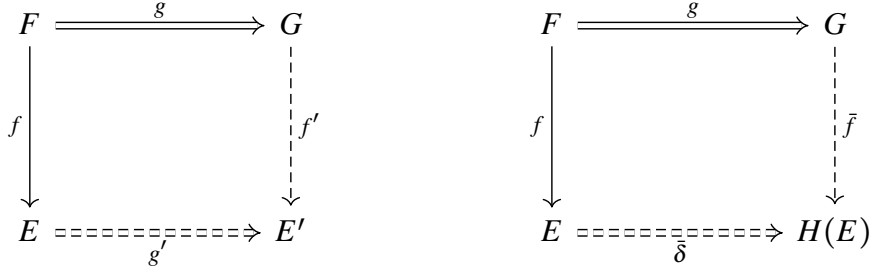

Let  $\mathscr H(\mathbf E)$ be the subset of $\mathcal L(E)$ consisting of metric forms $h$ such that the intersection of balls $B(x, h(x))$ for $x\in E$ is nonempty. If  $\mathcal{H}$ is a Heyting algebra, we may equip $\mathscr H(\mathbf E)$ with the distance induced by the sup-distance on $\mathbf H^E$. We call it the \emph{replete space}.

We recall the following two results of \cite{JaMiPo}.

\begin{lemma}(see Lemma II-4.3 p. 195) If $\mathbf E:= (E, d)$ is a metric space over a Heyting algebra $\mathcal{H}$  then  $\overline \delta: \mathbf E\rightarrow \mathscr H(\mathbf E)$  defined by $\overline \delta (x)(y):= d(y,x)$ is a hole-preserving map  from $\mathbf E$ into $\mathscr H(\mathbf E)$. Furthermore $\mathscr H(\mathbf E)$ is an absolute retract w.r.t. the hole-preserving maps (i.e.,  this is a retract of every extension by a hole-preserving map).
\end{lemma}

\begin{proof} The proof of this lemma is almost immediate. We just indicate that  $\mathscr H(\mathbf E)$ is an absolute retract.  Let $\mathbf E':=(E', d')$ be a hole-preserving extension of $\mathscr H(\mathbf E)$. For $u\in \mathbf E'$, set $\tilde u: E\rightarrow \mathcal H$ defined by setting $\tilde u(x):= d' (\overline \delta(x),u)$ for all $x\in E$. By construction, $\tilde u$ is a metric form; moreover,  it belongs to $\mathscr H(\mathbf E)$. To conclude, observe that the map $u \mapsto \tilde u$ is a retraction. \end{proof}

\begin{lemma}(see Lemma II-4.5 p. 196)\label{lem:extension} If $\mathbf E:= (E, d)$ and $\mathbf F$ are  two  metric spaces over a Heyting algebra $\mathcal{H}$  then  every nonexpansive map $f: \mathbf F\rightarrow \mathbf E$ extends to a nonexpansive map $\mathscr H_f:\mathscr H(\mathbf F)\rightarrow \mathscr H(\mathbf E)$. \end{lemma}

\begin{proof} The proof,  substantial,  relies on Lemma \ref{lemmakey}. 
We define $\tilde f:  \mathbf  H^{F}\rightarrow \mathbf  H^{E}$ by setting $\tilde f (h):= h_f$ where   $h_f$ is the map from $E$ into $\mathcal{H}$  defined by $h_f(x):= \bigwedge \{h(y): f(y)=x\}$). One check  first that this map is nonexpansive and next  that if $h\in \mathcal C(\mathbf F)$ then $\tilde f(h)\in \mathcal C( \mathbf E)$. For $k\in \mathcal C(E)$, let $k_M$ be the largest metric form below $k$  given by Lemma \ref{lemmakey}.  Let $r$ be the retraction from $\mathcal C(\mathbf E)$ onto $\mathcal L(\mathbf E)$ defined by setting $r(k):= k_M$ for all $k\in \mathcal C(E)$. The composition $r\circ \tilde f:\mathcal C( \mathbf F)\rightarrow \mathcal L(\mathbf E)$ is nonexpansive as a composition of nonexpansive maps. It extends $f$ once $F$ and $E$ are identified to their images $\overline \delta (F))$ and $\overline \delta(E)$, that is $(r\circ \tilde f)(\overline \delta (y))= \overline \delta ( f(y))$ for all $y\in F$. Indeed, observe first that $\tilde f(\overline \delta (y))(f(y))=0$. Next, since by definition of $r$,   $r (\tilde f(\overline \delta(y)))\leq \tilde f(\overline \delta(y))$, one has  $(r\circ \tilde f)(\overline \delta (y)(f(y))=0$. Since $r \circ  \tilde f(\overline \delta(y))$ is a metric form, this imposes that $r \circ  \tilde f(\overline \delta(y))= \overline \delta (f(y))$ (indeed, $d(\overline\delta (y),  r\circ \tilde f (\overline \delta (y)))=0$). Finally,  by Lemma \ref{lemmakey}, we have $\bigcap \{B(y, v(y)): y\in F\}=\bigcap \{B(x, \tilde f (v)(x): x\in E\}= \bigcap \{B(x, (r\circ \tilde f )(v)(x): x\in E\}$. Consequently,  $r\circ {\tilde f} (v) \in \mathscr {H}(\mathbf E)$ for every $v\in \mathscr {H}(\mathbf F)$. 
The restriction $\mathscr H_f$ of $r\circ \tilde f$ to $\mathscr H(\mathbf F)$ has the required property. \end{proof}

\begin{lemma}(see Lemma II-4.6 p. 197) The hole-preserving maps are transferable. \end{lemma}
\begin{proof}Let $f: \mathbf F\rightarrow \mathbf E$ be a nonexpansive map and $g: \mathbf F\rightarrow \mathbf G$ be a hole-preserving map. As above denote by $\overline \delta$ the map from $\mathbf E$ into $\mathscr H(\mathbf E)$ defined by $\overline \delta (x):= d (z, x)$ for $z\in E$. We define $\hat f: \mathbf G\rightarrow \mathscr H(\mathbf E)$ in such a way that $\hat f\circ g= \overline \delta\circ f$. 

For this purpose, define a nonexpansive map $\mathscr I_g:  \mathbf G\rightarrow \mathcal H$ as follows. For every $u\in G$, set   $\hat u:\mathbf F\rightarrow \mathcal H$  defined by $\hat u(y):=(d(g(y),u)$ and set $\mathscr J(u):= \hat u$. We check successively that the map $\hat u$ belongs to $\mathscr {H}(\mathbf F)$ (indeed, $u\in \bigcap_{y\in F} B(g(y), d(g(y),u))$; since $g$ is hole-preserving, $\bigcap_{y\in F} B(y, d(g(y),u))= \bigcap_{y\in F} B(y, \hat u(y))$ is non empty), hence $\mathscr I_g$ maps   $\mathbf G$ into  $\mathcal H$. Next, $\mathscr I_g$ is nonexpansive and finally $\mathscr I_g (g(y))= \overline \delta (y)$ for every $y\in F$ (since $g$ is hole-preserving, it is an isometry, thus $\mathscr I _g(g(y))(z)=d(g(z), g(x))=d(z,y)= \overline \delta(y)(z)$ for every $z\in F$).  Set $\hat f:= \mathscr H_f\circ \mathscr J_g$ where $\mathscr H_f$ is given by Lemma \ref{lem:extension}. Then $\hat f\circ g= \overline \delta \circ f$. This proves that $f$ is transferable. 
\end{proof}

 \subsubsection{Hole-preserving maps and one-local retracts}
 In his study of the fixed point property, Khamsi \cite{khamsi} introduced  a notion of one-local  retracts. This notion, defined for ordinary metric space,  extends to metric spaces over a Heyting algebra. In fact, it extends  to  metric spaces over an ordered monoid  equipped with an involution and more generally  to   binary structures which are reflexive and involutive in the sense of   \cite{khamsi-pouzet}.  It plays a crucial role in the fixed point theorem presented in the next section. In the sequel, otherwise stated, we do not suppose that $\mathcal H$ satisfies the distributivity condition. 
 
 Let $\mathbf  E:= (E,d)$ be a  metric space over  $\mathcal{H}$ and $A$ be a subset of $E$.  We say that $\mathbf E_{\restriction A}:= (A, d_{\restriction A})$  is a \emph{one-local retract} of $\mathbf E$  if it is a retract of $\mathbf E_{\restriction A\cup \{x\}} := (A\cup \{x\},  d_{\restriction A\cup \{x\}})$ (via the identity map) for every $x\in E$.

 \begin{lemma}\label{lem:one-local-retract} Let $\mathbf E:=   (E, d)$ be a metric space over $\mathcal{H}$ and $A$ be a subset of $E$. Then $\mathbf E_{\restriction A}$ is a one-local retract  of $\mathbf E$ iff for every family of balls $\Big(B(x_i, r_i )\Big)_{i \in I}$, with $x_i \in A$, $r_i\in \mathcal H$ for any $i \in I$, such that $\bigcap\limits_{i \in I}\ B_{\mathbf E}(x_i, r_i )$ is not empty, the intersection  $\bigcap\limits_{i \in I}\ B_{\mathbf E}(x_i, r_i ) \cap A$ is not empty.  
 \end{lemma}
 
 \begin{proof} Suppose that $\mathbf E_{\restriction A}$ is a one-local retract  of $\mathbf E$. Let $I$ be a set.  Consider a family of balls $\Big(B_{\mathbf E}(x_i, r_i )\Big)_{i \in I}$, with $x_i \in A$, $r_i\in \mathcal H$ for any $i \in I$, such that $B = \bigcap\limits_{i \in I}\ B_{\mathbf E}(x_i, r_i )$ is not empty. Let $a\in B$ and let $h$ be a retraction from $\mathbf E_{\restriction A \cup \{a\}}$ onto $\mathbf E_{\restriction A}$.  Set $a':= h(a)$. Since  $h$ fixes $A$ and retracts $a$ onto $a'$,  $a'\in  B_{\mathbf E}(x_i, r_i)$, hence  $a'\in \bigcap\limits_{i \in I}\ B_{\mathbf E}(x_i, r_i ) \cap A$.  Conversely, ones  proves that $\mathbf E_{\restriction A}$ is a one-local retract  provided that the intersection property of balls is satisfied.  Let $a\in E \setminus A$. 
 	Let
 	$$\mathcal B:= \{B(u, r):\; u\in A,\ a\in B(u, r)\;\text{and}\; r\in \mathcal H \}.$$
 	Set $B:= \bigcap \mathcal B$.  Then  $a\in B$ which implies $B\not=\emptyset$. According to the ball's property, $B\cap A\not =\emptyset$. Let $a'\in B\cap A$.  The map $h: A\cap \{a\}\rightarrow A$ which is the identity on $A$ and satisfies $h(a) = a'$ is a retraction of $\mathbf E_{\restriction A\cup \{a\}}$. 
 	
 \end{proof} 
 \begin{lemma} \label{lem:hole-localretract}Let $\mathbf E$ and $\mathbf E'$ be two metric spaces over $\mathcal{H}$. A nonexpansive map $f$ from $\mathbf E$ into    $\mathbf E'$ is hole-preserving iff  $f$ is an isometry of $\mathbf E$ onto its image and this image is a one-local retract  of $\mathbf E'$.
 \end{lemma}
 The routine proof is based on Lemma \ref{lem:one-local-retract}. We omit it.

 \section{Fixed point property}
 
 A central result in the category of ordinary metric spaces endowed with nonexpansive maps is the Sine-Soardi's fixed point theorem \cite{sine, soardi}
 asserting that every nonexpansive map on a bounded hyperconvex metric space  has a fixed point. 
 
 This result was generalized in two directions. First, Penot \cite{penot} introduced the notion of space endowed with a compact normal structure, extending the notion of bounded hyperconvex space. With this notion, Kirk's theorem \cite{kirk} amounted to the fact that  every nonexpansive  map on a space endowed with a compact normal structure has a fixed point. The existence of a common fixed point for a commuting set of nonexpansive maps was considered by several authors (see \cite {bruck-all, demarr, lim}). In 1986,  Baillon \cite{baillon}, extending  Sine-Soardi's theorem,  proved that every set of nonexpansive maps  which commute on a bounded hyperconvex space has a common fixed point. Khamsi \cite{khamsi} extended this  result to metric spaces endowed with a compact and normal structure. In \cite{JaMiPo} Sine-Soardi's theorem was extended to bounded hyperconvex spaces over a Heyting algebra, for an appropriate notion of boundedness. The possible extension to commuting set of nonexpansive maps was left unresolved (only the case of a countable set was settled). In \cite{khamsi-pouzet} the notion of compact normal structure for metric spaces over  Heyting algebra (and more generally for systems of binary relations) was introduced and Khamsi's theorem extended to families of nonexpansive maps  which commute on a space endowed with a compact and normal structure. 
 
 Here we present first the generalization of Sine-Soardi's theorem to bounded hyperconvex spaces over a Heyting algebra. Next, we introduce the notion of compact and normal structure and we present briefly the result of Khamsi-Pouzet.

 In the sequel we consider generalized metric spaces whose the set of values $\mathcal H$ does not satisfy necessarily the distributivity condition. We define for these spaces   the notions of diameter,  radius and Chebyshev center.  
 
 Let $\mathbf E:= (E, d)$ be a metric space over $\mathcal H$.  We denote by $\mathcal B_{\mathbf E}$ the set of balls of $\mathbf E$. Let $A$  be a nonempty subset of $E$ and $r \in \mathcal H$.  The $r$-{\emph center} is the set $C_{\mathbf E}(A, r ):= \{ x\in E : A \subseteq  B(x, r )\}$. Set $\Cov_{\mathbf E}(A):= \bigcap \{ B \in   \mathcal B_{\mathbf E}  : A\subseteq B\}$.
 The \emph{diameter} of $A$ is
 $\bigvee \{d(x,y): x,y\in A\}$. The \emph{radius} $r(A)$ is $\bigwedge \{v\in \mathcal{H}: A\subseteq B(x,v)\; \text{for some}\; x\in A\}$.  
 A subset $A$ of $E$ is \emph{equally centered} if $\delta (A)=r(A)$.

 \subsection{The case of hyperconvex spaces} 
 
 We suppose that $\mathcal H$ is a Heyting algebra. We define  the notion  of boundedness.
 
 An element $v\in V$ is \emph{self-dual} if $\overline v=v$, it is \emph{accessible} if there is some $r\in V$ with $v\not \leq r$ and $v \leq r\oplus \overline r$ and \emph{inaccessible} otherwise. Clearly, $0$ is inacessible;   every inaccessible element $v$ is self-dual (otherwise, $\overline v$ is incomparable to $v$ and we may choose $r:= \overline v$).
 
 \begin{definition}
 	We say that a space $(E, d)$ is \emph{bounded} if  $0$ is the only inaccessible element below $\delta(E)$.\end{definition}
 
 \begin{lemma}\label{equallycentered}
 	Let $A$ be an intersection of balls  of $(E, d)$. If $\delta(A)$ is inacessible  then  $A$ is equally centered; the  converse holds if $(E, d)$ is hyperconvex.
 \end{lemma}
 \begin{proof}
 	Suppose that $v:=\delta(A)$ is inaccessible. Let $r \in \mathcal H$  such that  $A \subseteq B(x, r)$. This yields $d(a,b)\leq d(a,x)\oplus d(x,b)\leq \overline r\oplus r$ for every $a,b\in A$. Thus $v\leq \overline r\oplus r$. Since $v$ is inacessible,  $v\leq r$, hence $v\leq r(A)$. Thus $v= r(A)$. Suppose that $A$ is equally centered. Let $r$ be such that $v\leq r\oplus \overline r$. The  balls $B(x, r)$ ($x\in A$)  intersect pairwise and intersect  each of the balls whose $A$ is an intersection; since $(E, d)$ is hyperconvex, these balls have a nonempty intersection. Any member $a$ of  this intersection is in $A$ and satisfies  $A \subseteq B(a, \overline r)$. Since $A$ is equally centered $r(A)=v$. Hence, $v\leq \overline r$. Since $v$ is self-dual,  $v\leq r$. Thus $v$ is inaccessible.
 \end{proof}
 
 \begin{lemma} \label{Mainlemma}
 	Let $\mathbf E$ be a non empty hyperconvex metric space  over a Heyting algebra $\mathcal H$ and $f:\mathbf E\longrightarrow \mathbf E$ be a nonexpansive mapping, then there is a non empty hyperconvex subspace $\mathbf S$ of $\mathbf E$ such that $f(S)\subseteq  S$ and its diameter $\delta(S)=\vee \{d(x,y): x,y\in S\} $ is inaccessible.
 \end{lemma}
 For a proof see Lemma III-1.1 of \cite{JaMiPo}.
 As a corollary, we have the following 
 \begin{lemma}
 	Let $\mathbf E$ be a non empty hyperconvex space. Then there is a non empty hyperconvex invariant subspace $\mathbf S$  whose diameter is inaccessible.
 \end{lemma}  

 \begin{theorem}\label{theoremSST}
 	Let $\mathbf E$ be a non empty bounded hyperconvex space. Then every nonexpansive map $f$ has a fixed point. Moreover, the restriction of $\mathbf E$ to the set  $Fix(f)$ of its fixed points is hyperconvex
 \end{theorem}
 \begin{proof}
 	Since $0$ is the unique inaccessible element below the diameter $\delta(E)$, the diameter of the non empty set $S$ given by lemma \ref{Mainlemma} is $0$, thus $S$ reduces to a single element fixed by $f$. 
 	Let $\{B_F(x_i,r_i):i\in I\}$ a family of balls of $Fix(f)$ with $d(x_i,x_j)\leq r_i+\overline{r}_j$ for all $i,j\in I$. Since $\mathbf E$ is hyperconvex, then $T=\cap \{B_F(x_i,r_i):i\in I\}\neq \emptyset$ and, as any intersection of balls of an hyperconvex space, $\mathbf E_{\restriction T}$ is hyperconvex  and, of course, bounded. Now, since $f$ is nonexpansive and the $x_i$ are fixed by $f$, we have $f(T)\subseteq T$. The f.p.p applied to $T$ gives an $x\in Fix(f)\cap T$. Thus, the above intersection is non empty and $Fix(f)$ is hyperconvex.

 \end{proof}
 \begin{corollary}
 	Let $\mathbf E$ be a non empty bounded hyperconvex space. Among the subspace of $\mathbf E$, the retracts of $E$ are the sets of fixed points of the nonexpansive maps from $\mathbf E$ into itself.
 \end{corollary}
 \begin{proof}
 	If $\mathbf  A$ is a retract of $\mathbf E$, then $A=Fix(g)$ for every retraction. Conversely, from the above result the set $Fix(f)$ of fixed points of a map $f:E\longrightarrow E$ is hyperconvex. But the hyperconvex are absolute retracts, thus $Fix(f)$ is a retract.
 \end{proof}

 \subsection{Compact and normal structure}
 
 We extend the notion of compact and normal structure defined by Penot, \cite{penot1, penot},  for ordinary metric spaces.  We consider generalized  metric spaces over an involutive ordered algebra $\mathcal H$ which, otherwise stated, does not necessarily satisfies the distributivity condition. We say that a metric space   $\mathbf E$  has a \emph{compact structure} if the collection of balls of $\mathbf E$ has the finite intersection property (f.i.p.) and it has  a \emph{normal structure} if for every   intersection of balls $A$, either $\delta(A)=0$ or $r(A)<\delta(A)$. This condition amounts to the fact that  the only equally centered intersections of balls are singletons.

 Lemma \ref{equallycentered} with the fact that the $2$-Helly property implies that the collection of balls has the finite intersection property, yields:
 
 \begin{corollary}\label{cor:compact+normal} If a generalized   metric space $\mathbf E:=(E,d)$ over a Heyting algebra  is bounded and hyperconvex then it has a compact normal structure.
 \end{corollary}

 We denote by $\hat {\mathcal B }_{\mathbf E}$, the collection of intersections of balls and by  $\hat {\mathcal B }_{\mathbf E}^*$ the set of the non empty one. 
 \begin{lemma}  \label{lem:infimum-invariant} Let  $\mathbf E := (E,d)$ be metric space over $\mathcal H$.  Let $f$ be  a nonexpansive map $\mathbf E$.  If $\mathbf {E}$ has a compact structure then every member of $\hat {\mathcal B }_{\mathbf E}^*$ preserved by  $f$ contains a minimal one. If $A \in \hat {\mathcal B}_{\mathbf E}^*$ is a minimal member preserved by $f$, then $\Cov_{\mathbf E}(f(A)) = A$ and $A$ is equally centered.
 \end{lemma}
 
 This lemma allows us to deduce Penot's formulation \cite{penot1, penot} of Kirk's fixed point theorem \cite{kirk} under our formulation.
 
 \begin{theorem}\label{cor:main}  Let $\mathbf E := (E, d)$ be a generalized metric space over $\mathcal H$.  Assume $\mathbf {E}$ has a compact normal structure.  Then every nonexpansive map $f$ on  $\mathbf E$ has a fixed point.
 \end{theorem}
 
 \medskip
 An easy consequence of Theorem \ref{cor:main}, we have the following beautiful structural result:
 
 \begin{proposition}\label{prop:main}  Let $\mathbf E := (E, d)$ be a a metric space over $\mathcal H$ with  a compact normal structure. Let $f$ be an endomorphism $\mathbf E$.  Then the restriction $\mathbf E_{\restriction Fix(f)}$, to the set $Fix(f)$ of fixed points of $f$, has a compact normal structure.
 \end{proposition}
 
 Proposition \ref{prop:main} will allow us to prove that a finite set of commuting endomorphism maps has a common fixed point and the restriction of $\mathbf E$ to the set of common fixed points has a compact normal structure.   Obviously one would like to know whether such a conclusion still holds for infinitely many maps.  In order to do this, one has to investigate carefully the structure of the fixed points of an endomorphism.  This will rely on the properties of oen local-retracts. 
 
 The next result is the most important one as it shows that a one-local retract  enjoys the same properties as the larger set.

 \begin{lemma}\label{lem:compact}  Let $\mathbf E:= (E, d)$ be a metric space over $\mathcal H$,   $X \subseteq E$  be a nonempty subset.  Assume $\mathbf {E}_{\restriction X}$ is a one-local retract  of $\mathbf E$. If $\mathbf E$  has a compact and normal structure, then $\mathbf {E}_{\restriction X}$ also has a normal compact structure.  \end{lemma}

 \begin{proposition}\label{prop:main2}  Let $\mathbf E:= (E, d)$ be a metric space over $\mathcal H$. Assume $\mathbf E$ has a compact normal structure.  Then for every nonexpansive map  $f$ of $\mathbf E$, the set of fixed points $Fix(f)$ of $f$ is a nonempty one-local retract  of $\mathbf {E}$.  Thus $\mathbf E_{\restriction Fix(f)}$ has a compact normal structure.
 \end{proposition}
 \begin{proof}  Let $I$ be a set.  Consider a family of balls $\Big(B_{\mathbf E}(x_i, r_i )\Big)_{i \in I}$, with $x_i\in Fix(f)$ and $r_i\in \mathcal E$ for  $i \in I$, such that $A = \bigcap\limits_{i \in I}\ B_{\mathbf E}(x_i, r_i )$ is not empty. Since each $x_i$ belongs to $Fix(f)$, $f$ preserves $A$. Since $A$ is an intersection of balls, Lemma \ref {lem:infimum-invariant} ensures that $A$ contains an intersection of balls $A'$ which is minimal,  preserved by $f$,  and equally centered. From the normality of $\mathbf E$,  $A'$  is reduced to a single element, i.e., $A'$ is reduced to a fix-point of $f$.  Consequently,  $A\cap  Fix(f)\not= \emptyset$. According to Lemma \ref{lem:one-local-retract}, $Fix(f)$ is a one-local retract.
 \end{proof}

 In \cite{khamsi-pouzet}, Khamsi and Pouzet proved that: 
 \begin{theorem} \label{thm:cor2}
 	If   a generalized metric space $\mathbf E:=(E,d)$   has a compact normal structure then every commuting family $\mathcal F$ of nonexpansive self maps has a common fixed point. Furthermore, the restriction of $\mathbf E$ to the  set $Fix (\mathcal F)$  of common fixed points of $\mathcal F$ is a one-local retract of $(E,d)$.
 \end{theorem}
 The fact that a space has a compact structure is an infinistic property (any finite metric space  enjoys it). A description of generalized metric spaces with a compact normal structure eludes us, even in the case of ordinary metric spaces.

 From Theorem \ref{thm:cor2},  we obtain:
 \begin{corollary} \label{thm:cor3}
 	If a generalized   metric space $\mathbf E$  is bounded and hyperconvex  then every commuting family  of nonexpansive self maps has a common fixed point.
 \end{corollary}
 \vskip3mm
 In order to prove the existence of a common fixed point for a family of nonexpansive mappings in the context of hyperconvex metric spaces, Baillon \cite{baillon} discovered an intersection property satisfied by this class of metric spaces.  In order to prove an analogue to Baillon's conclusion under our setting, we will need the following lemma.

 \medskip
 \begin{lemma}\label{lem:descending-intersection}  Let $\mathbf E:= (E, d)$ be a metric space  over $\mathcal H$, endowed with a compact normal structure.  Let $\kappa$ be an infinite  cardinal. For every ordinal $\alpha$, $\alpha< \kappa$, let $B_{\alpha}$ and $E_{\alpha}$ be subsets  of $E$ such that:
 	\begin{enumerate}
 		\item $B_{\alpha} \supseteq B_{\alpha+1}$ and $E_{\alpha} \supseteq E_{\alpha+1}$ for every $\alpha<\kappa$;
 		\item $\bigcap\limits_{\gamma<\alpha} B_{\gamma}=B_{\alpha}$ and $\bigcap\limits_{\gamma<\alpha} E_{\gamma}=E_{\alpha}$ for every limit ordinal $\alpha< \kappa$;
 		\item $\mathbf E_{\alpha}:=\mathbf E_{\restriction  E_{\alpha}}$ is a one-local retract of $\mathbf E$ and $B_{\alpha}$ is a nonempty intersection of balls of $\mathbf  E_{\alpha}$.
 	\end{enumerate}
 	Then $B_{\kappa}:=\bigcap\limits_{\alpha< \kappa} B_{\alpha}\not = \emptyset$.
 \end{lemma}
 
 The proof is  in \cite{khamsi-pouzet}, it  is beyond the scope of this paper.

 From Lemma \ref{lem:descending-intersection} follows:
 
 \begin{theorem}\label{thm:best} Let $\mathbf E:= (E, d)$ be generalized metric space. Assume that $\mathbf E$ has a compact normal structure.  Then, the intersection of every down-directed family $\mathcal F$ of one-local retracts is a nonempty one-local retract. 
 \end{theorem}
 
 \begin{proof} Let $\mathbf E: = (E,d)$ be a generalized metric space.  Assume $\mathbf E$ has a compact normal structure. Let $P$ be the set, ordered by inclusion, of nonempty subsets $A$ of $E$ such that $\mathbf E_{\restriction A}$ is a one-local retract of $\mathbf E$.  As any ordered set,  every down-directed subset of $P$ has an infimum  iff  every totally ordered subset of $P$ has an infimum (see \cite{cohn} Proposition 5.9 p 33).  We claim that $P$ is closed under intersection of every chain of its members.  Indeed, we argue by induction on the size of totally ordered families of one-local retracts  of $\mathbf E$. First we may suppose that $E$ has more than one element.  Next, we may suppose that these families are dually well ordered by induction. Thus,   given an infinite cardinal $\kappa$, let $(\mathbf E_{\restriction E_{\alpha}})_{\alpha<\kappa}$ be a descending sequence of one-local retracts  of $\mathbf E$. From the induction hypothesis, we may suppose that  the restriction of $\E$ to $E'_{\alpha}:= \bigcap\limits_{\gamma<\alpha}E_{\gamma}$ is a one-local retract  of $\mathbf E$ for each limit ordinal $\alpha< \kappa$. Hence, we may suppose that  $E_{\alpha}:= \bigcap\limits_{\gamma<\alpha}E_{\gamma}$ for each limit ordinal $\alpha< \kappa$. Since $\mathbf  E_{\alpha}:= \mathbf E_{\restriction E_{\alpha}}$ is a one-local retract of $\mathbf E$ and  $\mathbf E$ has a normal structure, $\mathbf E_{\alpha}$ has a normal structure (Lemma \ref{lem:compact}).  Hence, either $E_{\alpha}$ is a singleton, say $x_{\alpha}$, or $r_{\mathbf E_{\alpha}}(E_{\alpha}) \setminus \delta_{\mathbf E_{\alpha}} (E_{\alpha})\not =\emptyset$. In both cases, $E_{\alpha}$ is a ball of $\mathbf E_{\alpha}$. Hence  the first case, $E_{\alpha}= B_{\mathbf E_{\alpha}}(x_{\alpha}, r_{\restriction {E_{\alpha}}})$,  whereas in second case,  $E_{\alpha}\subseteq B_{\mathbf E_{\alpha}}(x, r)$ for some $x\in E_{\alpha}, r \in r_{\mathbf E_{\alpha}}(E_{\alpha}) \setminus \delta_{\mathbf E_{\alpha}} (E_{\alpha})$.  Hence,  Lemma \ref{lem:descending-intersection} applies with $B_{\alpha}= E_{\alpha}$ and gives that $E_{\kappa}$ is nonempty. Let us prove that $\mathbf E_{\kappa}:= \mathbf E_{\restriction E_{\kappa}}$ is a one-local retract  of $\mathbf E$. We apply Lemma \ref{lem:one-local-retract}. Let $(B_{\mathbf E}(x_i, r_i ))_{i\in I}, x_i \in E_{\kappa}$, $r_i\in \mathcal H$ be a family of balls such that the intersection is nonempty. Since $\mathbf E_{\alpha}$ is a one-local retract  of $\mathbf E$, the intersection $B_{\alpha}:= E_{\alpha}\bigcap \bigcap\limits_{i\in I}B_{\mathbf E}(x_i, r_i)$ is nonempty for every $\alpha<\kappa$. Now, Lemma \ref{lem:descending-intersection} applied to the sequence $(E_{\alpha}, B_{\alpha})_{\alpha<\kappa}$ tells us that $B_{\kappa}:=E_{\kappa}\bigcap \bigcap\limits_{i\in I} B_{\mathbf E}(x_i, r_i)$ is nonempty. According to Lemma \ref{lem:one-local-retract}, $\mathbf {E}_{\restriction B_{\kappa}}$ is a one-local retract  of $\mathbf E$.
 \end{proof}
 \medskip
 
 From this result, we  prove our fixed point result. 
 
 \noindent {\bf Proof of Theorem \ref{thm:cor2}.}
 For a subset $\mathcal {F'}$ of $\mathcal F$, let $Fix (\mathcal {F'})$ be the set of fixed points of $\mathcal {F'}$.  Using Proposition \ref{prop:main2}, we conclude that $\mathbf {E}_{\restriction Fix(\mathcal {F'})}$ is a nonempty one-local retract  of $\mathbf E$ for every finite subset $\mathcal {F'}$ of $\mathcal F$.  Indeed, we show this by induction on the number $n$ of elements of $\mathcal {F'}$. If $n=1$, this is the conclusion of Proposition \ref{prop:main2}.  Let $n\geq 1$. Suppose that the property holds for every subset $\mathcal {F''}$ of $\mathcal F'$ such that $\vert\mathcal {F''}\vert<n$. Let $f\in \mathcal {F'}$  and $\mathcal F'':= \mathcal {F'} \setminus \{f\}$. From our inductive hypothesis, $\mathbf {E}_{\restriction Fix(\mathcal {F''})}$  is a one-local retract  of $\mathbf E$. Thus, according to lemma \ref{lem:compact},   $\mathbf {E}_{\restriction Fix(\mathcal {F''})}$ has a compact normal structure. Now since $f$ commutes with every member $g$ of $\mathcal {F''}$, $f$ preserves $Fix(\mathcal {F''})$.  Indeed, if  $u \in Fix(\mathcal {F''})$,  we have $g(f(u))=f(g(u))=f(u)$, that is $f(u) \in Fix(\mathcal {F''})$. Thus $f$ induces  an endomorphism $f''$ of $\mathbf {E}_{\restriction Fix(\mathcal {F''})}$.    According to Proposition \ref{prop:main2},  the restriction of $\mathbf E_{\restriction Fix(\mathcal {F''}) }$ to $Fix (f'')$, that is $\mathbf E_{\restriction Fix(\mathcal {F'})}$,  is a nonempty one-local retract  of $\mathbf E_{\restriction Fix(\mathcal {F''})}$. Since the notion of one-local retract  is transitive  it follows that $\mathbf E_{\restriction Fix(\mathcal {F'})}$ is a nonempty one-local retract of $\mathbf E$.  Let $\mathcal P:= \{Fix (\mathcal {F''}): \vert \mathcal {F''}\vert < \aleph_{0} \}$ and $P:= \bigcap \mathcal P$. According to theorem \ref{thm:best},  $\mathbf E_{\restriction  P}$ is a one-local retract of $\mathbf E$. Since $P= Fix(\mathcal {F})$ the conclusion follows.
 \hfill $\Box$
 \section{Illustrations}

 \subsection{The case of ordinary metric and ultrametric spaces} 
 Let $\mathcal{H}:=\mathbb R^{+} \cup \{+\infty\}$.  The inaccessible  elements are $0$ and $+\infty$ hence, if one  deals with ordinary metric spaces,  unbounded  spaces in the above sense are those which are unbounded in the ordinary sense.  If one deals with ordinary metric spaces, infinite products  can yield spaces  for which $+\infty$ is attained. On may replace powers with $\ell^{\infty}$-spaces. Doing so, the notions of  absolute retract, injective, hyperconvex and retract of some $\ell^{\infty}_{\mathbb R}(I)$-space coincide. This is the result of Aronzajn and Panitchpakdi \cite{ArPa}. The existence of an injective envelope  was proved by Isbell \cite{isbell}. The injective enveloppe of a $2$-element ordinary metric space is a closed interval  $[0, v]$ of the real line with the absolute value distance; injective envelopes of ordinary  metric spaces up to five elements  have been described \cite {Dr1}. For some applications,  see \cite{Dr1, Che}.  
 
 The existence of a fixed point for a nonexpansive map on a bounded hyperconvex space is the famous result of Sine and Soardi. Theorem \ref{thm:cor2} applied to a bounded  hyperconvex metric space  is  Baillon's fixed point theorem.  Applied to
 a metric space with  a compact normal structure, this is  a result obtained by Khamsi 
 \cite{khamsi}. For a survey about hyperconvex spaces see \cite{espinola-khamsi}.
 
 The results presented about injective spaces   apply to ultrametric spaces  over $\mathbb R^{+} \cup \{+\infty\} $. A similar characterization to ours  was obtained in \cite{bayod-martinez};  a description of the injective envelope  is also given.  The paper \cite{PR} contains a study of ultrametric spaces  over a complete lattice satisfying our  distributivity condition, also called an \emph{op-frame}. Metric spaces over op-frame are  studied in \cite{ackerman}. Ultrametric spaces  over a lattice and their  connexion with collections of equivalence relations are also studied in \cite{braunfeld}.  
 
 \subsection{The case of ordered sets}\label{subsection:orderedset}
  
 Set $\mathcal H:=\{-, 0, 1, +\}$ with $0< -, +< 1$. The  retracts of powers of this lattice are all complete lattices. This is confirmed by the following fact.

 \begin{proposition}
 	A metric space $(\mathbf E,d)$ over $\mathcal{H}$ is hyperconvex iff the corresponding poset is a complete lattice.
 \end{proposition}

 Since $0$ is the only inacessible element of $\mathcal{H}$, Theorem \ref{thm:cor3} applies: \emph{Every commuting family of order-preserving maps on a complete lattice has a common fixed point}. This is Tarski's theorem (in full).
 
 Posets coming from  $\mathcal{H}$-metric spaces with a compact normal structure are a bit more general than complete lattices, hence Theorem \ref{thm:cor2} on compact normal structure could say a bit more than Tarski's fixed point theorem. As we will see below, in the case of one order-preserving map, this is no more than  Abian-Brown's fixed-point theorem.
 
 Let $\mathbf P:= (E, \leq)$ be a poset. We observe first that the f.i.p. property of the collection of balls $\mathcal B_{\mathbf P}:= \{\downarrow x: x\in E\}\cup \{\uparrow y: y\in E\}$ is an infinistic condition: it holds for every finite poset. In order to describe it we introduce the following notions.

 A pair of subsets $(A, B)$ of  $E$ is called a \emph{gap} of $\mathbf P$ if every element of $A$ is dominated by every element of $B$ but there is no element of $E$ which dominates every element of $A$ and is dominated by every element of $B$ (cf. \cite{DuRi}). In other words:
 $(\bigcap_{x\in A} B(x, \leq ))\cap (\bigcap_{y\in B} B(y, \geq ))= \emptyset$ while $B(x, \leq )\cap B(y, \geq ) \not =\emptyset$ for every $x\in A, y\in B$. A \emph{subgap} of $(A,B)$ is any pair $(A', B')$ with $A'\subseteq A$, $B'\subseteq B$, which is a gap. The gap $(A,B)$ is \emph{finite} if $A$ and $B$ are finite, otherwise it is \emph{infinite}.  Say that an ordered set $\mathbf Q$ \emph{preserves} a gap $(A,B)$ of $\mathbf P$ if there is an order-preserving map  $g$ of $\mathbf P$ to $\mathbf Q$ such that $(g(A), g(B))$ is a gap of $\mathbf Q$. On the preservation of gaps, see \cite{nevermann-wille}.

 \begin{lemma}\label{lem:gap}
 	Let $\mathbf P:= (E, \leq )$ be a poset. Then:
 	\begin{enumerate}
 		\item $\mathbf P$ is  a complete lattice iff $\mathbf P$ contains no gap;
 		\item An order-preserving map $f:\mathbf P\rightarrow \mathbf Q$ preserves all gaps of $\mathbf P$ iff it preserves all holes of $\mathbf P$ with values in $\mathcal H\setminus \{0\}$ iff $\mathbf Q_{\restriction {P}}$ is  a one-local retract  of $\mathbf Q$;
 		\item $\mathcal B_{\mathbf P}$ satisfies the f.i.p. iff every gap of $\mathbf P$ contains a finite subgap iff every hole is finite.
 	\end{enumerate}
 	
 \end{lemma}

 The routine proof is omitted. We may note the similarity of $(b)$ and lemma \ref{lem:hole-localretract}.
 
 From item (c) of lemma \ref{lem:gap} it follows that every nonempty chain in  a poset $\mathbf P$ for which the collection of balls has the $f.i.p$,  has a supremum and an infimum. Such a poset is called \emph{chain-complete}.
 
 Abian-Brown's theorem \cite {abian-brown} asserts that \emph{in a chain-complete poset with a least or largest element, every order-preserving map  has a fixed point}.
 The fact that the collection of intersection of balls of $\mathbf P$ has a normal structure  means that every nonempty intersection of balls of $\mathbf P$ has either a least or largest element. Being the intersection of the empty family of balls,   $\mathbf P$ has either  a least element or a largest element. Consequently, if $\mathbf P$ has a compact normal structure, we may suppose without loss of generality  that it has a least element. Since every nonempty chain has a supremum, it follows from Abian-Brown's theorem that every order preserving map has a fixed point. 
 
 On the other hand,  a description of posets with a compact normal structure is still open. 
 We just  observe that retracts of powers of $\bigvee$ or retracts of powers of $\bigwedge$ have a compact normal structure.

 Theorem \ref{thm:cor2} above yields a fixed point theorem for a commuting family of order-preserving maps  on any retract of a power of $\bigvee$ or of a power of  $\bigwedge$. But this result  says nothing about retracts of products of $\bigvee$ and $\bigwedge$.  These two posets fit in the category of fences. As we have seen in Subsection \ref{subsection:fencedistance}, every poset embeds isometrically (w.r.t. the fence distance)  into a product of fences. Fences are hyperconvex, hence from Theorem \ref{thm:cor3} it follows that:

 \begin{theorem}\cite[Theorem 4.18]{khamsi-pouzet}\label{thm:cor5}
 	If a poset $\mathbf Q$ is a retract of a product $\mathbf P$ of   finite fences of bounded length, every commuting set of order-preserving maps  on $\mathbf Q$ has a fixed point.
 \end{theorem}

 Since every complete lattice is a retract of a power of the two-element chain, this result contains Tarski's fixed point theorem.

 \subsection{The case of graphs}\label{subsection:directedgraphs}
 
 Retracts of (undirected) graphs have been considered by various authors, for reflexive graphs as well as irreflexive graphs (see \cite {bandelt-farber-hell, hell16, hell18, He1, hell-rival}. The existence of the injective envelope  of an undirected graph (presented in \cite{JaMiPo}) is in \cite{pesch}, a characterization of injective graphs is in \cite{bandelt-pesh}.

 To each directed graph,  we have associated its zigzag distance,  yielding a metric with values in $\mathcal H_{\Lambda}: =\mathbf {F}(\Lambda^{*})$. Metric spaces over $\mathcal H_{\Lambda}$ such that the distance is the zigzag distance associated with a  reflexive directed graph  were  characterized by Lemma  \ref{lem:connexity}. 
 The condition stated there is a weak form of convexity, thus it holds for hyperconvex spaces.  Let $\mathcal D:=\mathcal H_{\Lambda}$ and  $\mathcal G_{\mathcal D}$ be the class  of graphs whose the zigzag distance belongs to $\mathcal D$. With the homomorphisms of graphs, it becomes a  category.  As a category, $\mathcal G_{\mathcal D}$ identifies to a full subcategory of $\mathcal M_{\mathcal D}$, the  category   of metric spaces  over $\mathcal D$, with the nonexpansive maps as morphisms (see Lemma \ref{graphmorphisms}). 
 
 According to Theorem\ref {thm: caracterisation-hyperconvexity}: 
 
 \begin{theorem}A member $\mathbf E:= (E, d)$ of $\mathcal M_{\mathcal D}$ is an absolute retract iff the distance on $E$ comes from a directed graph  and this graph is an absolute retract in $\mathcal G_{\mathcal D}$ with respect  to isometric embedding.  \end{theorem}
 
 These members  of $\mathcal G_{\mathcal D}$ are described by the following result: 
 
 \begin{theorem}\cite{JaMiPo}
 	For a reflexive directed graph $\mathbf G=(E, \mathcal E)$,
 	the following conditions are equivalent 
 	\begin{enumerate}[{(i)}]
 		\item   $\mathbf G$ is an absolute retract with respect to isometries; 
 		\item   $\mathbf G$ is injective with respect to isometries; 
 		\item $G$ has the extension property; 
 		\item The collection of balls $B(x,\uparrow\alpha)_{x \in E, \alpha \in
 			\Lambda^{*}}$ has the 2-Helly property;  
 		\item $\mathbf G$ is a retract of a power of $\mathbf G_{\mathcal H_{\Lambda}}$.
 	\end{enumerate}
 \end{theorem}

 Every metric space $\mathbf E$ over $\mathcal H_{\Lambda}$ has an injective envelope; being injective, its metric comes from a graph. If $\mathbf E$ comes from a graph, the graph corresponding to the injective envelope of $\mathbf E$ is the injective envelope in $\mathcal G_{\mathcal D}$. For more recent facts about  the injective envelope see \cite{kabil-pouzet}. 
 
 We just mention a simple example  of hyperconvex graph.

 \begin{lemma}\label{zigzag1} The metric space associated to any directed zigzag $\mathbf Z$  has the extension property. In particular, every nonexpansive map sending two vertices of a reflexive directed graph $\mathbf G$ on the extremities of $\mathbf Z$ extends to a graph   homorphism from $\mathbf G$ to $\mathbf Z$. \end{lemma}
 
 \begin{proof}
 	Let $\mathbf Z$ be a directed zigzag (with loops). Its symmetric hull (obtained  by deleting the orientation of arcs in $\mathbf Z$) is a path. The balls in $\mathbf Z$ are intervals of that path, and each of these intervals  is either finite or the full path. Hence, every family of balls has the $2$-Helly property. Since convexity holds trivially, $\mathbf Z$, as a metric space over $\mathcal H_{\Lambda}$, is hyperconvex, hence according to   Theorem \ref{thm: caracterisation-hyperconvexity}, it satisfies the extension property. 
 \end{proof}
 
 For more,  we refer to \cite{KP1, KP2}.

 \subsection{The case of oriented graphs}\label{subsection:orientedgraphs}
 The situation  of oriented graphs is  different. These graphs cannot be modeled over a Heyting algebra (theorem IV-3.1 of  \cite{JaMiPo} is erroneous), but the absolute retracts in this category can be (this was proved by Bandelt, Sa\"{\i}dane and the second author and included in \cite{Sa}; see also the forthcoming paper \cite{bandelt-pouzet-saidane}). The appropriate Heyting algebra is the  \emph{MacNeille completion} of $\Lambda^{\ast}$ where $\Lambda:= \{+,-\}$.
 
 The MacNeille completion of $\Lambda^{\ast}$ is in some sense the least complete lattice extending $\Lambda^{\ast}$. The definition goes as follows. If $X$ is a subset of $\Lambda^{\ast}$ ordered by the subword ordering then
 
 $$\uparrow\! X:= \{ \beta  \in \Lambda^{\ast}: \alpha \leq \beta\;  \text{for some}\;  \alpha\in X\}$$  is the \emph {final segment generated by} $X$ and
 $$\downarrow \!X := \{\alpha  \in \Lambda^{\ast}:  \alpha \leq \beta \;  \text{for some}\;  \beta\in X\}$$   is the \emph{initial segment  generated by} $X$. For a singleton $X =
 \{\alpha\}$, we omit the set brackets and call $\uparrow\! \alpha $ and $\downarrow
 \!\alpha $ a  \emph{principal final segment} and a \emph{principal initial segment} respectively.
 
 $$X^{\Delta}:= \bigcap_{x \in X} \uparrow x$$
 is the {\it upper cone} generated by $X$, and
 $$X^{\nabla}:= \bigcap_{x \in X} \downarrow x$$
 is the {\it lower cone} generated by $X$.  The pair $(\Delta, \nabla)$ of mappings on the complete
 lattice of subsets of  $\Lambda^{\ast}$ constitutes a Galois connection. Thus,  a set $Y$ is an upper cone  if
 and only if $Y = Y^{\nabla \Delta}$, while a set $W$ is
 an lower cone if and only if $W = W^{\Delta \nabla}.$ This Galois connection
 $(\Delta, \nabla)$ yields the {\it MacNeille completion} of
 $\Lambda^{\ast}.$ This completion is realized  as the complete
 lattice $\{W^{\nabla}:  W\subseteq \Lambda^{\ast}\}$
 ordered by inclusion or alternatively $\{Y ^{\Delta} : Y\subseteq \Lambda^{\ast}\}$ ordered by reverse inclusion. We choose as completion the set $\{Y ^{\Delta} : Y\subseteq \Lambda^{\ast}\}$ ordered by reverse inclusion  that we denote by $\mathbf {N}(\Lambda^{\ast})$. This complete lattice is studied in detail  in \cite{bandelt-pouzet}.

 We recall the following characterization of members of the MacNeille completion of $\Lambda^*$.
 \begin{proposition}\label{cancellation} \cite{bandelt-pouzet} corollary 4.5.
 	A member  $Z$ of $\mathbf{F}(\Lambda^{*})$ belongs to $\mathbf {N}(\Lambda^{*})$ if and only if it satisfies the following cancellation rule:
 	if $u + v \in Z$ and $u - v \in Z$ then $u v \in Z$.
 \end{proposition}
 
 The concatenation, order and involution defined on $\mathbf {F}(\Lambda^{\ast})$ induce a Heyting algebra $\mathcal N_{\Lambda}$ on  $\mathbf {N}(\Lambda^{\ast})$ (see Proposition 2.2 of \cite{bandelt-pouzet}). Being a  Heyting algebra, $\mathcal N_{\Lambda}$ supports a  distance $d_{\mathcal N_{\Lambda}}$ and this distance is the zigzag distance of a graph $\G_{\mathcal {N}_{\Lambda}}$. But, it is not true that every oriented graph embeds isometrically into a power of that graph. For example, an oriented cycle cannot be embedded.  The following result characterizes  graphs which can be  isometrically embedded, via the zigzag distance, into  products of reflexive and oriented zigzags. It is stated in part in Subsection IV-4 of \cite{JaMiPo}, cf. Proposition IV-4.1.

 \begin{theorem}\label{theo:isometric}
 	For a  directed graph   $\G: = (E, \mathcal E)$ equipped with the zigzag distance,   the following properties are equivalent:
 	\begin{enumerate} [{(i)}]
 		\item $\G$ is isometrically embeddable into a product of  reflexive and oriented zigzags;
 		\item $\G$ is isometrically embeddable into a power of $\G_{\mathcal  N_{\Lambda}}$;
 		\item The values of the zigzag distance between  vertices of $\G$ belong to $\mathcal N_{\Lambda}$.
 	\end{enumerate}
 \end{theorem}
 The proof follows the same lines as the proof of  Proposition IV-5.1 p.212 of \cite{JaMiPo}.

 We may note that  the product can be infinite  even if the graph $\G$ is finite. Indeed, if $\G$ consists of two  vertices $x$ and $y$ with no value on the pair $\{x, y\}$ (that is the underlying graph is disconnected) then we need infinitely many zigzags of arbitrarily long length.
 \begin{theorem}\label{thm:ARcompletion}
 	An oriented graph $\G:= (V, \mathcal E)$ is an absolute retract in the category of oriented graphs if and only if it is a retract of a product of  oriented zigzags.\end{theorem}
 We just give a sketch. For details,  see section V of \cite{Sa} and the  forthcoming paper \cite{bandelt-pouzet-saidane}. The proof has three  steps. Let $\G$ be an absolute retract. First, one proves that $\G$ has no $3$-element cycle. Second, one proves that the zigzag distance  between two vertices of $\G$ satisfies the cancelation rule. From Proposition \ref{cancellation},  it belongs to $\mathbf{N}(\Lambda^*)$; from theorem \ref{theo:isometric}, $\G$ isometrically embeds into a product of oriented zigzags. Since $\G$ is an absolute retract, it is a retract of that product.
 As illustrated by the results of Tarski and Sine-Soardi, absolute retracts are appropriate candidates for the fixed point property. Reflexive graphs with the fixed point property must be antisymmetric, i.e., oriented. Having described absolute retracts among oriented graphs,  we derive from Theorem \ref{thm:cor3} that the  bounded ones have the fixed point property.
 We start with a characterization of accessible elements of $\mathcal N_{\Lambda}$. The proof is omitted.
 \begin{lemma} \label{lem-accessible}Every element $v$ of $\mathcal N_{\Lambda}\setminus \{\Lambda^{\ast}, \emptyset\}$ is accessible.
 \end{lemma}
 \begin{theorem}\label{thm:cor4}
 	If a graph $\G$, finite or not, is a retract of a product of reflexive and directed zigzags of bounded length then every commuting set of endomorphisms has a common fixed point.
 \end{theorem}
 \begin{proof}
 	We may suppose that $\G$ has more than one vertex. The diameter of $\G$ equipped with the zigzag distance  belongs to $\mathcal N_{\Lambda}\setminus \{\Lambda^{\ast}, \emptyset\}$. According to  Lemma \ref{lem-accessible},  it is accessible, hence as a metric space,  $\G$ is bounded. Being a retract of a product of hyperconvex metric spaces  it is hyperconvex. Theorem \ref{thm:cor3} applies.
 \end{proof}

 The properties of  reflexive and involutive transition systems extend almost verbatim the properties of directed graphs. They have been extended to non necessary reflexive transition systems (\cite{PR}, \cite{hudry1, hudry2}. Instead of presenting these properties, we illustrate their use in the following section. 
 \section{An illustration of the usefulness of the injective envelope}\label{section-illustration}
 
 Using the notion of injective envelope, we prove that on an ordered alphabet   $\Lambda$ the monoid $\F^{\circ}(\Lambda^{\ast}): =\F(\Lambda^{\ast})\smallsetminus  \{ \emptyset\} $ is  free. This result is exposed in \cite{KPR}.

 \begin{theorem}\label{prop1}
 	$\F^{\circ}(\Lambda^{\ast})$ is a free monoid.
 \end{theorem}
 
 We recall that a member $F$ of $\F(\Lambda^{\ast})$ is \emph{irreducible} if it is distinct from $\Lambda^{\ast}$ and  is not the concatenation of two members  of $\F(\Lambda^{\ast})$ distinct of $F$ (note that with this definition, the empty set is irreducible).  The fact that $\F^{\circ}(\Lambda^{\ast})$ is free amounts to the fact that each  member decomposes in a unique way as a concatenation of  irreducible elements. We give a synctactical proof in \cite{KPR} and   one with a geometric flavor.

 We present the last one. 
 
 We suppose that $\Lambda$ is equipped with an involution (this is not a restriction: we may choose the identity on $\Lambda$ as our involution). Then, we consider   metric spaces  whose  values of the distance  belong to $\mathbf F(\Lambda^{\ast})$. The category of metric spaces over $\mathbf F(\Lambda^{\ast})$, with the nonexpansive maps as morphisms,  has enough injectives. Furthermore, for every  final segment $F$ of $\Lambda^{\ast}$,
 the $2$-element  metric space $\mathbf E:=  (\left\{
 x,y\right\}, d)$ such that $d(x,y)=F$,  has an \emph{injective envelope} $\mathcal S_F$. 
 
 We define the gluing of two metric spaces by a common vertex.  Suppose that two metric spaces $\mathbf E_1:= (E_1, d_1)$ and $\mathbf E_2:= (E_2, d_2)$ have only one common vertex, say $r$. On the union $E_1\cup E_2$ we may define a distance  extending both $d_1$,  $d_2$,  setting $d(x, y):= d_i(x, r)\oplus d_j(r,y)$ for $x\in E_i, y\in E_j,$ and $i\not =j$. If $E_1$ and $E_2$ are arbitrary, we may replace them by isometric copies with a common vertex. We apply this construction to the injective envelope of two-element metric spaces. Let  $v_1$ and $v_2$ be  two  elements of a Heyting algebra and    $\mathcal S_{v_1}, \mathcal S_{v_2}$ be their injective envelopes. 
 Suppose that $\mathcal S_{v_1}$ is the injective envelope of $\{x_1, y_1\}$ with $x_1:=0, y_1:=v_1$ and that $\mathcal S_{v_2}$ is the injective envelope of $\{x_2, y_2\}$ with $x_2:=v_1$ and has no other element in common with $\mathcal S_{v_1}$. Let  $\mathcal S_{v_1}\oplus \mathcal S_{v_2}$ be their gluing. Since the distance from $x_1$ to $y_2$ is $v_1\oplus v_2$, this space embeds isometrically into the injective envelope $\mathcal S_{v_1\oplus v_2}$. For some Heyting algebras (and $ v_1$, $v_2$ distinct from $1$), these two spaces are isometric (see Figure \ref{fig:interpretation} for a geometric  interpretation). This is the case  of the Heyting algebra $\F(\Lambda^{\ast})$ (Corollary 4.9, p. 177 of \cite{KP2}). In terms of this Heyting algebra, this yields with obvious notations 
 \begin {equation}\label {eq:sum injective}
 \mathcal S_{F_1} \mathcal S_{F_2}\cong \mathcal S_{F_1 F_2}\;   \text{for all}\;F_1, F_2\in \F^{\circ}(\Lambda^{\ast}).
\end{equation} 

Say that an  injective which is not the gluing of two proper injectives is \emph{irreducible}. 
From  (\ref {eq:sum injective}) it follows that an injective of the form  $\mathcal S_F$ is irreducible iff $F$ is irreducible. 

In order to prove  that the decomposition of a final segment $F$ into a concatenation of irreducible final segments  is unique, we consider the  transition system  $\mathcal {M}_{F}$ on the alphabet $\Lambda $, with transitions $(p, a, q)$ if $a \in d(p,q)$, corresponding to the injective envelope $\mathcal S_F$. The automaton  $\mathcal{A}_{F}:=  ({\mathcal M}_{F},\left\{ x\right\}, \left\{
y\right\} )$ with $x=\Lambda^{\ast}$ as initial state and $y=F$ as final state accepts $F$.  A transition system yields a directed graph  whose arcs are the ordered pairs $(x,y)$ linked by some  transition. Since   the transition system $\mathcal {M}_{F}$ is reflexive and involutive,  the corresponding  graph $\mathbf {G}_{F}$ is undirected and has  a loop  at every vertex.  For an example, if $F= \Lambda^{\ast}$, $\mathcal S_F$ is the one-element metric space and $\mathbf{G}_{F}$ reduces to a loop. If $F= \emptyset$, $\mathcal S_F$ is the two-elements  metric space $E:= (\{x, y\}, d)$ with $d(x,y)= \emptyset$ and $\mathbf {G}_{F}$ has no edge. The cut vertices  of  $\mathbf {G}_{F}$ (vertices whose deletion increases the number of connected components) allow  to reconstruct the irreducible components of $\mathcal S_F$. 

\begin{figure}[h!]

\begin{center}
	\unitlength=1cm
	\definecolor{ttqqqq}{rgb}{0.2,0.,0.}
	\begin{tikzpicture}
	\draw[-] ( 0 , 0 ) .. controls ( 0.7 , 0.4 ) and ( 1.3 , 0.4 ) .. ( 1.8 , 0.3 ) node[above,sloped,pos=0.5] {};
	\draw[-] ( 0 , 0 ) .. controls ( 0.6 , -0.5 ) and ( 1.2 , -0.5 ) .. ( 1.8 , 0.3 ) node[above,sloped,pos=0.5] {};
	\draw ( 0.5 , 0.35 ) node[anchor=north west] {$S_{v_1}$};
	\draw[-] ( 1.8 , 0.3 ) .. controls ( 2.5 , 0.5 ) and ( 3.2 , 0.5 ) .. ( 4.3 , 0 ) node[above,sloped,pos=0.5] {};
	\draw[-] ( 1.8 , 0.3 ) .. controls ( 2.5 , -0.5 ) and ( 3.2 , -0.5 ) .. ( 4.3 , 0 ) node[above,sloped,pos=0.5] {};
	\draw ( 2.6 , 0.3 ) node[anchor=north west] {$S_{v_2}$};
	\draw ( 0 , 0 ) -- ( 0 , -1.0 );
	\draw ( 1.8 , 0.3 ) -- ( 1.8 , -0.8 );
	\draw ( 4.3 , 0 ) -- ( 4.3 , -1.0 );
	\draw[->] ( 0.03 , -0.9 ) -- ( 4.27 , -0.9 );
	\draw[->] ( 0.03 , -0.7 ) -- ( 1.77 , -0.7 );
	\draw ( 0.6 , -0.33 ) node[anchor=north west] {\scriptsize $v_1$};
	\draw[->] ( 1.84 , -0.7 ) -- ( 4.27 , -0.7 );
	\draw ( 2.6 , -0.33 ) node[anchor=north west] {\scriptsize $v_2$};
	\draw ( 2 , -0.83 ) node[anchor=north west] {\scriptsize $v$};
	\put( 1.8 , 1.4 ){\circle*{.08}}
	\put( 1.8 , 0.3 ){\circle*{.08}}
	\put( 0 , 0 ){\circle*{.08}}
	\put( 4.3 , 0 ){\circle*{.08}}
	\draw[-] ( 0 , 0 ) -- ( 1.8 , 1.4 );
	\draw[->] ( 0.8 , 0.62 ) -- ( 0.9 , 0.7 );
	\draw ( 0.5 , 1.1 ) node[anchor=north west] {\scriptsize $u_1$};
	\draw[-] ( 1.8 , 1.4 ) -- ( 4.3 , 0 );
	\draw[->] ( 1.8 , 1.4 ) -- ( 3.05 , 0.7 );
	\draw ( 2.8 , 1.1 ) node[anchor=north west] {\scriptsize $u_2$};
	\draw[-] ( 4.5 , 0 ) -- ( 5.4 , 1.8 );
	\draw[-] ( 4.5 , 0 ) -- ( 5.4 , -1.8 );
	
	\draw[-] ( 5.6 , 1.5 ) .. controls ( 6.3 , 1.9 ) and ( 6.9 , 1.9 ) .. ( 7.4 , 1.8 ) node[above,sloped,pos=0.5] {};
	\draw[-] ( 5.6 , 1.5 ) .. controls ( 6.2 , 1.0 ) and ( 6.8 , 1.0 ) .. ( 7.4 , 1.8 ) node[above,sloped,pos=0.5] {};
	\draw ( 5.9 , 1.63 ) node[anchor=north west] {\scriptsize $S_{v_1}$};
	\draw[-] ( 7.4 , 1.8 ) .. controls ( 8.1 , 2.0 ) and ( 8.8 , 2.0 ) .. ( 9.9 , 1.5 ) node[above,sloped,pos=0.5] {};
	\draw[-] ( 7.4 , 1.8 ) .. controls ( 8.1 , 1.0 ) and ( 8.8 , 1.0 ) .. ( 9.9 , 1.5 ) node[above,sloped,pos=0.5] {};
	\draw ( 8.2 , 1.8 ) node[anchor=north west] {$S_{v_2}$};
	\draw ( 5.6 , 1.5 ) -- ( 5.6 , 0.7 );
	\draw ( 7.4 , 1.8 ) -- ( 7.4 , 0.7 );
	\draw ( 9.9 , 1.5 ) -- ( 9.9 , 0.7 );
	\draw[->] ( 5.63 , 0.8 ) -- ( 7.37 , 0.8 );
	\draw ( 6.2 , 1.17 ) node[anchor=north west] {\scriptsize $v_1$};
	\draw[->] ( 7.44 , 0.8 ) -- ( 9.87 , 0.8 );
	\draw ( 8.2 , 1.17 ) node[anchor=north west] {\scriptsize $v_2$};
	\put( 7.4 , 2.9 ){\circle*{.08}}
	\put( 7.4 , 1.8 ){\circle*{.08}}
	\put( 5.6 , 1.5 ){\circle*{.08}}
	\put( 9.9 , 1.5 ){\circle*{.08}}
	\draw[-] ( 5.6 , 1.5 ) -- ( 7.4 , 2.9 );
	\draw[->] ( 6.4 , 2.12 ) -- ( 6.5 , 2.2 );
	\draw ( 6.1 , 2.6 ) node[anchor=north west] {\scriptsize $u_1$};
	\draw[-] ( 7.4 , 2.9 ) -- ( 9.9 , 1.5 );
	\draw[->] ( 7.4 , 2.9 ) -- ( 8.65 , 2.2 );
	\draw ( 8.4 , 2.6 ) node[anchor=north west] {\scriptsize $u_2$};
	\draw[->] ( 7.4 , 2.9  ) -- ( 6.5 , 1.49 );
	\put( 6.5 , 1.5 ){\circle*{.08}}
	\draw[-] ( 6.5 , 1.5 ) .. controls (6.3 , 1.6 ) and (5.9 , 1.6 ) .. (5.6 , 1.5 ) node[above,sloped,pos=0.5] {};
	\draw[-] ( 6.5 , 1.5 ) .. controls (6.8 , 1.82 ) and (7.1 , 1.82 ) .. (7.4 , 1.8 ) node[above,sloped,pos=0.5] {};
	\draw[- ] (6.5 , 1.5   ) -- ( 6.5 , 1.3 );
	\draw[->] (6.5 , 1.3   ) -- ( 7.4 , 1.3 );
	\draw ( 6.8 , 1.4 ) node[anchor=north west] {\scriptsize $u_2^1$};
	\draw ( 6.8 , 0.71 ) node[anchor=north west] {\scriptsize $Case~1$};

	\draw[-] ( 5.6 , -1.5 ) .. controls ( 6.3 , -1.1 ) and ( 6.9 , -1.1 ) .. ( 7.4 , -1.2 ) node[above,sloped,pos=0.5] {};
	\draw[-] ( 5.6 , -1.5 ) .. controls ( 6.2 , -2.0 ) and ( 6.8 , -2.0 ) .. ( 7.4 , -1.2 ) node[above,sloped,pos=0.5] {};
	\draw ( 6.1 , -1.15 ) node[anchor=north west] {$S_{v_1}$};
	\draw[-] ( 7.4 , -1.2 ) .. controls ( 8.1 , -1.0 ) and ( 8.8 , -1.0 ) .. ( 9.9 , -1.5 ) node[above,sloped,pos=0.5] {};
	\draw[-] ( 7.4 , -1.2 ) .. controls ( 8.1 , -2.0 ) and ( 8.8 , -2.0 ) .. ( 9.9 , -1.5 ) node[above,sloped,pos=0.5] {};
	\draw ( 8.63 , -1.33 ) node[anchor=north west] {\scriptsize $S_{v_2}$};
	\draw ( 5.6 , -1.5 ) -- ( 5.6 , -2.3 );
	\draw ( 7.4 , -1.2 ) -- ( 7.4 , -2.3 );
	\draw ( 9.9 , -1.5 ) -- ( 9.9 , -2.3 );
	
	\draw[->] ( 5.63 , -2.2 ) -- ( 7.4 , -2.2 );
	\draw ( 6.2 , -1.83 ) node[anchor=north west] {\scriptsize $v_1$};
	\draw[->] ( 7.44 , -2.2 ) -- ( 9.87 , -2.2 );
	\draw ( 8.2 , -1.83 ) node[anchor=north west] {\scriptsize $v_2$};
	\put( 7.4 , -0.1 ){\circle*{.08}}
	\put( 7.4 , -1.2 ){\circle*{.08}}
	\put( 5.6 , -1.5 ){\circle*{.08}}
	\put( 9.9 , -1.5 ){\circle*{.08}}
	\draw[-] ( 5.6 , -1.5 ) -- ( 7.4 , -0.1 );
	\draw[->] ( 6.4 , -0.88 ) -- ( 6.5 , -0.8 );
	\draw ( 6.1 , -0.4 ) node[anchor=north west] {\scriptsize $u_1$};
	\draw[-] ( 7.4 , -0.1 ) -- ( 9.9 , -1.5 );
	\draw[->] ( 7.4 , -0.1 ) -- ( 8.65 , -0.8 );
	\draw ( 8.4 , -0.4 ) node[anchor=north west] {\scriptsize $u_2$};
	
	\draw[->] ( 7.4 , -0.1  ) -- ( 8.7 , -1.49 );
	\put(  8.7 , -1.5 ){\circle*{.08}}
	\draw[-] ( 8.7 , -1.5 ) .. controls (9.1 , -1.4 ) and (9.4 , -1.3 ) .. (9.9 , -1.5) node[above,sloped,pos=0.5] {};
	\draw[- ] (7.4 , -1.2    ) -- ( 8.7 , -1.5 );
	\draw[->] (7.4 , -1.2  ) -- ( 8.05 , -1.35 );
	\draw ( 7.99 , -1.32 ) node[anchor=north west] {\scriptsize $u_1^2$};
	\draw ( 6.8 , -2.3 ) node[anchor=north west] {\scriptsize $Case~2$};
	\end{tikzpicture}
	\caption{Interpretation of the convexity property of a pair $(v_1,v_2)$.}\label{fig:interpretation}
\end{center}
\end{figure}
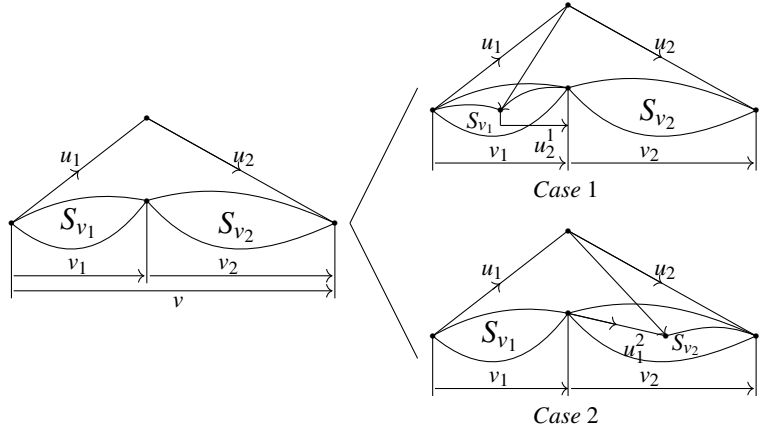

With the notion of cut vertex and block borrowed from graph theory, we prove:

\begin{theorem}\label{thm:blockdecomposition}
Let $F$ be a  final segment of $\Lambda^{\ast}$ distinct from $\Lambda^{\ast}$. Then $F$ is irreducible if and only if  $\mathcal S_F$ is irreducible if and only if $\mathbf {G}_F$ has no cut vertex. If $F$ is not irreducible,
the blocks of $\mathbf {G}_F$ are  the vertices of a finite path $C_0, \dots, C_{n-1}$ with $n\geq 2$,  whose  end vertices  $C_0$ and $C_{n-1}$ contain respectively  the initial state $x$ and the final state $y$ of the automaton $\mathcal{A}_{F}$ accepting $F$. Furthermore, $F$ is the concatenation $F_0 \dots  F_i \dots  F_{n-1}$, the automaton  $\mathcal{A}_{F_i}$ accepting $F_i$  being isomorphic to  $({\mathcal M}_{F}\!\restriction C_{i},\left\{ x_i\right\}, \left\{
x_{i+1}\right\} )$, where $x_0:=x$, $x_{n}:= y$ and  $\{x_{i+1}\}=C_{i} \cap C_{i+1}$ for $0\leq i< n-1$.
\end{theorem}
From this result, the freeness of $\F^{\circ}(\Lambda^{\ast})$ follows.

This result does not yield a concrete test for irreducibility, the size of the injective envelope $\mathcal S_F$ in terms of the length of  words generating $F$ can be  a double exponential (see Subsection 4.5 of \cite{KPR}). But it suggests a similar result for the minimal automaton recognizing $F$. In \cite{KPR}, we prove
\begin{theorem} \label{thm:minimal}
If  $\mathcal A$ is the minimal deterministic automaton  recognizing a final segment $F\in \F^{\circ}(\Lambda^{\ast})$  then $F$ is irreducible iff there is no vertex $z$ distinct from the initial  state $x$ and the final state $y$ which lies on all directed paths going from $x$ to $y$.
\end{theorem}

\section{Further developments}\label{section:further development}
There are  several interesting examples of generalized metric spaces for which the set of values is not a Heyting algebra.

This is the case for metric spaces over a Boolean algebra (except if the Boolean algebra is the power set of a set).  If $B$ is a Boolean algebra, not necessarily complete, or not satisfying the distributivity condition, residuation holds (i.e., for every $x, y\in B$,  $y\setminus x$ is the least element $r$ of $B$ such that $x\leq y\vee r$); hence, one may define a distance $d$ over $B$: the distance $d(p, q)$ between two elements $p,q$  of $B$ is the symmetric difference  $p\Delta q:= p\setminus q \cup q\setminus p$. If 
$B$ is complete, Theorem  \ref{thm:embedding} holds.  

An other example comes from arithmetic. The Chinese remainder theorem  can be viewed as a property of balls in a metric space. Indeed, if $a_i, r_i$ $(i\in I)$ is a family of pairs of  integers
we may view each congruence class of $a_i$ modulo $r_i$ in $\Z$ as a (closed) ball $B(a_i, r_i):= \{x\in \Z: d(a_i, x)\preceq  r_i\}$ for a suitable distance $d$ on $\Z$ and order $\preceq$ on the set of values of the distance. The Chinese remainder theorem expresses when these balls have a non-empty intersection.  As we have seen, Helly property and convexity are  the keywords to ensure  a non-empty intersection of balls. In our case, $\Z$ has a structure of ultrametric space with values in $\N$ provided that  we order $\N$ by the reverse of divisibility  setting $n\preceq m$ if $n$ is a multiple of $m$. Doing so, $(\N, \preceq)$ is  a  distributive  complete lattice, the least element being $0$, the largest $1$, the  join $n\vee m$ of $n$ and $m$ being the largest common divisor.   Replacing  the addition by the join and for two elements $a,b\in \Z$, setting $d(a,b):= \vert a- b\vert$, we have   $d(a,b)=0$ iff $a=b$; $d(a,b)=d(b,a)$ and $d(a,b)\preceq  d(a,c)\vee d(c,b)$ for all $a,b, c\in \Z$. With this definition, closed balls in $\Z$ are congruence classes of the additive group ($\Z, +)$.  In  an ordinary metric space $\mathscr V:= (V, \preceq)$, the  necessary condition for the non-emptiness of the intersection of two balls $B(a_i, r_i)$ and $B(a_j, r_j)$ is the \emph{convexity property}:   the distance between centers is at most the sum of the radii. Here this yields  $d(a_i,a_j)\preceq r_i\vee r_j$, i.e.   $a_i$ and $a_j$ are congruent modulo $lcd(r_i, r_j)$. The Chinese remainder theorem  expresses that the  intersection of \emph{finitely many} balls  $B(a_i, r_i)$ is non-empty iff this  family of balls $B(a_i, r_i)$ satisfies the convexity property  and the finite $2$-Helly property.  This property does not extend to infinite families:  the space $\Z$ is not hyperconvex (and $\N$ equipped with the join as a monoid operation is not a Heyting algebra); we may say that  it is \emph{finitely hyperconvex}. 

Metric spaces over $(\N, \preceq)$,  like $\Z$, are example of metric spaces over a join-semilattice $\mathscr V:= (V, \preceq)$ with a $0$. They fit in the category of ultrametric spaces.  If $\mathbf E:=(E, d)$ is such a metric space,  set $\equiv_r:= \{(x,y)\in E: d(x,y)\preceq r\}$ for each $r$;  this relation is an   equivalence relation  on $E$. Let $\Eqv(E)$ be  the set of equivalence relations on $E$ and $\Eqv_{d}(E):= \{\equiv_r: r\in V\}$. Then, as it is easy to see, any two members of $\Eqv_d(E)$ commute and $\equiv_r\circ \equiv_s=\equiv_s\circ \equiv_r= \equiv_{r\vee s}$ for every $r,s\in  V$  iff $(E,d)$ is convex. If the meet of every non-empty subset of $V$ exists, then $\Eqv_d(E)$ is an intersection closed subset of $\Eqv(E)$. Furthermore, $(E, d)$ is hyperconvex iff $\Eqv_d(E)$ is a completely meet-distributive lattice of $Eq(E)$ (Proposition 3.12 of \cite{pouzet}). A sublattice $L$ of the lattice $\Eqv(E)$ of equivalence  relations is \emph{arithmetical} (see \cite {kaarli-pixley})  if it is distributive and pairs of members of $L$ commute with respect to composition. As it is well known (see \cite{kaarli-pixley}), arithmetic lattices can be characterized in terms of the Chinese remainder conditions (expressed as in the theorem mentionned above). This property amounting to finite hyperconvexity, it yields  the one-extension property for maps with finite domain and, if $E$ is countable, the fact that every partial nonexpansive map from a finite subset of $E$ extends to $E$  \cite{kaarli}. The study of maps preserving congruences, nonexpansive maps in our setting, is a very basic subject of universal algebra (for a beautiful recent result, see \cite{cgg2}). Some results about metric spaces over meet-distributive lattices and their nonexpansive maps were obtained in \cite{PR,  pouzet2}. The relation with universal algebra (and arithmetic) suggests to look at  possible extensions.

\section*{acknowledgement}This paper is a preliminary version of a book chapter of "New Trends in Analysis and Geometry". We thank  Mohamed Amine Khamsi for his encouragements.

\end{document}